\documentclass{amsproc}
\usepackage{amssymb,amscd}
\usepackage[cp1251]{inputenc}
\usepackage[T2A]{fontenc}
\usepackage[matrix,arrow]{xy}

\newtheorem{theorem}{Theorem}[section]
\newtheorem{proposition}[theorem]{Proposition}
\newtheorem{corollary}[theorem]{Corollary}
\newtheorem{lemma}[theorem]{Lemma}
\newtheorem{question}[theorem]{Question}
\theoremstyle{definition}
\newtheorem{definition}[theorem]{Definition}
\newtheorem{example}[theorem]{Example}
\newtheorem{construction}[theorem]{Construction}
\theoremstyle{remark}
\newtheorem*{remark}{Remark}

\numberwithin{equation}{section}
\numberwithin{figure}{section}

\def\C{\mathbb C}
\def\D{\mathbb D}
\def\I{\mathbb I}

\def\R{\mathbb R}
\def\T{\mathbb T}
\def\Z{\mathbb Z}

\def\sK{\mathcal K}

\def\k{\mathbf k}

\def\phi{\varphi}

\newcommand{\mb}[1]{{\textbf {\textit#1}}}

\renewcommand{\ge}{\geqslant}
\renewcommand{\le}{\leqslant}

\def\dbs{/\!\!/}
\newcommand{\id}{\mathrm{id}}

\def\Ann{\mathop{\mathrm{Ann}}}

\renewcommand{\Im}{\mathop{\mathrm{Im}}\nolimits}
\newcommand{\Ker}{\mathop{\rm Ker}}
\newcommand{\rank}{\mathop{\mathrm{rank}}}
\renewcommand{\Re}{\mathop{\mathrm{Re}}}

\def\Spec{\mathop{\mathrm{Spec}}}

\newcommand{\cc}{\mathop{\rm cc}}

\newcommand{\lk}{\mathop{\rm lk}\nolimits}
\newcommand{\st}{\mathop{\rm St}\nolimits}

\def\conv{\mathop{\mathrm{conv}}}

\newcommand{\zk}{\mathcal Z_{\mathcal K}}

\newcommand{\zp}{\mathcal Z_P}

\begin{document}
\title{Geometric structures on moment-angle manifolds}

\author{Taras Panov}
\address{Department of Mathematics and Mechanics, Moscow
State University, Leninskie Gory, 119991 Moscow, Russia,
\newline\indent Institute for Theoretical and Experimental Physics,
Moscow, Russia,\quad \emph{and}
\newline\indent Institute for Information Transmission Problems,
Russian Academy of Sciences} \email{tpanov@mech.math.msu.su}

\thanks{The author was supported by the Russian Foundation for
Basic Research, grants~12-01-00873 and 13-01-91151-ГФЕН, a grant
from Dmitri Zimin's `Dynasty' foundation, grants НШ-4995-2012.1
and МД-111.2013.1 from the President of Russia, and grant
11.G34.31.0053 from the Government of Russia.}


\begin{abstract}
The moment-angle complex $\zk$ is cell complex with a torus action
constructed from a finite simplicial complex~$\sK$. When this
construction is applied to a triangulated sphere~$\sK$ or, in
particular, to the boundary of a simplicial polytope, the result
is a manifold. Moment-angle manifolds and complexes are central
objects in toric topology, and currently are gaining much interest
in homotopy theory, complex and symplectic geometry.

The geometric aspects of the theory of moment-angle complexes are
the main theme of this survey. We review constructions of
non-K\"ahler complex-analytic structures on moment-angle manifolds
corresponding to polytopes and complete simplicial fans, and
describe invariants of these structures, such as the Hodge numbers
and Dolbeault cohomology rings. Symplectic and Lagrangian aspects
of the theory are also of considerable interest. Moment-angle
manifolds appear as level sets for quadratic Hamiltonians of torus
actions, and can be used to construct new families of
Hamiltonian-minimal Lagrangian submanifolds in a complex space,
complex projective space or toric varieties.
\end{abstract}

\maketitle

\tableofcontents

\section{Introduction}
Moment-angle complex $\zk$ is a cell complex with a torus action
patched from products of discs $D^2$ and circles~$S^1$ which are
parametrised by faces of a simplicial complex~$\sK$. By replacing
the pair $(D^2,S^1)$ by an arbitrary cellular pair $(X,A)$ we
obtain the \emph{polyhedral product}~$(X,A)^\sK$. Moment-angle
complexes and polyhedral products are key players in the emerging
field of \emph{toric topology}, which lies on the borders between
topology, algebraic and symplectic geometry, and
combinatorics~\cite{bu-pa12}.

Both homotopical and geometric aspects of the theory of
moment-angle complexes and polyhedral products have been actively
studied recently. On the homotopy-theoretical side of the story,
the stable and unstable decomposition techniques developed
in~\cite[Ch.~6]{bu-pa02}, \cite{gr-th07}, \cite{b-b-c-g10},
\cite{ir-ki} have led to an improved understanding of the topology
of moment-angle complexes and related toric spaces.

In this survey we concentrate on the geometric aspects of the
theory. The construction of moment-angle complexes has many
interesting geometric interpretations. For example, the
moment-angle complex $\zk$ is homotopy equivalent to the
complement $U(\sK)$ of the arrangement of coordinate subspaces
in~$\C^m$ defined by~$\sK$. The space $U(\sK)$ plays an important
role in geometry of toric varieties and the theory of
configuration spaces.

The moment-angle complex $\zk$ corresponding to a triangulated
sphere~$\sK$ is a topological manifold. Moment-angle manifolds
corresponding to simplicial polytopes or, more generally, complete
simplicial fans, are smooth. In the polytopal case a smooth
structure arises from the realisation of $\zk$ by a nondegenerate
intersection of Hermitian quadrics in $\C^m$, similar to a level
set of the moment map in the construction of symplectic quotients.
The relationship between polytopes and systems of quadrics is
described by the convex-geometric notion of Gale duality.

Another way to give $\zk$ a smooth structure is to realise it as
the quotient of the coordinate subspace arrangement complement
$U(\sK)$ by an action of the multiplicative group~$\R^{m-n}_>$.
This is similar to the well-known quotient construction of toric
varieties in algebraic geometry. The quotient of a non-compact
manifold $U(\sK)$ by the action of a non-compact group
$\R^{m-n}_>$ is Hausdorff precisely when $\sK$ is the underlying
complex of a simplicial fan.

If $m-n=2\ell$, then the action of the real group $\R^{m-n}_>$ on
$U(\sK)$ can be turned into a holomorphic action of a complex (but
not algebraic) group isomorphic to~$\C^{\ell}$. In this way the
moment-angle manifold $\zk\cong U(\sK)/\C^\ell$ acquires a
complex-analytic structure. The resulting family of non-K\"ahler
complex manifolds generalises the well-known series of Hopf and
Calabi--Eckmann manifolds (see~\cite{bo-me06} and~\cite{pa-us12}).

Finally, the intersections of Hermitian quadrics defining
polytopal moment-angle manifolds were also used in~\cite{miro04}
to construct Lagrangian submanifolds in $\C^m$ with special
minimality properties.

\medskip

Different spaces with torus actions, or \emph{toric spaces}, will
feature throughout the paper. The most basic example of a toric
space is the complex $m$-dimensional space~$\C^m$, on which the
\emph{standard torus}
\[
  \T^m=\bigl\{\mb t=(t_1,\ldots,t_m)\in\C^m\colon|t_i|=1
  \quad\text{for }i=1,\ldots,m\bigr\}
\]
acts coordinatewise. That is, the action is given by
\begin{align*}
  &\T^m\times\C^m\longrightarrow\C^m,\\
  &(t_1,\ldots,t_m)\cdot(z_1,\ldots,z_m)=(t_1z_1,\ldots,t_mz_m).
\end{align*}
The quotient $\C^m/\T^m$ of this action is the \emph{positive
orthant}
\[
  \R^m_\ge=\bigl\{(y_1,\ldots,y_m)\in\R^m\colon y_i\ge0
  \quad\text{for }
  i=1,\ldots,m\bigr\}
\]
with the quotient projection given by
\begin{align*}
  \mu\colon
  \C^m&\longrightarrow\R^m_\ge,\\
  (z_1,\ldots,z_m)&\longmapsto(|z_1|^2,\ldots,|z_m|^2).
\end{align*}

We use the blackboard bold capitals in the notation $\I^m$,
$\T^m$, $\D^m$ of the standard unit cube in $\R^m$, the standard
(unit) torus, and the unit polydisc in $\C^m$ respectively. We use
italic $T^m$ to denote an abstract $m$-torus, i.e. a compact
abelian Lie group isomorphic to a product of $m$ circles. The
underlying space of the unit disc $\D$ is a topological 2-disc,
which we denote by $D^2$. We shall also denote the standard unit
circle by $\mathbb S$ or $\mathbb T$ occasionally, to distinguish
it from an abstract circle~$S^1$.

\section{Preliminaries: polytopes and Gale
duality}\label{galediag} Let $\R^n$ be a Euclidean space with
scalar product $\langle\;,\:\rangle$. A convex \emph{polyhedron}
$P$ is an intersection of finitely many halfspaces in~$\R^n$.
Bounded polyhedra are called \emph{polytopes}. Alternatively, a
polytope can be defined as the convex hull $\conv(\mb
v_1,\ldots,\mb v_q)$ of a finite set of points $\mb v_1,\ldots,\mb
v_q\in\R^n$.

A \emph{supporting hyperplane} of $P$ is a hyperplane $H$ which
has common points with $P$ and for which the polyhedron is
contained in one of the two closed half-spaces determined by~$H$.
The intersection $P\cap H$ with a supporting hyperplane is called
a \emph{face} of the polyhedron. Denote by $\partial P$ and
$\mathop{\mathrm{int}} P=P\setminus\partial P$ the topological
boundary and interior of $P$ respectively. In the case $\dim P=n$
the boundary $\partial P$ is the union of all faces of~$P$.
Zero-dimensional faces are called \emph{vertices}, one-dimensional
faces are \emph{edges}, and faces of codimension one are
\emph{facets}.

Two polytopes are \emph{combinatorially equivalent} if there is a
bijection between their faces preserving the inclusion relation. A
\emph{combinatorial polytope} is a class of combinatorially
equivalent polytopes. Two polytopes are combinatorially equivalent
if there is a homeomorphism between them preserving the face
structure.

The faces of a given polytope $P$ form a partially ordered set (a
\emph{poset}) with respect to inclusion. It is called the
\emph{face poset} of~$P$. Two polytopes are combinatorially
equivalent if and only if their face posets are isomorphic.

\medskip

Consider a system of $m$ linear inequalities defining a convex
polyhedron in~$\R^n$:
\begin{equation}\label{ptope}
  P=\bigl\{\mb x\in\R^n\colon\langle\mb a_i,\mb x\rangle+b_i\ge0\quad\text{for }
  i=1,\ldots,m\bigr\},
\end{equation}
where $\mb a_i\in\R^n$ and $b_i\in\R$. We refer to~\eqref{ptope}
as a \emph{presentation} of the polyhedron~$P$ by inequalities.
These inequalities contain more information than the
polyhedron~$P$, for the following reason. It may happen that some
of the inequalities $\langle\mb a_i,\mb x\rangle+b_i\ge0$ can be
removed from the presentation without changing~$P$; we refer to
such inequalities as \emph{redundant}. A presentation without
redundant inequalities is called \emph{irredundant}. An
irredundant presentation of a given polyhedron is unique up to
multiplication of the pairs $(\mb a_i,b_i)$ by positive numbers.

We shall assume (unless stated otherwise) that the polyhedron $P$
defined by~\eqref{ptope} has a vertex, which is equivalent to that
the vectors $\mb a_1,\ldots,\mb a_m$ span the whole~$\R^n$. This
condition is automatically satisfied for polytopes.

A presentation~\eqref{ptope} is said to be \emph{generic} if $P$
is nonempty and the hyperplanes defined by the equations
$\langle\mb a_i,\mb x\rangle+b_i=0$ are in general position at any
point of~$P$ (that is, for any $\mb x\in P$ the normal vectors
$\mb a_i$ of the hyperplanes containing $\mb x$ are linearly
independent). If presentation~\eqref{ptope} is generic, then $P$
is $n$-dimensional. If $P$ is a polytope, then the existence of a
generic presentation implies that $P$ is \emph{simple}, that is,
exactly $n$ facets meet at each vertex of~$P$. A generic
presentation may contain redundant inequalities, but, for any such
inequality, the intersection of the corresponding hyperplane with
$P$ is empty (i.e., the inequality is strict at any $\mb x\in P$).
We set
\[
  F_i=\bigl\{\mb x\in P\colon\langle\mb a_i,\mb x\rangle+b_i=0\bigr\}.
\]
If presentation~\eqref{ptope} is generic, then each $F_i$ is
either a facet of~$P$ or is empty.

The \emph{polar set} of a polyhedron $P\subset\R^n$ is defined as
\begin{equation}\label{polarset}
  P^*=\bigl\{\mb u\in\R^n\colon\langle\mb u,\mb x\,\rangle+1\ge0
  \:\text{ for all }\:\mb x\in P\bigr\}.
\end{equation}
The set $P^*$ is a convex polyhedron. (In fact, it is naturally a
subset in the dual space $(\R^n)^*$, but we shall not make this
distinction, assuming $\R^n$ to be Euclidean.)

The following properties are well known in convex geometry:

{\samepage
\begin{theorem}[{see~\cite[\S2.9]{bron83} or~\cite[Theorem~2.11]{zieg95}}]\label{polarity}\
\begin{itemize}
\item[(a)] $P^*$ is bounded if and only if $\mathbf 0\in\mathop{\mathrm{int}} P$;

\item[(b)] $P\subset (P^*)^*$, and $(P^*)^*=P$ if and only if $\mathbf 0\in
P$;

\item[(c)] if a polytope $Q$ is given as a convex hull,
$Q=\mathop{\mathrm{conv}(\mb a_1,\ldots,\mb a_m)}$, then $Q^*$ is
given by inequalities~\eqref{ptope} with $b_i=1$ for $1\le i\le
m$; in particular, $Q^*$ is a convex polyhedron, but not
necessarily bounded;

\item[(d)] if a polytope $P$ is given by inequalities~\eqref{ptope}
with $b_i=1$, then $P^*=\mathop{\mathrm{conv}(\mb a_1,\ldots,\mb
a_m)}$; furthermore, $\langle\mb a_i,\mb x\,\rangle+1\ge0$ is a
redundant inequality if and only if $\mb a_i\in\conv(\mb a_j\colon
j\ne i)$.
\end{itemize}
\end{theorem}
}

\begin{remark}
A polyhedron $P$ admits a presentation~\eqref{ptope} with $b_i=1$
if and only if $\mathbf 0\in\mathop{\mathrm{int}} P$. In general,
$(P^*)^*=\conv(P,{\bf0})$.
\end{remark}

Any combinatorial polytope $P$ has a presentation~\eqref{ptope}
with $b_i=1$ (take the origin to the interior of~$P$ by a parallel
transform, an then divide each of the inequalities
in~\eqref{ptope} by the corresponding~$b_i$). Then $P^*$ is also a
polytope with ${\bf 0}\in P^*$, and $(P^*)^*=P$. We refer to the
combinatorial polytope $P^*$ as the \emph{dual} of the
combinatorial polytope~$P$. (We shall not introduce a new notation
for the dual polytope, keeping in mind that polarity is a
convex-geometric notion, while duality of polytopes is
combinatorial.)

\begin{theorem}[{see~\cite[\S2.10]{bron83}}]\label{dualpol}
If $P$ and $P^*$ are dual polytopes, then the face poset of $P^*$
is obtained from the face poset of $P$ by reversing the inclusion
relation.
\end{theorem}

If $P$ is a simple polytope, then it follows from the theorem
above that each face of $P^*$ is a simplex. Such a polytope is
called \emph{simplicial}.

The following construction realises any polytope~\eqref{ptope} of
dimension $n$ by the intersection of the orthant $\R^m_\ge$ with
an affine $n$-plane. It will be used in the next section to define
intersections of quadrics and moment-angle manifolds.

\begin{construction}\label{dist}
Form the $n\times m$-matrix $A$ whose columns are the vectors $\mb
a_i$ written in the standard basis of~$\R^n$. Note that $A$ is of
rank~$n$ if and only if the polyhedron $P$ has a vertex. Also, let
$\mb b=(b_1,\ldots,b_m)^t\in\R^m$ be the column vector of~$b_i$s.
Then we can write~\eqref{ptope} as
\[
  P=P(A,\mb b)=\bigl\{\mb x\in\R^n\colon(A^t\mb x+\mb b)_i\ge 0\quad
  \text{for }i=1,\ldots,m\},
\]
where $\mb x=(x_1,\ldots,x_n)^t$ is the column of coordinates.
Consider the affine map
\[
  i_{A,\mb b}\colon \R^n\to\R^m,\quad i_{A,\mb b}(\mb x)=A^t\mb x+\mb b
  =\bigl(\langle\mb a_1,\mb x\rangle+b_1,\ldots,\langle\mb a_m,\mb x\rangle+b_m\bigr)^t.
\]
If $P$ has a vertex, then the image of $\R^n$ under $i_{A,\mb b}$
is an $n$-dimensional affine plane in $\R^m$, which we can write
by $m-n$ linear equations:
\begin{equation}\label{iabrn}
\begin{aligned}
  i_{A,\mb b}(\R^n)&=\{\mb y\in\R^m\colon\mb y=A^t\mb x+\mb b\quad
  \text{for some }\mb x\in\R^n\}\\
  &=\{\mb y\in\R^m\colon\varGamma\mb y=\varGamma\mb b\},
\end{aligned}
\end{equation}
where $\varGamma=(\gamma_{jk})$ is an $(m-n)\times m$-matrix whose
rows form a basis of linear relations between the vectors~${\mb
a}_i$. That is, $\varGamma$ is of full rank and satisfies the
identity $\varGamma A^t=0$.

We have $i_{A,\mb b}(P)=\R^m_\ge\cap i_{A,\mb b}(\R^n)$.
\end{construction}

\begin{construction}[Gale duality]\label{galeduality}
Let $\mb a_1,\ldots,\mb a_m$ be a configuration of vectors that
span the whole~$\R^n$. Form an $(m-n)\times m$-matrix
$\varGamma=(\gamma_{jk})$ whose rows form a basis in the space of
linear relations between the vectors~${\mb a}_i$. The set of
columns $\gamma_1,\ldots,\gamma_m$ of $\varGamma$ is called a
\emph{Gale dual} configuration of $\mb a_1,\ldots,\mb a_m$. The
transition from the configuration of vectors $\mb a_1,\ldots,\mb
a_m$ in $\R^n$ to the configuration of vectors
$\gamma_1,\ldots,\gamma_m$ in $\R^{m-n}$ is called the (linear)
\emph{Gale transform}. Each configuration determines the other
uniquely up to isomorphism of its ambient space. In other words,
each of the matrices $A$ and $\varGamma$ determines the other
uniquely up to multiplication by an invertible matrix from the
left.

Using the coordinate-free notation, we may think of $\mb
a_1,\ldots,\mb a_m$ as a set of linear functions on an
$n$-dimensional space~$W$. Then we have an exact sequence
\[
  0\to W\stackrel
  {A^t}\longrightarrow\R^m\stackrel{\varGamma}\longrightarrow L\to0,
\]
where $A^t$ is given by $\mb x\mapsto \bigl(\langle\mb a_1,\mb
x\rangle,\ldots,\langle\mb a_m,\mb x\rangle\bigr)$, and the map
$\varGamma$ takes $\mb e_i$ to $\gamma_i\in L\cong\R^{m-n}$.
Similarly, in the dual exact sequence
\[
  0\to L^*\stackrel
  {\varGamma^t}\longrightarrow\R^m\stackrel{A}\longrightarrow W^*\to0,
\]
the map $A$ takes $\mb e_i$ to $\mb a_i\in W^*\cong\R^n$.
Therefore, two configurations $\mb a_1,\ldots,\mb a_m$ and
$\gamma_1,\ldots,\gamma_m$ are Gale dual if they are obtained as
the images of the standard basis of $\R^m$ under the maps $A$ and
$\varGamma$ in a pair of dual short exact sequences.
\end{construction}

Here is an important property of Gale dual configurations:

\begin{theorem}\label{galespan}
Let $\mb a_1,\ldots,\mb a_m$ and $\gamma_1,\ldots,\gamma_m$ be
Gale dual configurations of vectors in $\R^n$ and $\R^{m-n}$
respectively, and let $I=\{i_1,\ldots,i_k\}$. Then the subset
$\{\mb a_i\colon i\in I\}$ is linearly independent if and only if
the subset $\{\gamma_j\colon j\notin I\}$ spans the
whole~$\R^{m-n}$.
\end{theorem}

The proof uses an algebraic lemma:

\begin{lemma}\label{2ses}
Let $\k$ be a field or $\Z$, and assume given a diagram
\[
\begin{CD}
  @.@.\begin{array}{c}0\\ \raisebox{2pt}{$\downarrow$}\\
  U\end{array}@.@.\\
  @.@.@VV{i_1}V@.@.\\
  0@>>> R@>{i_2}>> S@>{p_2}>> T @>>>0\\
  @.@.@VV {p_1} V@.@.\\
  @.@.\begin{array}{c} V\\ \downarrow\\ 0\end{array}@.@.\\
\end{CD}
\]
in which both vertical and horizontal lines are short exact
sequences of vector spaces over $\k$ or free abelian groups. Then
$p_1i_2$ is surjective (respectively, injective or split
injective) if and only if $p_2i_1$ is surjective (respectively,
injective or split injective).
\end{lemma}
\begin{proof}
This is a simple diagram chase. Assume first that $p_1i_2$ is
surjective. Take $\alpha\in T$; we need to cover it by an element
in~$ U$. There is $\beta\in S$ such that $p_2(\beta)=\alpha$. If
$\beta\in i_1( U)$, then we are done. Otherwise set
$\gamma=p_1(\beta)\ne0$. Since $p_1i_2$ is surjective, we can
choose $\delta\in R$ such that $p_1i_2(\delta)=\gamma$. Set
$\eta=i_2(\delta)\ne0$. Hence, $p_1(\eta)=p_1(\beta)(=\gamma)$ and
there is $\xi\in U$ such that $i_1(\xi)=\beta-\eta$. Then
$p_2i_1(\xi)=p_2(\beta-\eta)=\alpha-p_2i_2(\delta)=\alpha$. Thus,
$p_2i_1$ is surjective.

Now assume that $p_1i_2$ is injective. Suppose $p_2i_1(\alpha)=0$
for a nonzero $\alpha\in U$. Set $\beta=i_1(\alpha)\ne0$. Since
$p_2(\beta)=0$, there is a nonzero $\gamma\in R$ such that
$i_2(\gamma)=\beta$. Then
$p_1i_2(\gamma)=p_1(\beta)=p_1i_1(\alpha)=0$. This contradicts the
assumption that $p_1i_2$ is injective. Thus, $p_2i_1$ is
injective.

Finally, if $p_1i_2$ is split injective, then its dual map
$i_2^*p_1^*\colon V^*\to R^*$ is surjective. Then
$i_1^*p_2^*\colon T^*\to U^*$ is also surjective. Thus, $p_2i_1$
is split injective.
\end{proof}

\begin{proof}[Proof of Theorem~\ref{galespan}]
Let $A$ be the $n\times m$-matrix with column vectors $\mb
a_1,\ldots,\mb a_m$, and let $\varGamma$ be the $(m-n)\times
m$-matrix with columns $\gamma_1,\ldots,\gamma_m$. Denote by $A_I$
the $n\times k$-submatrix formed by the columns $\{\mb a_i\colon
i\in I\}$ and denote by $\varGamma_{\widehat I}$ the
$(m-n)\times(m-k)$-submatrix formed by the columns
$\{\gamma_j\colon j\notin I\}$. We also consider the corresponding
maps $A_I\colon\R^k\to\R^n$ and $\varGamma_{\widehat
I}\colon\R^{m-k}\to\R^{m-n}$.

Let $i\colon\R^k\rightarrow\R^m$ be the inclusion of the
coordinate subspace spanned by the vectors $\mb e_i$, $i\in I$,
and let $p\colon\R^m\rightarrow\R^{m-k}$ the projection sending
every such $\mb e_i$ to zero. Then $A_I=A\cdot i$ and
$\varGamma_{\widehat I}^t=p\cdot\varGamma^t$. The vectors $\{\mb
a_i\colon i\in I\}$ are linearly independent if and only if
$A_I=A\cdot i$ is injective, and the vectors $\{\gamma_j\colon
j\notin I\}$ span~$\R^{m-n}$ if and only if $\varGamma_{\widehat
I}^t=p\cdot\varGamma^t$ is injective. These two conditions are
equivalent by Lemma~\ref{2ses}.
\end{proof}

\begin{construction}[Gale diagram]
Let $P$ be a polytope~\eqref{ptope} with $b_i=1$ and let
$P^*=\conv(\mb a_1,\ldots,\mb a_m)$ be the polar polytope. Let
$\widetilde A^t=(A^t\;\mathbf 1)$ be the $m\times(n+1)$-matrix
obtained by appending a column of units to~$A^t$. The matrix
$\widetilde A^t$ has full rank $n+1$ (indeed, otherwise there is
$\mb x\in\R^n$ such that $\langle\mb a_i,\mb x\rangle=1$ for
all~$i$, and then $\lambda\mb x$ is in~$P$ for any $\lambda\ge1$,
so that $P$ is unbounded). A configuration of vectors $G=(\mb
g_1,\ldots,\mb g_m)$ in $\R^{m-n-1}$ which is Gale dual to
$\widetilde A$ is called a \emph{Gale diagram} of~$P^*$. A Gale
diagram $G=(\mb g_1,\ldots,\mb g_m)$ of $P^*$ is therefore
determined by the conditions
\[
  GA^t=0,\quad \rank G=m-n-1,\quad\text{and }
  \sum_{i=1}^m\mb g_i=\mathbf 0.
\]
The rows of the matrix $G$ from a basis of \emph{affine
dependencies} between the vectors $\mb a_1,\ldots,\mb a_m$, i.e. a
basis in the space of $\mb y=(y_1,\ldots,y_m)^t$ satisfying
\[
  y_1\mb a_1+\cdots+y_m\mb a_m=\mathbf0,\quad
  y_1+\cdots+y_m=0.
\]
\end{construction}

\begin{proposition}\label{propcf}
The polyhedron $P=P(A,\mb b)$ is bounded if and only if the matrix
$\varGamma=(\gamma_{jk})$ can be chosen so that the affine plane
$i_{A,\mb b}(\R^n)$ is given by
\begin{equation}\label{bplane}
  i_{A,\mb b}(\R^n)= \left\{\begin{array}{ll}
  \mb y\in\R^m\colon&\gamma_{11}y_1+\cdots+\gamma_{1m}y_m=c,\\[1mm]
  &\gamma_{j1}y_1+\cdots+\gamma_{jm}y_m=0,\quad
  2\le j\le m-n,
  \end{array}\right\}
\end{equation}
where $c>0$ and $\gamma_{1k}>0$ for all $k$.

Furthermore, if $b_i=1$ in~\eqref{ptope}, then the vectors $\mb
g_i=(\gamma_{2i},\ldots,\gamma_{m-n,i})^t$, $i=1,\ldots,m$, form a
Gale diagram of the polar polytope $P^*=\conv(\mb a_1,\ldots,\mb
a_m)$.
\end{proposition}
\begin{proof}
The image $i_{A,\mb b}(P)$ is the intersection of the $n$-plane
$L=i_{A,\mb b}(\R^n)$ with $\R^m_\ge$. It is bounded if and only
if $L_0\cap\R_\ge^m=\{\mathbf0\}$, where $L_0$ is the $n$-plane
through~$\mathbf 0$ parallel to~$L$. Choose a hyperplane $H_0$
through $\mathbf 0$ such that $L_0\subset H_0$ and
$H_0\cap\R_\ge^m=\{\mathbf 0\}$. Let $H$ be the affine hyperplane
parallel to $H_0$ and containing~$L$. Since $L\subset H$, we may
take the equation defining $H$ as the first equation in the system
$\varGamma\mb y=\varGamma\mb b$ defining~$L$. The conditions on
$H_0$ imply that $H\cap\R_\ge^m$ is nonempty and bounded, that is,
$c>0$ and $\gamma_{1k}>0$ for all~$k$. Now, subtracting the first
equation from the other equations of the system $\varGamma\mb
y=\varGamma\mb b$ with appropriate coefficients, we achieve that
the right hand sides of the last $m-n-1$ equations become zero.

To prove the last statement, we observe that in our case
\[
  \varGamma=\begin{pmatrix}
  \gamma_{11}&\cdots&\gamma_{1m}\\
  \mb g_1&\cdots&\mb g_m\end{pmatrix}.
\]
The conditions $\varGamma A=0$ and $\rank\varGamma=m-n$ imply that
$GA=0$ and $\rank G=m-n-1$. Finally, by comparing~\eqref{iabrn}
with~\eqref{bplane} we obtain $\varGamma\mb
b=\begin{pmatrix}c\\\mathbf 0\end{pmatrix}$. Since $b_i=1$, we get
$\sum_{i=1}^m\mb g_i=\mathbf 0$. Thus, $G=(\mb g_1,\ldots,\mb
g_m)$ is a Gale diagram of~$P^*$.
\end{proof}

\begin{corollary}\label{Pbound}
A polyhedron $P=P(A,\mb b)$ is bounded if and only if the vectors
$\mb a_1,\ldots,\mb a_m$ satisfy $\alpha_1\mb
a_1+\cdots+\alpha_m\mb a_m=\mathbf0$ for some positive
numbers~$\alpha_k$.
\end{corollary}
\begin{proof}
If $\mb a_1,\ldots,\mb a_m$ satisfy $\sum_{k=1}^m\alpha_k\mb
a_k=\mathbf0$ with positive~$\alpha_k$, then we can take
$\sum_{k=1}^m\alpha_k y_k=\sum_{k=1}^m\alpha_k b_k$ as the first
equation defining the $n$-plane $i_{A,\mb b}(\R^n)$ in~$\R^m$. It
follows that $i_{A,\mb b}(P)$ is contained in the intersection of
the hyperplane $\sum_{k=1}^m\alpha_k y_k=\sum_{k=1}^m\alpha_k b_k$
with $\R^m_\ge$, which is bounded since all $\alpha_k$ are
positive. Therefore, $P$ is bounded.

Conversely, if $P$ is bounded, then it follows from
Proposition~\ref{propcf} and Gale duality that $\mb a_1,\ldots,\mb
a_m$ satisfy $\gamma_{11}\mb a_1+\cdots+\gamma_{1m}\mb
a_m=\mathbf0$ with $\gamma_{1k}>0$.
\end{proof}
A Gale diagram of $P^*$ encodes its combinatorics (and the
combinatorics of~$P$) completely. We give the corresponding
statement in the generic case only:

\begin{proposition}\label{galecomb}
Assume that~\eqref{ptope} is a generic presentation with $b_i=1$.
Let $P^*=\conv(\mb a_1,\ldots,\mb a_m)$ be the polar simplicial
polytope and $G=(\mb g_1,\ldots,\mb g_m)$ be its Gale diagram.
Then the following conditions are equivalent:
\begin{itemize}
\item[(a)]
$F_{i_1}\cap\cdots\cap F_{i_k}\ne\varnothing$ in~$P=P(A,\mathbf
1)$;

\item[(b)] $\conv(\mb a_{i_1},\ldots,\mb a_{i_k})$ is a face
of~$P^*$;

\item[(c)]
$\mathbf 0\in\conv\bigr(\mb g_j\colon
j\notin\{i_1,\ldots,i_k\}\bigl)$.
\end{itemize}
\end{proposition}
\begin{proof}
The equivalence (a)$\Leftrightarrow\,$(b) follows from
Theorems~\ref{polarity} and~\ref{dualpol}.

(b)$\Rightarrow\,$(c). Let $\conv(\mb a_{i_1},\ldots,\mb a_{i_k})$
be a face of~$P^*$. We first observe that each of $\mb
a_{i_1},\ldots,\mb a_{i_k}$ is a vertex of this face, as otherwise
presentation~\eqref{ptope} is not generic. By definition of a
face, there exists a linear function $\xi$ such that $\xi(\mb
a_j)=0$ for $j\in\{i_1,\ldots,i_k\}$ and $\xi(\mb a_j)>0$
otherwise. The condition $\mathbf 0\in\mathop{\mathrm{int}}P^*$
implies that $\xi({\bf 0})>0$, and we may assume that $\xi$ has
the form $\xi(\mb u)=\langle\mb u,\mb x\rangle+1$ for some $\mb
x\in\R^n$. Set $y_j=\xi(\mb a_j)=\langle\mb a_j,\mb x\rangle+1$,
i.e. $\mb y=A^t\mb x+\mathbf 1$. We have
\[
  \sum_{j\notin\{i_1,\ldots,i_k\}}\mb g_jy_j=
  \sum_{j=1}^m\mb g_j y_j=G\mb y=G(A^t\mb x+\mathbf 1)
  =G\mathbf 1=\sum_{j=1}^m\mb g_j=\mathbf 0,
\]
where $y_j>0$ for $j\notin\{i_1,\ldots,i_k\}$. It follows that
$\mathbf 0\in\conv(\mb g_j\colon j\notin\{i_1,\ldots,i_k\})$.

(c)$\Rightarrow\,$(b). Let $\sum_{j\notin\{i_1,\ldots,i_k\}}\mb
g_jy_j=\mathbf 0$ with $y_j\ge0$ and at least one $y_j$ nonzero.
This is a linear relation between the vectors~$\mb g_j$. The space
of such linear relations has basis formed by the columns of the
matrix $\widetilde A^t=(A^t\;\mathbf 1)$. Hence, there exists $\mb
x\in\R^n$ and $b\in\R$ such that $y_j=\langle\mb a_j,\mb
x\rangle+b$. The linear function $\xi(\mb u)=\langle\mb u,\mb
x\rangle+b$ takes zero values on $\mb a_j$ with
$j\in\{i_1,\ldots,i_k\}$ and takes nonnegative values on the
other~$\mb a_j$. Hence, $\mb a_{i_1},\ldots,\mb a_{i_k}$ is a
subset of the vertex set of a face. Since $P^*$ is simplicial,
$\mb a_{i_1},\ldots,\mb a_{i_k}$ is a vertex set of a face.
\end{proof}

\begin{remark}
We allow redundant inequalities in Proposition~\eqref{galecomb}.
In this case we obtain the equivalences
\[
  F_i=\varnothing\quad\Leftrightarrow\quad
  \mb a_i\in\conv(\mb a_j\colon j\ne i)\quad\Leftrightarrow\quad
  \mathbf 0\notin\conv(\mb g_j\colon j\ne i).
\]
\end{remark}

A configuration of vectors $G=(\mb g_1,\ldots,\mb g_m)$ in
$\R^{m-n-1}$ with the property
\[
  \mathbf 0\in\conv\bigr(\mb g_j\colon
  j\notin\{i_1,\ldots,i_k\}\bigl)\quad\Leftrightarrow\quad
  \conv(\mb
  a_{i_1},\ldots,\mb a_{i_k})\text{ is a face of }P^*
\]
is called a \emph{combinatorial Gale diagram} of~$P^*=\conv(\mb
a_1,\ldots,\mb a_m)$. For example, a configuration obtained by
multiplying each vector in a Gale diagram by a positive number is
a combinatorial Gale diagram. Furthermore, the vectors of a
combinatorial Gale diagram can be moved as long as the origin does
not cross the `walls', i.e. affine hyperplanes spanned by subsets
of $\mb g_1,\ldots,\mb g_m$. A combinatorial Gale diagram of $P^*$
is a Gale diagram of a polytope which is combinatorially
equivalent to~$P^*$.

Gale diagrams provide an efficient tool for studying the
combinatorics of higher-dimensional polytopes with few vertices,
as in this case a Gale diagram translates the higher-dimensional
structure to a low-dimensional one. For example, Gale diagrams are
used to classify $n$-polytopes with up to $n+3$ vertices and to
find unusual examples when the number of vertices is~$n+4$,
see~\cite[\S6.5]{zieg95}.

\section{Intersections of quadrics}\label{intquad}
Here we describe the correspondence between
polyhedra~\eqref{ptope} and intersections of quadrics.

\subsection{From polyhedra to quadrics}
\begin{construction}[{\cite{bu-pa02}, see also~\cite[\S3]{b-p-r07}}]
Let $P=P(A,\mb b)$ be a presentation~\eqref{ptope} of a polyhedron
with a vertex. Recall the map $i_{A,\mb b}\colon\R^n\to\R^m$, \
$\mb x\mapsto A^t\mb x+\mb b$ (see Construction~\ref{dist}). It
embeds $P$ into $\R^m_\ge$ (since the vectors $\mb a_1,\ldots,\mb
a_m$ span~$\R^n$). We define the space $\mathcal Z_{A,\mb b}$ from
the commutative diagram
\begin{equation}\label{cdiz}
\begin{CD}
  \mathcal Z_{A,\mb b} @>i_Z>>\C^m\\
  @VVV\hspace{-0.2em} @VV\mu V @.\\
  P @>i_{A,\mb b}>> \R^m_\ge
\end{CD}
\end{equation}
where $\mu(z_1,\ldots,z_m)=(|z_1|^2,\ldots,|z_m|^2)$. The torus
$\T^m$ acts on $\mathcal Z_{A,\mb b}$ with quotient $P$, and $i_Z$
is a $\T^m$-equivariant embedding.

By replacing $y_k$ by $|z_k|^2$ in the equations defining the
affine plane $i_{A,\mb b}(\R^n)$ (see~\eqref{iabrn}) we obtain
that $\mathcal Z_{A,\mb b}$ embeds into $\C^m$ as the set of
common zeros of $m-n$ quadratic equations (\emph{Hermitian
quadrics}):
\begin{equation}\label{zpqua}
  i_Z(\mathcal Z_{A,\mb b})=\Bigl\{\mb z\in\C^m\colon\sum_{k=1}^m\gamma_{jk}|z_k|^2=
  \sum_{k=1}^m\gamma_{jk}b_k,\;\text{ for }1\le j\le m-n\Bigr\}.
\end{equation}
\end{construction}

The following property of $\mathcal Z_{A,\mb b}$ follows easily
from its construction.

\begin{proposition}\label{easyzp}
Given a point $z\in\mathcal Z_{A,\mb b}$, the $j$th coordinate of
$i_Z(z)\in\C^m$ vanishes if and only if $z$ projects onto a point
$\mb x\in P$ such that $\mb x\in F_j$.
\end{proposition}

\begin{theorem}
\label{zpsmooth} The following conditions are equivalent:
\begin{itemize}
\item[(a)]
presentation~\eqref{ptope} determined by the data $(A,\mb b)$ is
generic;
\item[(b)]
the intersection of quadrics in~\eqref{zpqua} is nonempty and
nondegenerate, so that $\mathcal Z_{A,\mb b}$ is a smooth manifold
of dimension $m+n$.
\end{itemize}
Furthermore, under these conditions the embedding
$i_Z\colon\mathcal Z_{A,\mb b}\to\C^m$ has $\T^m$-equivariantly
trivial normal bundle; a $\T^m$-framing is determined by a choice
of matrix $\varGamma$ in~\eqref{iabrn}.
\end{theorem}
\begin{proof}
For simplicity we identify $\mathcal Z_{A,\mb b}$ with its
embedding $i_Z(\mathcal Z_{A,\mb b})\subset\C^m$. We calculate the
gradients of the $m-n$ quadrics in~\eqref{zpqua} at a point $\mb
z=(x_1,y_1,\ldots,x_m,y_m)\in\mathcal Z_{A,\mb b}$, where
$z_k=x_k+iy_k$:
\begin{equation}\label{grve}
  2\left(\gamma_{j1}x_1,\,\gamma_{j1}y_1,\,\dots,\,\gamma_{jm}x_m,\,\gamma_{jm}y_m\right),\quad
  1\le j\le m-n.
\end{equation}
These gradients form the rows of the $(m-n)\times 2m$ matrix
$2\varGamma\varDelta$, where
\[
\varDelta\;=\;
\begin{pmatrix}
  x_1&y_1& \ldots & 0  &0 \\
  \vdots & \vdots & \ddots & \vdots & \vdots\\
  0  & 0 & \ldots &x_m &y_m
\end{pmatrix}.
\]
Let $I=\{i_1,\ldots,i_k\}=\{i\colon z_i=0\}$ be the set of zero
coordinates of~$\mb z$. Then the rank of the gradient matrix
$2\varGamma\varDelta$ at $\mb z$ is equal to the rank of the
$(m-n)\times(m-k)$ matrix $\varGamma_{\widehat I}$ obtained by
deleting the columns $i_1,\ldots,i_k$ from~$\varGamma$.

Now let~\eqref{ptope} be a generic presentation. By
Proposition~\ref{easyzp}, a point $\mb z$ with
$z_{i_1}=\cdots=z_{i_k}=0$ projects to a point in
$F_{i_1}\cap\cdots\cap F_{i_k}\ne\varnothing$. Hence the vectors
$\mb a_{i_1},\ldots,\mb a_{i_k}$ are linearly independent. By
Theorem~\ref{galespan}, the rank of $\varGamma_{\widehat I}$
is~$m-n$. Therefore, the intersection of quadrics~\eqref{zpqua} is
nondegenerate.

On the other hand, if~\eqref{ptope} is not generic, then there is
a point $\mb z\in\mathcal Z_{A,\mb b}$ such that the vectors
$\{\mb a_{i_1},\ldots,\mb a_{i_k}\colon
z_{i_1}=\cdots=z_{i_k}=0\}$ are linearly dependent. By
Theorem~\ref{galespan}, the columns of the corresponding matrix
$\varGamma_{\widehat I}$ do not span~$\R^{m-n}$, so its rank is
less than~$m-n$ and the intersection of quadrics~\eqref{zpqua} is
degenerate at~$\mb z$.

The last statement follows from the fact that $\mathcal Z_{A,\mb
b}$ is a nondegenerate intersection of quadratic surfaces, each of
which is $\T^m$-invariant.
\end{proof}

\subsection{From quadrics to polyhedra}
This time we start with an intersection of $m-n$ Hermitian
quadrics in~$\C^m$:
\begin{equation}\label{zgamma}
  \mathcal Z_{\varGamma,\delta}=\Bigl\{\mb z=(z_1,\ldots,z_m)\in\C^m\colon
  \sum_{k=1}^m\gamma_{jk}|z_k|^2=\delta_j,\quad\text{for }
  1\le j\le m-n\Bigr\}.
\end{equation}
The coefficients of quadrics form an $(m-n)\times m$-matrix
$\varGamma=(\gamma_{jk})$, and we denote its column vectors by
$\gamma_1,\ldots,\gamma_m$. We also consider the column vector of
the right hand sides,
$\delta=(\delta_1,\ldots,\delta_{m-n})^t\in\R^{m-n}$.

These intersections of quadrics are considered up to \emph{linear
equivalence}, which corresponds to applying a nondegenerate linear
transformation of $\R^{m-n}$ to $\varGamma$ and~$\delta$.
Obviously, such a linear equivalence does not change the
set~$\mathcal Z_{\varGamma,\delta}$.

We denote by $\R_\ge\langle\gamma_1,\ldots,\gamma_m\rangle$ the
cone spanned by the vectors $\gamma_1,\ldots,\gamma_m$ (i.e., the
set of linear combinations of these vectors with nonnegative real
coefficients).

A version of the following proposition appeared in~\cite{lope89},
and the proof below is a modification of the argument
in~\cite[Lemma~0.3]{bo-me06}. It allows us to determine the
nondegeneracy of an intersection of quadrics directly from the
data $(\varGamma,\delta)$:

\begin{proposition}\label{zgsmooth}
The intersection of quadrics~\eqref{zgamma} is nonempty and
nondegenerate if and only if the following two conditions are
satisfied:
\begin{itemize}
\item[(a)] $\delta\in
\R_\ge\langle\gamma_1,\ldots,\gamma_m\rangle$;\\[-0.7\baselineskip]

\item[(b)] if $\delta\in\R_\ge\langle
\gamma_{i_1},\ldots\gamma_{i_k}\rangle$, then $k\ge m-n$.
\end{itemize}
Under these conditions, $\mathcal Z_{\varGamma,\delta}$ is a
smooth submanifold in $\C^m$ of dimension $m+n$, and the vectors
$\gamma_1,\ldots,\gamma_m$ span~$\R^{m-n}$.
\end{proposition}
\begin{proof}
First, assume that (a) and (b) are satisfied. Then (a) implies
that $\mathcal Z_{\varGamma,\delta}\ne\varnothing$. Let $\mb
z\in\mathcal Z_{\varGamma,\delta}$. Then the rank of the matrix of
gradients of~\eqref{zgamma} at $\mb z$ is equal to
$\mathop{\mathrm{rk}}\{\gamma_k\colon z_k\ne0\}$. Since $\mb
z\in\mathcal Z_{\varGamma,\delta}$, the vector $\delta$ is in the
cone generated by those $\gamma_k$ for which $z_k\ne0$. By the
Carath\'eodory Theorem, $\delta$ belongs to the cone generated by
some $m-n$ of these vectors, that is,
$\delta\in\R_\ge\langle\gamma_{k_1},\ldots,\gamma_{k_{m-n}}\rangle$,
where $z_{k_i}\ne0$ for $i=1,\ldots,m-n$. Moreover, the vectors
$\gamma_{k_1},\ldots,\gamma_{k_{m-n}}$ are linearly independent
(otherwise, again by the Carath\'eodory Theorem, we obtain a
contradiction with~(b)). This implies that the $m-n$ gradients of
quadrics in~\eqref{zgamma} are linearly independent at~$\mb z$,
and therefore $\mathcal Z_{\varGamma,\delta}$ is smooth and
$(m+n)$-dimensional.

To prove the other implication we observe that if (b) fails, that
is, $\delta$ is in the cone generated by some $m-n-1$ vectors of
$\gamma_1,\ldots,\gamma_m$, then there is a point $\mb
z\in\mathcal Z_{\varGamma,\delta}$ with at least $n+1$ zero
coordinates. The gradients of quadrics in~\eqref{zgamma} cannot be
linearly independent at such~$\mb z$.
\end{proof}

The torus $\T^m$ acts on $\mathcal Z_{\varGamma,\delta}$, and the
quotient $\mathcal Z_{\varGamma,\delta}/\T^m$ is identified with
the set of nonnegative solutions of the system of $m-n$ linear
equations
\begin{equation}\label{linsys}
  \sum_{k=1}^m\gamma_ky_k=\delta.
\end{equation}
This set may be described as a convex polyhedron $P(A,\mb b)$
given by~\eqref{ptope}, where $(b_1,\ldots,b_m)$ is any solution
of~\eqref{linsys} and the vectors $\mb a_1,\ldots,\mb a_m\in\R^n$
form the transpose of a basis of solutions of the homogeneous
system $\sum_{k=1}^m\gamma_ky_k=\mathbf 0$.  We refer to $P(A,\mb
b)$ as the \emph{associated polyhedron} of the intersection of
quadrics~$\mathcal Z_{\varGamma,\delta}$. If the vectors
$\gamma_1,\ldots,\gamma_m$ span $\R^{m-n}$, then $\mb
a_1,\ldots,\mb a_m$ span~$\R^n$. In this case the two vector
configurations are Gale dual.

We summarise the results and constructions of this section as
follows:

\begin{theorem}\label{polquad}
A presentation of a polyhedron
\[
  P=P(A,\mb b)=\bigl\{\mb x\in\R^n\colon\langle\mb a_i,\mb
  x\rangle+b_i\ge0\quad\text{for }
  i=1,\ldots,m\bigr\}
\]
(with $\mb a_1,\ldots,\mb a_m$ spanning~$\R^n$) defines an
intersection of Hermitian quadrics
\[
  \mathcal Z_{\varGamma,\delta}=\Bigl\{\mb
  z=(z_1,\ldots,z_m)\in\C^m\colon
  \sum_{k=1}^m\gamma_{jk}|z_k|^2=\delta_j\quad\text{for }
  j=1,\ldots,m-n\Bigr\}.
\]
(with $\gamma_1,\ldots,\gamma_m$ spanning~$\R^{m-n}$) uniquely up
to a linear isomorphism of~$\R^{m-n}$, and an intersection of
quadrics~$\mathcal Z_{\varGamma,\delta}$ defines a
presentation~$P(A,\mb b)$ uniquely up to an isomorphism of~$\R^n$.

The systems of vectors $\mb a_1,\ldots,\mb a_m\in\R^n$ and
$\gamma_1,\ldots,\gamma_m\in\R^{m-n}$ are Gale dual, and the
vectors $\mb b\in\R^{m}$ and $\delta\in\R^{m-n}$ are related by
the identity $\delta=\varGamma\mb b$.

The intersection of quadrics $\mathcal Z_{\varGamma,\delta}$ is
nonempty and nondegenerate if and only if the presentation
$P(A,\mb b)$ is generic.
\end{theorem}

\begin{example}[$m=n+1$: one quadric]\label{mamsimplex}
If presentation~\eqref{ptope} is generic and $P$ is bounded, then
$m\ge n+1$. The case $m=n+1$ corresponds to a simplex. The
\emph{standard} simplex is given by the following $n+1$
inequalities:
\[
  \varDelta^n=\bigl\{\mb x\in\R^n\colon x_i\ge0\quad\text{for }
  i=1,\ldots,n,\quad\text{and } -x_1-\cdots-x_n+1\ge0\bigr\}.
\]
We therefore have $\mb a_i=\mb e_i$ (the $i$th standard basis
vector) for $i=1,\ldots,n$ and $\mb a_{n+1}=-\mb e_1-\cdots-\mb
e_n$. By taking $\varGamma=(1\cdots1)$ we obtain
\[
  \mathcal Z_{A,\mb b}=\mathbb S^n=\{\mb z\in\C^{n+1}\colon
  |z_1|^2+\cdots+|z_{n+1}|^2=1\}.
\]

More generally, a presentation~\eqref{ptope} with $m=n+1$ and $\mb
a_1,\ldots,\mb a_n$ spanning $\R^n$ can be taken by an isomorphism
of $\R^n$ to the form
\[
  P=\bigl\{\mb x\in\R^n\colon x_i+b_i\ge0\quad\text{for }
  i=1,\ldots,n,\quad\text{and }
  -c_1x_1-\cdots-c_nx_n+b_{n+1}\ge0\bigr\}.
\]
We therefore have $\varGamma=(c_1\cdots c_n\, 1)$, and $\mathcal
Z_{A,\mb b}$ is given by the single equation
\[
  c_1|z_1|^2+\cdots+c_n|z_n|^2+|z_{n+1}|^2=c_1b_1+\cdots+c_nb_n+b_{n+1}.
\]
If the presentation is generic and bounded, then $\mathcal
Z_{A,\mb b}$ is nonempty, nondegenerate and bounded by
Theorem~\ref{zpsmooth}. This implies that all $c_i$ and the right
hand side above are positive, and $\mathcal Z_{A,\mb b}$ is an
ellipsoid.
\end{example}

\section{Moment-angle manifolds from polytopes}\label{mampol}
Here we consider the case when the polyhedron $P$ defined
by~\eqref{ptope} (or equivalently, intersection of
quadrics~\eqref{zgamma}) is bounded. We also assume
that~\eqref{ptope} is a generic presentation, so that $P$ is an
$n$-dimensional simple polytope and $\mathcal Z_{A,\mb b}=\mathcal
Z_{\varGamma,\delta}$ is an $(m+n)$-dimensional closed smooth
manifold.

We start with the construction of an identification space, which
goes back to the work of Vinberg~\cite{vinb71} on Coxeter groups
and was presented in the form described below in the work of Davis
and Januszkiewicz~\cite{da-ja91}. It was the first construction of
what later became known as the moment-angle manifold.

\begin{construction}
Let $[m]=\{1,\ldots,m\}$ be the standard $m$-element set.
For each $I\subset[m]$ we consider the coordinate subtorus
\[
  \T^I=\{(t_1,\ldots,t_m)\in\T^m\colon t_j=1\text{ for
  }j\notin I\}\;\subset\;\T^m.
\]
In particular, $T^\varnothing$ is the trivial
subgroup~$\{1\}\subset\T^m$.

Define the map $\R_\ge\times\T\to\C$ by $(y,t)\mapsto yt$. Taking
product we obtain a map $\R^m_\ge\times\T^m\to\C^m$. The preimage
of a point $\mb z\in\C^m$ under this map is $\mb
y\times\T^{\omega(\mb z)}$, where $y_i=|z_i|$ for $1\le i\le m$
and
\[
  \omega(\mb z)=\{i\colon z_i=0\}\subset[m]
\]
is the set of zero coordinates of~$\mb z$. Therefore, $\C^m$ can
be identified with the quotient
\begin{equation}\label{Cmids}
  \R^m_\ge\times\T^m/{\sim\:}\quad\text{where }(\mb y,\mb t_1)\sim(\mb y,\mb
  t_2)\text{ if }\mb t_1^{-1}\mb t_2\in\T^{\omega(\mb y)}.
\end{equation}

Given $\mb x\in P$, set
\[
  I_{\mb x}=\{i\in[m]\colon{\mb x}\in F_i\}
\]
(the set of facets containing~$\mb x$).
\end{construction}

\begin{proposition}\label{zpids}
$\mathcal Z_{A,\mb b}$ is $\T^m$-equivariantly homeomorphic to the
quotient
\[
  P\times\T^m/{\sim\:}\quad\text{where }
  (\mb x,\mb t_1)\sim(\mb x,\mb t_2)\:\text{ if }\:\mb t_1^{-1}\mb t_2\in
  \T^{I_{\mb x}}.
\]
\end{proposition}
\begin{proof}
Using~\eqref{cdiz}, we identify $\mathcal Z_{A,\mb b}$ with
$i_{A,\mb b}(P)\times\T^m/{\sim\:}$, where $\sim$ is the
equivalence relation from~\eqref{Cmids}. A point $\mb x\in P$ is
mapped by $i_{A,\mb b}$ to $\mb y\in\R^m_\ge$ with $I_{\mb
x}=\omega(\mb y)$.
\end{proof}

An important corollary of this construction is that the
topological type of $\mathcal Z_{A,\mb b}$ depends only on the
combinatorics of~$P$:

\begin{proposition}\label{zpcomb}
Assume given two generic presentations:
\[
  P=\bigl\{\mb x\in\R^n\colon(A^t\mb x+\mb b)_i\ge0\bigr\}\quad\text{and}
  \quad P'=\bigl\{\mb x\in\R^n\colon(A'^t\mb x+\mb b')_i\ge 0\bigr\}
\]
such that $P$ and $P'$ are combinatorially equivalent simple
polytopes.
\begin{itemize}
\item[(a)] If both presentations are irredundant, then the corresponding manifolds
$\mathcal Z_{A,\mb b}$ and $\mathcal Z_{A',\mb b'}$ are
$\T^m$-equivariantly homeomorphic.

\item[(b)] If the second presentation is obtained from the first one by
adding $k$ redundant inequalities, then $\mathcal Z_{A',\mb b'}$
is homeomorphic to a product of $\mathcal Z_{A,\mb b}$ and a
$k$-torus~$T^k$.
\end{itemize}
\end{proposition}
\begin{proof} (a) By Proposition~\ref{zpids},
$\mathcal Z_{A,\mb b}\cong P\times\T^m/{\sim\:}$ and $\mathcal
Z_{A',\mb b'}\cong P'\times\T^m/{\sim\:}$. If both presentations
are irredundant, then any $F_i$ is a facet of $P$, and the
equivalence relation $\sim$ depends on the face structure of $P$
only. Therefore, any homeomorphism $P\to P'$ preserving the face
structure extends to a $\T^m$-homeomorphism
$P\times\T^m/{\sim\:}\to P'\times\T^m/{\sim\:}$.

(b) Suppose the first presentation has $m$ inequalities, and the
second has~$m'$ inequalities, so that $m'-m=k$. Let $J\subset[m']$
be the subset corresponding to the $k$ added redundant
inequalities; we may assume that $J=\{m+1,\ldots,m'\}$. Since
$F_j=\varnothing$ for any $j\in J$, we have $I_{\mb x}\cap
J=\varnothing$ for any $\mb x\in P'$. Therefore, the equivalence
relation $\sim$ does not affect the factor $\T^J\subset\T^{m'}$,
and we have
\[
  \mathcal Z_{A',\mb b'}\cong P'\times\T^{m'}/{\sim\:}\cong
  (P\times\T^m/{\sim\:})\times\T^J\cong\mathcal Z_{A,\mb b}\times
  T^k.\qedhere
\]
\end{proof}

\begin{remark}
A $\T^m$-homeomorphism in Proposition~\ref{zpcomb}~(a) can be
replaced by a $\T^m$-diffeomorphism (with respect to the smooth
structures of Theorem~\ref{zpsmooth}), but the proof is more
technical. It follows from the fact that two simple polytopes are
combinatorially equivalent if and only if they are diffeomorphic
as `smooth manifolds with corners'. For an alternative argument,
see~\cite[Corollary~4.7]{bo-me06}.

Statement (a) remains valid without assuming that the presentation
is generic, although $\mathcal Z_{A,\mb b}$ is not a manifold in
this case.
\end{remark}

\begin{definition}
We refer to the $(m+n)$-dimensional manifold $\mathcal Z_{A,\mb
b}$ defined by any irredundant presentation~\eqref{ptope} of an
$n$-dimensional simple polytope $P$ with $m$ facets as the
\emph{moment-angle manifold} corresponding to $P$, and denote it
by~$\zp$. Moment-angle manifolds appearing in this way are called
\emph{polytopal}; more general moment-angle manifolds will be
considered later.
\end{definition}

\begin{proposition}\label{malink}
The moment-angle manifold $\zp$ is $\T^m$-equivariantly
diffeomorphic to a nondegenerate intersection of quadrics of the
following form:
\begin{equation}\label{link}
  \left\{\begin{array}{lrcl}
  \mb z\in\C^m\colon&\sum_{k=1}^m|z_k|^2&=&1,\\[1mm]
  &\sum_{k=1}^m\mb g_k|z_k|^2&=&\mathbf0,
  \end{array}\right\}
\end{equation}
where $(\mb g_1,\ldots,\mb g_m)\subset\R^{m-n-1}$ is a
combinatorial Gale diagram of~$P^*$.
\end{proposition}
\begin{proof}
It follows from Proposition~\ref{propcf} that $\zp$ is given by
\[
  \left\{\begin{array}{ll}
  \mb z\in\C^m\colon&\gamma_{11}|z_1|^2+\cdots+\gamma_{1m}|z_m|^2=c,\\[1mm]
  &\mb g_1|z_1|^2+\cdots+\mb g_m|z_m|^2=\mathbf0,
  \end{array}\right\}
\]
where $\gamma_{1k}$ and $c$ are positive. Divide the first
equation by~$c$, and then replace each $z_k$ by $\sqrt\frac
c{\gamma_{1k}}z_k$. As a result, each $\mb g_k$ is multiplied by a
positive number, so that $(\mb g_1,\ldots,\mb g_m)$ remains to be
a combinatorial Gale diagram for~$P^*$.
\end{proof}

By adapting Proposition~\ref{zgsmooth} to the special case of
quadrics~\eqref{link}, we obtain

\begin{proposition}\label{nondeglink}
The intersection of quadrics given by~\eqref{link} is nonempty
nondegenerate if and only if the following two conditions are
satisfied:
\begin{itemize}
\item[(a)] $\mathbf0\in
\conv(\mb g_1,\ldots,\mb g_m)$;

\item[(b)] if ${\mathbf 0}\in\conv(\mb g_{i_1},\ldots,\mb g_{i_k})$,
then $k\ge m-n$.
\end{itemize}
\end{proposition}

Following~\cite{bo-me06}, we refer to a nondegenerate
intersection~\eqref{link} of $m-n-1$ homogeneous quadrics with a
unit sphere in~$\C^m$ as a \emph{link}. We therefore obtain that
any moment-angle manifold is diffeomorphic to a link, and any link
is a product of a moment-angle manifold and a torus.

As we have seen in Example~\ref{mamsimplex}, the moment-angle
manifold corresponding to an $n$-simplex is a sphere~$S^{2n+1}$.
This is also the link of an empty system of homogeneous quadrics,
corresponding to the case $m=n+1$.

\begin{example}[$m=n+2$: two quadrics]\label{prodsimex}
A polytope $P$ defined by $m=n+2$ inequalities either is
combinatorially equivalent to a product of two simplices (when
there are no redundant inequalities), or is a simplex (when one
inequality is redundant). In the case $m=n+2$ the
link~\eqref{link} has the form
\[
  \left\{\begin{array}{ll}
  \mb z\in\C^m\colon&|z_1|^2+\cdots+|z_m|^2=1,\\[1mm]
  &g_1|z_1|^2+\cdots+g_m|z_m|^2=1,
  \end{array}\right\}
\]
where $g_k\in\R$. Condition~(b) of Proposition~\ref{nondeglink}
implies that all $g_i$ are nonzero; assume that there are $p$
positive and $q=m-p$ negative numbers among them. Then
condition~(a) implies that $p>0$ and $q>0$. Therefore, the link is
the intersection of the cone over a product of two ellipsoids of
dimensions $2p-1$ and $2q-1$ (given by the second quadric) with a
unit sphere of dimension $2m-1$ (given by the first quadric). Such
a link is diffeomorphic to $S^{2p-1}\times S^{2q-1}$. The case
$p=1$ or $q=1$ corresponds to one redundant inequality. In the
irredundant case ($P$ is a product
$\varDelta^{p-1}\times\varDelta^{q-1}$, $p,q>1$) we obtain that
$\zp\cong S^{2p-1}\times S^{2q-1}$.
\end{example}

\section{Hamiltonian toric manifolds and moment maps}\label{symred}
Particular examples of polytopal moment-angle manifolds $\zp$
appear as level sets for the moment maps used in the construction
of Hamiltonian toric manifolds via symplectic reduction. In this
case the left hand sides of the equations in~\eqref{zpqua} are
quadratic Hamiltonians of a torus action on~$\C^m$.

\subsection{Symplectic reduction}
We briefly review the background material in symplectic geometry,
referring the reader to monographs by Audin~\cite{audi91} and
Guillemin~\cite{guil94} for further details.

A \emph{symplectic manifold} is a pair $(W,\omega)$ consisting of
a smooth manifold $W$ and a closed differential 2-form $\omega$
which is nondegenerate at each point. The dimension of a
symplectic manifold $W$ is necessarily even.

Assume now that a torus $T$ acts on $W$ preserving the symplectic
form~$\omega$. We denote the Lie algebra of the torus~$T$ by
$\mathfrak t$ (since $T$ is commutative, its Lie algebra is
trivial, but the construction can be generalised to noncommutative
Lie groups). Given an element $\mb v\in\mathfrak t$, we denote by
$X_{\mb v}$ the corresponding $T$-invariant vector field on~$W$.
The torus action is called \emph{Hamiltonian} if the 1-form
$\omega(X_{\mb v},\,\cdot\,)$ is exact for any $\mb v\in\mathfrak
t$. In other words, an action is Hamiltonian if for any $\mb
v\in\mathfrak t$ there exist a function $H_{\mb v}$ on $W$ (called
a \emph{Hamiltonian}) satisfying the condition
\[
  \omega(X_{\mb v},Y)=dH_{\mb v}(Y)
\]
for any vector field $Y$ on~$W$. The function $H_{\mb v}$ is
defined up to addition of a constant.  Choose a basis $\{\mb
e_i\}$ in $\mathfrak t$ and the corresponding Hamiltonians
$\{H_{\mb e_i}\}$. Then the \emph{moment map}
\[
  \mu\colon W\to\mathfrak t^*,\qquad (x,\mb e_i)\mapsto
  H_{\mb e_i}(x)
\]
(where $x\in W$) is defined. Observe that changing the
Hamiltonians $H_{\mb e_i}$ by constants results in shifting the
image of $\mu$ by a vector in~$\mathfrak t^*$. According to a
theorem of Atiyah and Guillemin--Sternberg, the image $\mu(W)$ of
the moment map is convex, and if $W$ is compact then $\mu(W)$ is a
convex polytope in~$\mathfrak t^*$.

\begin{example}\label{simcm}
The most basic example is $W=\C^m$ with symplectic form
\[
  \omega=i\sum_{k=1}^m dz_k\wedge d\overline{z}_k=
  2\sum_{k=1}^mdx_k\wedge dy_k,
\]
where $z_k=x_k+iy_k$. The coordinatewise action of the torus
$\T^m$ on $\C^m$ is Hamiltonian. The moment map
$\mu\colon\C^m\to\R^m$ is given by
$\mu(z_1,\ldots,z_m)=(|z_1|^2,\ldots,|z_m|^2)$. The image of $\mu$
is the positive orthant~$\R^m_\ge$.
\end{example}

\begin{construction}[symplectic reduction]
Assume given a Hamiltonian action of a torus $T$ on a symplectic
manifold~$W$. Assume further that the moment map $\mu\colon
W\to\mathfrak t^*$ is \emph{proper}, i.e. $\mu^{-1}(V)$ is compact
for any compact subset $V\subset\mathfrak t^*$  (this is always
the case if $W$ itself is compact). Let $\mb u\in\mathfrak t^*$ be
a \emph{regular value} of the moment map, i.e. the differential
$\mathcal T_x W\to\mathfrak t^*$ is surjective for all
$x\in\mu^{-1}(\mb u)$. Then the level set $\mu^{-1}(\mb u)$ is a
smooth compact $T$-invariant submanifold in~$W$. Furthermore the
$T$-action on $\mu^{-1}(\mb u)$ is almost free, i.e. all
stabilisers are finite subgroups.

Assume now that the $T$-action on $\mu^{-1}(\mb u)$ is free. The
restriction of the symplectic form $\omega$ to $\mu^{-1}(\mb u)$
may be degenerate. However, the quotient manifold $\mu^{-1}(\mb
u)/T$ is endowed with a unique symplectic form $\omega'$ such that
\[
  p^*\omega'=i^*\omega,
\]
where $i\colon\mu^{-1}(\mb u)\to W$ is the inclusion and $p\colon
\mu^{-1}(\mb u)\to \mu^{-1}(\mb u)/T$ the projection.

We therefore obtain a new symplectic manifold $(\mu^{-1}(\mb
u)/T,\omega')$ which is referred to as the \emph{symplectic
reduction}, or the \emph{symplectic quotient} of $(W,\omega)$
by~$T$.

The construction of symplectic reduction works also under milder
assumptions on the action (see~\cite{du-he82} and more references
there), but the generality described here will be enough for our
purposes.
\end{construction}

\subsection{The toric case}
We want to study symplectic quotients of $\C^m$ by torus subgroups
$T\subset\T^m$. Such a subgroup of dimension $m-n$ has the form
\begin{equation}\label{Tgamma}
  T_\varGamma=
  \bigl\{\bigr(e^{2\pi i\langle\gamma_1,\varphi\rangle},
  \ldots,e^{2\pi i\langle\gamma_m,\varphi\rangle}\bigl)
  \in\T^m\bigr\},
\end{equation}
where $\varphi\in\R^{m-n}$ is an $(m-n)$-dimensional parameter,
and $\varGamma=(\gamma_1,\ldots,\gamma_m)$ is a set of $m$ vectors
in~$\R^{m-n}$. In order for $T_\varGamma$ to be an $(m-n)$-torus,
the configuration of vectors $\gamma_1,\ldots,\gamma_m$ must be
\emph{rational}, i.e. the set of all their integral linear
combinations $L=\Z\langle\gamma_1,\ldots,\gamma_m\rangle$ must be
an $(m-n)$-dimensional discrete subgroup (\emph{lattice})
in~$\R^{m-n}$. Let
\[
  L^*=\{\lambda^*\in{\mathbb R}^{m-n}\colon
  \langle\lambda^*,\lambda\rangle\in{\mathbb Z}
  \text{ for all }\lambda\in L\}
\]
be the dual lattice. We shall represent the elements of
$T_\varGamma$ by $\varphi\in\R^{m-n}$ occasionally, so that
$T_\varGamma$ is identified with the quotient
$\R^{m-n}/\displaystyle L^*$.

The restricted action of $T_\varGamma\subset\T^m$ on $\C^m$ is
obviously Hamiltonian, and the corresponding moment map is the
composition
\begin{equation}\label{tmoma}
  \mu_\varGamma\colon\C^m\longrightarrow\R^m\longrightarrow
  \mathfrak t_\varGamma^*,
\end{equation}
where $\R^m\to\mathfrak t_\varGamma^*$ is the map of the dual Lie
algebras corresponding to $T_\varGamma\to\T^m$. The map
$\R^m\to\mathfrak t_\varGamma^*$ takes the $i$th basis vector $\mb
e_i\in\R^m$ to $\gamma_i\in\mathfrak t_\varGamma^*$. By choosing a
basis in $L\subset\mathfrak t_\varGamma^*$ we can write the map
$\R^m\to\mathfrak t_\varGamma^*$ by an \emph{integer}
matrix~$\varGamma=(\gamma_{jk})$. The moment map~\eqref{tmoma} is
then given by
\[
  (z_1,\ldots,z_m)\longmapsto
  \Bigl(\sum_{k=1}^m\gamma_{1k}|z_k|^2,\ldots,\sum_{k=1}^m
  \gamma_{m-n,k}|z_k|^2\Bigr).
\]
Its level set $\mu_\varGamma^{-1}(\delta)$ corresponding to a
value $\delta=(\delta_1,\ldots,\delta_{m-n})^t\in\mathfrak
t_\varGamma^*$ is exactly the intersection of quadrics $\mathcal
Z_{\varGamma,\delta}$ given by~\eqref{zgamma}.

To apply the symplectic reduction we need to identify when the
moment map~$\mu_\varGamma$ is proper, find its regular
values~$\delta$, and finally identify when the action of
$T_\varGamma$ on $\mu_\varGamma^{-1}(\delta)=\mathcal
Z_{\varGamma,\delta}$ is free. In Theorem~\ref{propmmap} below,
all these conditions are expressed in terms of the polyhedron $P$
associated with $\mathcal Z_{\varGamma,\delta}$ as described in
Section~\ref{intquad}. We need a couple more definitions before we
state this theorem.

It follows from Gale duality that $\gamma_1,\ldots,\gamma_m$ span
a lattice $L$ in $\R^{m-n}$ if and only if the dual configuration
$\mb a_1,\ldots,\mb a_m$ spans a lattice $N=\Z\langle\mb
a_1,\ldots,\mb a_m\rangle$ in~$\R^n$. We refer to a
presentation~\eqref{ptope} as~\emph{rational} if $\Z\langle\mb
a_1,\ldots,\mb a_m\rangle$ is a lattice.

Recall that for each $\mb x\in P$ we defined
\[
  I_{\mb x}=\{i\in[m]\colon\langle\mb a_i,\mb x\rangle+b_i=0\}
  =\{i\in[m]\colon{\mb x}\in F_i\}
\]
(the set of facets containing~$\mb x$). A polyhedron $P$ is called
\emph{Delzant} if it has a rational presentation~\eqref{ptope}
such that for any $\mb x\in P$ the vectors $\{\mb a_i\colon i\in
I_{\mb x}\}$ constitute a part of a basis of $N=\Z\langle\mb
a_1,\ldots,\mb a_m\rangle$. Equivalently, $P$ is Delzant if it is
simple and for any vertex $\mb x\in P$ the vectors $\mb a_i$
normal to the $n$ facets meeting at~$\mb x$ form a basis of the
lattice~$N$. The term comes from the classification of Hamiltonian
toric manifolds due to Delzant~\cite{delz88}, which we shall
briefly review later.

Now let $\delta\in\mathfrak t_\varGamma$ be a value of the moment
map $\mu_\varGamma\colon\C^m\to\mathfrak t_\varGamma^*$, and
$\mu_\varGamma^{-1}(\delta)=\mathcal Z_{\varGamma,\delta}$ the
corresponding level set, which is an intersection of
quadrics~\eqref{zgamma}. We associate with $\mathcal
Z_{\varGamma,\delta}$ a presentation~\eqref{ptope} as described in
Section~\ref{intquad} (see Theorem~\ref{polquad}).

\begin{theorem}\label{propmmap}
Let $T_\varGamma\subset\T^m$ be a torus subgroup~\eqref{Tgamma},
determined by a rational configuration of vectors
$\gamma_1,\ldots,\gamma_m$.
\begin{itemize}
\item[(a)] The moment map
$\mu_\varGamma\colon\C^m\to\mathfrak t_\varGamma^*$ is proper if
and only if its level set $\mu_\varGamma^{-1}(\delta)$ is bounded
for some (and then for any) value $\delta\in\mathfrak
t_\varGamma^*$. Equivalently, the map $\mu_\varGamma$ is proper if
and only if the Gale dual configuration $\mb a_1,\ldots,\mb a_m$
satisfies $\alpha_1\mb a_1+\cdots+\alpha_m\mb a_m=\mathbf0$ for
some positive numbers~$\alpha_k$.

\item[(b)] $\delta\in\mathfrak t^*_\varGamma$ is a regular value of
$\mu_\varGamma$ if and only if the intersection of quadrics
$\mu_\varGamma^{-1}(\delta)=\mathcal Z_{\varGamma,\delta}$ is
nonempty and nondegenerate. Equivalently, $\delta$ is a regular
value if and only if the associated presentation $P=P(A,\mb b)$ is
generic.

\item[(c)] The action of $T_\varGamma$ on
$\mu_\varGamma^{-1}(\delta)=\mathcal Z_{\varGamma,\delta}$ is free
if and only if the associated polyhedron $P$ is Delzant.
\end{itemize}
\end{theorem}
\begin{proof}
(a) If $\mu_\varGamma$ is proper then
$\mu_\varGamma^{-1}(\delta)\subset\mathfrak t^*_\varGamma$ is
compact, so it is bounded.

Now assume that $\mu_\varGamma^{-1}(\delta)=\mathcal
Z_{\varGamma,\delta}$ is bounded for some~$\delta$. Then the
corresponding polyhedron~$P$ is also bounded. By
Corollary~\ref{Pbound}, this is equivalent to vanishing of a
positive linear combination of $\mb a_1,\ldots,\mb a_m$. This
condition is independent of~$\delta$, and we conclude that
$\mu_\varGamma^{-1}(\delta)$ is bounded for any~$\delta$. Let
$X\subset\mathfrak t^*_\varGamma$ be a compact subset. Since
$\mu_\varGamma^{-1}(X)$ is closed, it is compact whenever it is
bounded. By Proposition~\ref{propcf} we may assume that, for any
$\delta\in X$, the first quadric defining
$\mu_\varGamma^{-1}(\delta)=\mathcal Z_{\varGamma,\delta}$ is
given by $\gamma_{11}|z_1|^2+\cdots+\gamma_{1m}|z_m|^2=\delta_1$
with $\gamma_{1k}>0$. Let $c=\max_{\delta\in X}\delta_1$. Then
$\mu_\varGamma^{-1}(X)$ is contained in the bounded set
\[
  \{\mb z\in\C^m\colon
  \gamma_{11}|z_1|^2+\cdots+\gamma_{1m}|z_m|^2\le c\}
\]
and is therefore bounded. Hence, $\mu_\varGamma^{-1}(X)$ is
compact, and $\mu_\varGamma$ is proper.

(b) The first statement is the definition of a regular value. The
equivalent statement is already proved as Theorem~\ref{zpsmooth}.

(c) We first need to identify the stabilisers of the
$T_\varGamma$-action on $\mu_\varGamma^{-1}(\delta)$. Although the
fact that these stabilisers are finite for a regular value
$\delta$ follows from the general construction of symplectic
reduction, we can prove this directly.

Given a point $\mb z=(z_1,\ldots,z_m)\in\mathcal
Z_{\varGamma,\delta}$, we define the sublattice
\[
  L_{\mb z}=\Z\langle\gamma_i\colon z_i\ne0\rangle\subset L=
  \Z\langle\gamma_1,\ldots,\gamma_m\rangle.
\]

\begin{lemma}\label{afree}
The stabiliser subgroup of $\mb z\in\mathcal Z_{\varGamma,\delta}$
under the action of $T_\varGamma$ is given by $\displaystyle
L^*_{\mb z}/L^*$. Furthermore, if $\mathcal Z_{\varGamma,\delta}$
is nondegenerate, then all these stabilisers are finite, i.e. the
action of $T_\varGamma$ on $\mathcal Z_{\varGamma,\delta}$ is
almost free.
\end{lemma}
\begin{proof}
An element $(e^{2\pi
i\langle\gamma_1,\varphi\rangle},\ldots,e^{2\pi
i\langle\gamma_m,\varphi\rangle})\in T_\varGamma$ fixes a point
$\mb z\in\mathcal Z_\varGamma$ if and only if $e^{2\pi
i\langle\gamma_k,\varphi\rangle}=1$ whenever $z_k\ne0$. In other
words, $\varphi\in T_\varGamma$ fixes $\mb z$ if and only if
$\langle\gamma_k,\varphi\rangle\in\Z$ whenever $z_k\ne0$. The
latter means that $\displaystyle\varphi\in\displaystyle L^*_{\mb
z}$. Since $\displaystyle\varphi\in L^*$ maps to $1\in
T_\varGamma$, the stabiliser of $\mb z$ is $\displaystyle L^*_{\mb
z}/L^*$.

Assume now that $\mathcal Z_{\varGamma,\delta}$ is nondegenerate.
In order to see that $\displaystyle L^*_{\mb z}/L^*$ is finite we
need to check that the sublattice $L_{\mb
z}=\Z\langle\gamma_i\colon z_i\ne0\rangle\subset L$ has full
rank~$m-n$. Indeed, $\mathop{\mathrm{rk}}\{\gamma_i\colon
z_i\ne0\}$ is the rank of the matrix of gradients of quadrics
in~\eqref{zgamma} at~$\mb z$. Since $\mathcal
Z_{\varGamma,\delta}$ is nondegenerate, this rank is $m-n$, as
needed.
\end{proof}

Now we can finish the proof of Theorem~\ref{propmmap}~(c). Assume
that $P$ is Delzant. By Lemma~\ref{afree}, the
$T_\varGamma$-action on $\mathcal Z_{\varGamma,\delta}$ is free if
and only if $L_{\mb z}=L$ for any $\mb z\in\mathcal
Z_{\varGamma,\delta}$. Let $i\colon\Z^k\rightarrow\Z^m$ be the
inclusion of the coordinate sublattice spanned by those $\mb e_i$
for which $z_i=0$, and let $p\colon\Z^m\rightarrow\Z^{m-k}$ be the
projection sending every such $\mb e_i$ to zero. We also have maps
of lattices
\[
  \varGamma^t\colon L^*\to\Z^m,\;
  \mb l\mapsto\bigl(\langle\gamma_1,\mb l\rangle,\ldots,
  \langle\gamma_m,\mb l\rangle\bigr),\quad\text{and}\quad
  A\colon\Z^m\to N,\;\mb e_k\mapsto\mb a_k.
\]
Consider the diagram
\begin{equation}\label{2Zses}
\begin{CD}
  @.@.\begin{array}{c}0\\ \raisebox{2pt}{$\downarrow$}\\
  L^*\end{array}@.@.\\
  @.@.@VV{\varGamma^t}V@.@.\\
  0@>>>\Z^{k}@>i>>\Z^m@>p>>\Z^{m-k} @>>>0\\
  @.@.@VV A V@.@.\\
  @.@.\begin{array}{c}N\\ \downarrow\\ 0\end{array}@.@.\\
\end{CD}
\end{equation}
in which the vertical and horizontal sequences are exact. Then the
Delzant condition is equivalent to that the composition $A\cdot i$
is split injective. The condition $L_{\mb z}=L$ is equivalent to
that $\varGamma\cdot p^*$ is surjective, or $p\cdot\varGamma^t$ is
split injective. These two conditions are equivalent by
Lemma~\ref{2ses}.
\end{proof}

\begin{corollary}
Let $P=P(A,\mb b)$ be a Delzant polytope,
$\varGamma=(\gamma_1,\ldots,\gamma_m)$ the Gale dual
configuration, and $\zp$ the corresponding moment-angle manifold.
Then
\begin{itemize}
\item[(a)]
$\delta=\varGamma\mb b$ is a regular value of the moment map
$\mu_\varGamma\colon\C^m\to\mathfrak t^*_\varGamma$ for the
Hamiltonian action of $T_\varGamma\subset\T^m$ on~$\C^m$;
\item[(b)] $\zp$ is the regular level set $\mu_\varGamma^{-1}(\varGamma\mb
b)$;
\item[(c)] the action of $T_\varGamma$ on $\zp$ is free.
\end{itemize}
\end{corollary}

We therefore may consider the symplectic quotient of $\C^m$
by~$T_\varGamma$. It is a compact $2n$-dimensional symplectic
manifold, which we denote $V_P=\zp/T_\varGamma$. This manifold has
a `residual' Hamiltonian action of the quotient $n$-torus
$\T^m/T_\varGamma$. It follows from the vertical exact sequence
in~\eqref{2Zses} that $\T^m/T_\varGamma$ can be identified
canonically with $N\otimes_\Z\mathbb S=\R^n/N$, and we shall
denote this torus by~$T_N$. We therefore obtain an exact sequence
of tori
\begin{equation}\label{cgrou}
  1\longrightarrow T_\varGamma\longrightarrow \T^m
  \stackrel{\exp A}\longrightarrow T_N\longrightarrow1,
\end{equation}
where $\exp A\colon\T^m\to T_N$ is the map of toric corresponding
to the map of lattices $A\colon\Z^m\to N$.

The symplectic $2n$-manifold $V_P=\zp/T_\varGamma$ with the
Hamiltonian action of the $n$-torus $T_N=\T^m/T_\varGamma$ is
called the \emph{Hamiltonian toric manifold} corresponding to a
Delzant polytope~$P$.

We denote by $\mu_V\colon V_P\to\mathfrak t_N^*$ the moment map
for the $T_N$-action on~$V_P$, where $\mathfrak t_N=N_\R$ is the
Lie algebra of~$T_N$. The dual Lie algebra $\mathfrak t_N^*$ is
naturally a subspace in~$\R^m$ (the dual Lie algebra of~$\T^m$),
with the inclusion given by $A^t\colon\mathfrak
t_N^*\cong\R^n\to\R^m$.

\begin{proposition}\label{tvmomentmap}
The image of the moment map $\mu_V\colon V_P\to\mathfrak t_N^*$ is
the polytope $P$, up to shifting by a vector in~$\mathfrak t_N^*$.
\end{proposition}
\begin{proof}
Let $\omega$ be the standard symplectic form on $\C^m$ and
$\mu\colon\C^m\to\R^m$ the moment map for the standard action
of~$\T^m$ (see Example~\ref{simcm}). Let $p\colon\zp\to V_P$ be
the quotient projection by the action of~$T_\varGamma$, and let
$i\colon\zp\to\C^m$ be the inclusion, so that the symplectic form
$\omega'$ on $V_P$ satisfies $p^*\omega'=i^*\omega$. Let $H_{\mb
e_i}\colon\C^m\to\R$ be the Hamiltonian of the $\T^m$-action on
$\C^m$ corresponding to the $i$th basis vector $\mb e_i$
(explicitly, $H_{\mb e_i}(\!\mb z)=|z_i|^2$), and let $H_{\mb
a_i}\colon V_P\to\R$ be the Hamiltonian of the $T_N$-action on
$V_P$ corresponding to $\mb a_i\in\mathfrak t$. Denote by $X_{\mb
e_i}$ the vector field on $\zp$ generated by $\mb e_i$, and denote
by $Y_{\mb a_i}$ the vector field on $V_P$ generated by~$\mb a_i$.
Observe that $p_*X_{\mb e_i}=Y_{\mb a_i}$. For any vector field
$Z$ on $\zp$ we have
\begin{multline*}
  dH_{\mb e_i}(Z)=i^*\omega(X_{\mb e_i},Z)=
  p^*\omega'(X_{\mb e_i},Z)\\=\omega'(Y_{\mb a_i},p_*Z)=
  dH_{\mb a_i}(p_*Z)=d(p^*H_{\mb a_i})(Z),
\end{multline*}
hence $H_{\mb e_i}=p^*H_{\mb a_i}$ or $H_{\mb e_i}(\!\mb
z)=H_{\!\mb a_i}(p\,(\!\mb z))$ up to constant. By definition of
the moment map this implies that $\mu_V(V_P)\subset\mathfrak
t^*_N\subset\R^m$ is identified with $\mu(\zp)\subset\R^m$ up to
shift by a vector in~$\R^m$. The inclusion $\mathfrak
t^*_N\subset\R^m$ is the map~$A^t$, and $\mu(\zp)=i_{A,\mb
b}(P)=A^t(P)+\mb b$ by definition of $\zp$, see~\eqref{cdiz}. We
therefore obtain that there exists $\mb c\in\R^m$ such that
\[
  A^t(\mu_V(V_P))+\mb c=A^t(P)+\mb b,
\]
i.e. $A^t(\mu_V(V_P))$ and $A^t(P)$ differ by $\mb b-\mb c\in
A^t(\mathfrak t^*_N)$. Since $A^t$ is monomorphic, the result
follows.
\end{proof}

We have described how to construct a Hamiltonian toric manifold
from a Delzant polytope. A theorem of Delzant~\cite{delz88} says
that \emph{any}  $2n$-dimensional compact connected symplectic
manifold $W$ with an effective Hamiltonian action of an
$n$-torus~$T$ is equivariantly symplectomorphic to a Hamiltonian
toric manifold $V_P$, where $P$ is the image of the moment map
$\mu\colon W\to\mathfrak t^*$ (whence the name `Delzant
polytope').

\begin{example}\label{cp2sq}
Consider the case $m-n=1$, i.e. $T_\varGamma$ is 1-dimensional,
and $\gamma_k\in\R$. By Theorem~\ref{propmmap}~(a), the moment map
$\mu_\varGamma$ is proper whenever each of its level sets
\[
  \mu^{-1}_\varGamma(\delta)=\{\mb z\in\C^m\colon
  \gamma_1|z_1|^2+\cdots+\gamma_m|z_m|^2=\delta\}
\]
is bounded. By Theorem~\ref{propmmap}~(b), $\delta$ is a regular
value whenever the quadratic hypersurface
$\gamma_1|z_1|^2+\cdots+\gamma_m|z_m|^2=\delta$ is nonempty and
nondegenerate. These two conditions together imply that the
hypersurface is an ellipsoid, and the associated polyhedron is an
$n$-simplex (see Example~\ref{mamsimplex}). By Lemma~\ref{afree},
the $T_\varGamma$-action on $\mu^{-1}_\varGamma(\delta)$ is free
if and only if $L_{\mb z}=L$ for any $\mb
z\in\mu^{-1}_\varGamma(\delta)$. This means that each $\gamma_k$
generates the same lattice as the whole set
$\gamma_1,\ldots,\gamma_m$, which implies that
$\gamma_1=\cdots=\gamma_m$. The Gale dual configuration satisfies
$\mb a_1+\cdots+\mb a_m=\mathbf0$. Then $T_\varGamma$ is the
diagonal circle in~$\T^m$, the hypersurface
$\mu^{-1}_\varGamma(\delta)=\zp$ is a sphere, and the associated
polytope~$P$ is a standard simplex up to shift and magnification
by a positive factor~$\delta$. The Hamiltonian toric manifold
$V_P=\zp/T_\varGamma$ is the complex projective space~$\C P^n$.
\end{example}

\section{Fans and toric varieties}
A toric variety is a normal algebraic variety on which an
\emph{algebraic torus}~$(\C^\times)^n$ acts with a dense (Zariski
open) orbit. Toric varieties are described by
combinatorial-geometric objects, rational fans.

A toric variety can be defined from a rational fan using an
algebraic version of symplectic reduction, also known as the `Cox
construction'. Different versions of this construction have
appeared in the work of several authors since the early 1990s. We
mainly follow the work of Cox~\cite{cox95} (and the modernised
version~\cite[Chapter~5]{c-l-s11}) in our exposition;
relationships between toric varieties and moment-angle manifolds
will be explored further in the next sections.

\subsection{Cones and fans}\label{conesfans}
A set of vectors $\mb a_1,\ldots,\mb a_k\in\R^n$ defines a
\emph{convex polyhedral cone}, or simply~\emph{cone},
\[
  \sigma=\R_\ge\langle\mb a_1,\ldots,\mb a_m\rangle=\{\mu_1\mb a_1+\cdots+\mu_k\mb
  a_k\colon\mu_i\in\R_\ge\}.
\]
Here $\mb a_1,\ldots\mb a_k$ are referred to as \emph{generating
vectors} (or \emph{generators}) of~$\sigma$. A \emph{minimal} set
of generators of a cone is defined up to multiplication of vectors
by positive constants. A cone is \emph{rational} if its generators
can be chosen from the integer lattice $\Z^n\subset\R^n$. If
$\sigma$ is a rational cone, then its generators $\mb
a_1,\ldots\mb a_k$ are usually chosen to be \emph{primitive}, i.e.
each $\mb a_i$ is the smallest lattice vector in the ray defined
by it.

A cone is \emph{strongly convex} if it does not contain a line. A
cone is \emph{simplicial} if it is generated by a part of basis
of~$\R^n$, and is \emph{regular} if it is generated by a part of
basis of~$\Z^n$.

Any cone $\sigma$ is an (unbounded) polyhedron, and \emph{faces}
of $\sigma$ are defined as its intersections with supporting
hyperplanes. Each face of a cone is a cone. If a cone is strongly
convex, then it has a unique vertex~$\mathbf 0$; otherwise there
are no vertices. A minimal generator set of a cone consists of
nonzero vectors along its edges.

A \emph{fan} is a finite collection
$\Sigma=\{\sigma_1,\ldots,\sigma_s\}$ of strongly convex cones in
some $\R^n$ such that every face of a cone in $\Sigma$ belongs to
$\Sigma$ and the intersection of any two cones in $\Sigma$ is a
face of each. A fan $\Sigma$ is \emph{rational} (respectively,
\emph{simplicial}, \emph{regular}) if every cone in $\Sigma$ is
rational (respectively, simplicial, regular). A fan
$\Sigma=\{\sigma_1,\ldots,\sigma_s\}$ is called \emph{complete} if
$\sigma_1\cup\cdots\cup\sigma_s=\R^n$.

Cones in a fan can be separated by hyperplanes:

\begin{lemma}[Separation Lemma]\label{seplemma}
Let $\sigma$ and $\sigma'$ be two cones whose intersection $\tau$
is a face of each. Then there exists a common supporting
hyperplane $H$ for $\sigma$ and $\sigma'$ such that
\[
  \tau=\sigma\cap H=\sigma'\cap H.
\]
\end{lemma}

For the proof, see e.g.~\cite[\S1.2]{fult93}. Miraculously, the
convex-geometrical separation property above will translate into
topological separation (Hausdorffness) of algebraic varieties and
topological spaces constructed from fans as described below.

Given a simplicial fan $\Sigma$ with $m$ edges generated by
vectors $\mb a_1,\ldots,\mb a_m$, define its \emph{underlying
simplicial complex} $\sK_\Sigma$ on $[m]=\{1,\ldots,m\}$ as the
collection of subsets $I\subset[m]$ such that $\{\mb a_i\colon
i\in I\}$ spans a cone of~$\Sigma$.

A simplicial fan $\Sigma$ in $\R^n$ is therefore determined by two
pieces of data:
\begin{itemize}
\item[--] a simplicial complex $\sK$ on $[m]$;
\item[--] a configuration of vectors $\mb a_1,\ldots,\mb a_m$ in
$\R^n$ such that the subset $\{\mb a_i\colon i\in I\}$ is linearly
independent for any simplex $I\in\sK$.
\end{itemize}
Then for each $I\in\sK$ we can define the simplicial cone
$\sigma_I$ spanned by $\mb a_i$ with $i\in I$. The `bunch of
cones' $\{\sigma_I\colon I\in\sK\}$ patches into a fan $\Sigma$
whenever any two cones $\sigma_I$ and $\sigma_J$ intersect in a
common face (which has to be $\sigma_{I\cap J}$). Equivalently,
the relative interiors of cones $\sigma_I$ are pairwise
non-intersecting. Under this condition, we say that the data
$\{\sK;\mb a_1,\ldots,\mb a_m\}$ \emph{define a fan}~$\Sigma$.

\medskip

The next construction assigns a complete fan to every convex
polytope.

\begin{construction}[Normal fan]\label{nf}
Let $P$ be a polytope~\eqref{ptope} with $m$ facets
$F_1,\ldots,F_m$ and normal vectors $\mb a_1,\ldots,\mb a_m$.
Given a face $Q\subset P$, we say that a vector $\mb a_i$ is
\emph{normal} to~$Q$ if $Q\subset F_i$. Define the \emph{normal}
cone $\sigma_Q$ as the cone generated by those $\mb a_i$ which are
normal to~$Q$. It can be given by
\[
  \sigma_Q=\{\mb u\in\R^n\colon
  \langle\mb u,\mb x'\rangle\le\langle\mb u,\mb x\rangle
  \text{ for all $\mb x'\in Q$ and $\mb x\in P$}\}.
\]
Then
\[
  \Sigma_P=\{\sigma_Q\colon Q\text{ is a face of }P\}\cup\{\mathbf0\}
\]
is a complete fan which is referred to as the \emph{normal fan} of
the polytope~$P$. If $\mathbf 0$ is contained in the interior of
$P$ then $\Sigma_P$ may be also described as the set of cones over
the faces of the polar polytope~$P^*$.

The normal fan $\Sigma_P$ is simplicial if and only if $P$ is
simple. In this case the cones of $\Sigma_P$ are generated by
those sets $\{\mb a_{i_1},\ldots,\mb a_{i_k}\}$ for which the
intersection $F_{i_1}\cap\cdots\cap F_{i_k}$ is nonempty. The
underlying simplicial complex $\sK_{\Sigma_P}$ is geometrically
the boundary of the polar simplicial polytope~$P^*$.
\end{construction}

The normal fan $\Sigma_P$ of a polytope $P$ contains the
information about the normals to the facets (the generators $\mb
a_i$ of the edges of~$\Sigma_P$) and the combinatorial structure
of $P$ (which sets of vectors $\mb a_i$ span a cone of $\Sigma_P$
is determined by which facets intersect at a face), however the
scalars $b_i$ in~\eqref{ptope} are lost. Not any complete fan can
be obtained by `forgetting the numbers $b_i$' from a presentation
of a polytope by inequalities, i.e. not any complete fan is a
normal fan. This is fails even for regular fans in~$\R^3$,
see~\cite[\S1.5]{fult93} for an example. Furthermore, complete
simplicial fans and simplicial polytopes differ even as
combinatorial objects: there are complete simplicial fans $\Sigma$
whose underlying simplicial complex $\sK_\Sigma$ cannot be
obtained as the boundary of any simplicial polytope (although no
regular examples of this sort are known).

\subsection{Toric varieties}
An \emph{algebraic torus} is a commutative complex algebraic group
isomorphic to a product $(\C^\times)^n$ of copies of the
multiplicative group $\C^\times=\C\setminus\{0\}$. It contains a
compact torus $T^n$ as a Lie (but not algebraic) subgroup.

We shall often identify an algebraic torus with the standard
model~$(\C^\times)^n$.

A \emph{toric variety} is a normal complex algebraic variety~$V$
containing an algebraic torus $(\C^\times)^n$ as a Zariski open
subset in such a way that the natural action of $(\C^\times)^n$ on
itself extends to an action on~$V$.

It follows that $(\C^\times)^n$ acts on $V$ with a dense orbit.

Algebraic geometry of toric varieties is translated completely
into the language of combinatorial and convex geometry. Namely,
there is a bijective correspondence between rational fans in an
$n$-dimensional space and complex $n$-dimensional toric varieties.
Under this correspondence,
\begin{align*}
\text{cones }&\longleftrightarrow\text{ affine varieties}\\
\text{complete fans }&\longleftrightarrow\text{ compact (complete) varieties}\\
\text{normal fans of polytopes}&\longleftrightarrow\text{ projective varieties}\\
\text{regular fans }&\longleftrightarrow\text{ nonsingular varieties}\\
\text{simplicial fans }&\longleftrightarrow\text{ orbifolds}
\end{align*}
The details of this classical correspondence can be found in any
standard source on toric geometry,
e.g.~\cite{dani78},~\cite{fult93} or~\cite{c-l-s11}. Along with
the classical construction, there is an alternative way to define
a toric variety: as the quotient of an open subset in $\C^m$ (the
complement of a coordinate subspace arrangement) by an action of a
commutative algebraic group (a product of an algebraic torus and a
finite group).

\subsection{Quotients in algebraic geometry} Taking quotients of algebraic varieties by algebraic
group actions is tricky for both topological and algebraic
reasons. First, as algebraic groups are often not compact (as
algebraic tori), their orbits may be not closed, and the quotients
may be non-Hausdorff. Second, even if the quotient is Hausdorff as
a topological space, it may fail to be an algebraic variety. This
may be remedied to some extent by the notion of the categorical
quotient.

Let $X$ be an algebraic variety with an action of an affine
algebraic group~$G$. An algebraic variety $Y$ is said to be a
\emph{categorical quotient} of $X$ by the action of~$G$ if there
exists a morphism $\pi\colon X\to Y$ which is constant on
$G$-orbits of~$X$ and has the following universal property: for
any morphism $\varphi\colon X\to Z$ which is constant on
$G$-orbits, there is a unique morphism $\widehat\varphi\colon Y\to
Z$ such that $\widehat\varphi\circ\pi=\varphi$. This is described
by the diagram
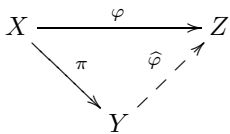
\[
\xymatrix{
  X \ar[rr]^{\varphi} \ar[dr]^{\pi} && Z\\
  & Y \ar@{-->}[ur]^{\widehat\varphi}
}
\]
A categorical quotient $Y$ is unique up to isomorphism, and we
shall denote it by $X\dbs G$ (although sometimes this notation is
reserved for categorical quotients with extra good properties).

Assume that $X=\Spec A$ is an affine variety, where $A=\C[X]$ is
the algebra of regular functions on~$X$, and $G$ is an algebraic
torus (in fact, the construction works for any \emph{reductive}
affine algebraic group). Then the subalgebra $\C[X]^{G}$ of
$G$-invariant functions (i.e. functions $f$ satisfying
$f(gx)=f(x)$ for any $g\in G$ and $x\in X$) is finitely generated,
and the corresponding affine variety $\Spec \C[X]^{G}$ is the
categorical quotient $X\dbs G$. The quotient morphism $\pi\colon
X\to X\dbs G$ is dual to the inclusion of algebras $\C[X]^{G}\to
\C[X]$. The morphism $\pi$ is surjective and induces a one-to-one
correspondence between points of $X\dbs G$ and \emph{closed}
$G$-orbits of~$X$ (i.e. $\pi^{-1}(x)$ contains a unique closed
$G$-orbit for any $x\in X\dbs G$,
see~\cite[Proposition~5.0.7]{c-l-s11}).

Therefore, if all $G$-orbits of an affine variety $X$ are closed,
then the categorical quotient $X\dbs G$ is identified as a
topological space with the ordinary `topological' quotient $X/G$.
In algebraic geometry quotients of this type are called
\emph{geometric} and also denoted by~$X/G$.

\begin{example}
Let $\C^\times$ act on $\C=\Spec(\C[z])$ by scalar multiplication.
There are two orbits: the closed orbit $0$ and the open orbit
$\C^\times$. The topological quotient $\C/\C^\times$ consists of
two points, one of which is not closed, so the space is not
Hausdorff.

On the other hand, the categorical quotient
$\C\dbs\C^\times=\Spec(\C[z]^{\C^\times})$ is a point, since any
$\C^\times$-invariant polynomial is constant (and there is only
one closed orbit).

Similarly, if $\C^\times$ acts on $\C^n=\Spec(\C[z_1,\ldots,z_n])$
diagonally, then an invariant polynomial satisfies $f(\lambda
z_1,\ldots,\lambda z_n)=f(z_1,\ldots,z_n)$ for all
$\lambda\in\C^\times$. Such polynomial must be constant, so that
$\C^n\dbs\C^{\times}$ is again a point.
\end{example}

In good cases categorical quotients of general (non-affine)
varieties $X$ may be constructed by `gluing from pieces' as
follows. Assume that $G$ acts on $X$ and $\pi\colon X\to Y$ is a
morphism of varieties that is constant on $G$-orbits. If $Y$ has
an open affine cover $Y=\bigcup_\alpha V_\alpha$ such that
$\pi^{-1}(V_\alpha)$ is affine and $V_\alpha$ is the categorical
quotient (that is,
$\pi|_{\pi^{-1}(V_\alpha)}\colon\pi^{-1}(V_\alpha)\to V_\alpha$ is
the morphism dual to the inclusion of algebras
$\C[\pi^{-1}(V_\alpha)]^G\to\C[\pi^{-1}(V_\alpha)]$), then $Y$ is
the categorical quotient $X\dbs G$.

\begin{example}
Let $\C^\times$ act on $\C^2\setminus\{0\}$ diagonally, where
$\C^2=\Spec(\C[z_0,z_1])$. We have an open affine cover
$\C^2\setminus\{0\}=U_0\cup U_1$, where
\begin{align*}
  U_0&=\C^2\setminus\{z_0=0\}=\C^\times\times\C=\Spec(\C[z_0^{\pm1},z_1]),\\
  U_1&=\C^2\setminus\{z_1=0\}=\C\times\C^\times=\Spec(\C[z_0,z_1^{\pm1}]),\\
  U_0\cap
  U_1&=\C^2\setminus\{z_0z_1=0\}=\C^\times\times\C^\times=\Spec(\C[z_0^{\pm1},z_1^{\pm1}]).
\end{align*}
The algebras of $\C^\times$-invariant functions are
\[
  \C[z_0^{\pm1},z_1]^{\C^\times}\!\!=\C[z_1/z_0],\quad
  \C[z_0,z_1^{\pm1}]^{\C^\times}\!\!=\C[z_0/z_1],\quad
  \C[z_0^{\pm1},z_1^{\pm1}]^{\C^\times}\!\!=\C[(z_1/z_0)^{\pm1}].
\]
It follows that $V_i=U_i\dbs\C^\times=\C$ glue together along
$V_0\cap V_1=(U_0\cap U_1)\dbs\C^\times=\C^\times$ in the standard
way to produce~$\C P^1$. We have that all $\C^\times$-orbits are
closed in $\C^2\setminus\{0\}$, hence $\C
P^1=(\C^2\setminus\{0\})/\C^\times$ is the geometric quotient.

Similarly, $\C P^n=(\C^{n+1}\setminus\{0\})/\C^\times$ is the
geometric quotient for the diagonal action of~$\C^\times$.
\end{example}

\begin{example}
Now we let $\C^\times$ act on $\C^2\setminus\{0\}$ by
$\lambda\cdot(z_0,z_1)=(\lambda z_0,\lambda^{-1}z_1)$. Using the
same affine cover of $\C^2\setminus\{0\}$ as in the previous
example, we obtain the following algebras of $\C^\times$-invariant
functions:
\[
  \C[z_0^{\pm1},z_1]^{\C^\times}\!\!=\C[z_0z_1],\quad
  \C[z_0,z_1^{\pm1}]^{\C^\times}\!\!=\C[z_0z_1],\quad
  \C[z_0^{\pm1},z_1^{\pm1}]^{\C^\times}\!\!=\C[(z_0z_1)^{\pm1}].
\]
This times gluing together $V_i=U_i\dbs\C^\times=\C$ along
$V_0\cap V_1=(U_0\cap U_1)\dbs\C^\times=\C^\times$ gives the space
obtained from two copies of $\C$ by identifying all nonzero
points. This space is not Hausdorff (the two zeros do not have
nonintersecting neighbourhoods in the usual topology), and
therefore it cannot be a categorical quotient, because algebraic
varieties are Hausdorff spaces in the usual topology.
\end{example}

A toric variety $V_\Sigma$ will be described as the categorical
(or in good cases, geometric) quotient of the `total space'
$U(\Sigma)$ by an action of a commutative algebraic group~$G$. We
now proceed to describe $G$ and~$U(\Sigma)$.

\subsection{Quotient construction of toric varieties}\label{quotv}
Following the algebraic tradition, we use the coordinate-free
notation here. We fix a lattice $N$ of rank~$n$, and denote by
$N_\R$ its ambient $n$-dimensional real vector space
$N\otimes_\Z\R\cong\R^n$. We also define the algebraic torus
$\C^\times_N=N\otimes_\Z\C^\times\cong(\C^\times)^n$.

Let $\Sigma$ be a rational fan in $N_\R$ with $m$ edges generated
by primitive vectors $\mb a_1,\ldots,\mb a_m$ of~$N$. We shall
assume that the linear span of $\mb a_1,\ldots,\mb a_m$ is the
whole~$N_\R$.

We consider the map of lattices $A\colon\Z^m\to N$ sending the
$i$th basis vector of $\Z^m$ to $\mb a_i\in N$. The corresponding
map of algebraic tori,
\[
  A\otimes_\Z\C^\times\colon(\C^\times)^m\to\C^\times_N
\]
is surjective. We shall denote this map by $\exp A$.

Define the group $G=G_\Sigma$ as the kernel of the map $\exp A$.
We therefore have an exact sequence of abelian algebraic groups
\begin{equation}\label{ggrou}
  1\longrightarrow G\longrightarrow (\C^\times)^m
  \stackrel{\exp A}\longrightarrow \C^\times_N\longrightarrow1.
\end{equation}
Explicitly, $G$ is given by
\begin{equation}\label{gexpl}
  G=\bigl\{(z_1,\ldots,z_m)\in(\C^\times)^m\colon
  \prod_{i=1}^m z_i^{\langle\mb a_i,\mb u\rangle}=1
  \quad\text{for all }\mb u\in N^*\bigr\}.
\end{equation}
The group $G$ is isomorphic to a product of $(\C^\times)^{m-n}$
and a finite abelian group. If $\Sigma$ is a regular fan with at
least one $n$-dimensional cone, then $G\cong(\C^\times)^{m-n}$.

\medskip

Given a cone $\sigma\in\Sigma$, set
$g(\sigma)=\{i_1,\ldots,i_k\}\subset[m]$ if $\sigma$ is spanned by
$\mb a_{i_1}\ldots,\mb a_{i_k}$. We define the simplicial complex
$\sK_\Sigma$ generated by all subsets $g(\sigma)\subset[m]$:
\[
  \sK_\Sigma=\{I\colon I\subset g(\sigma)\quad\text{for some
  }\sigma\in\Sigma\}.
\]
If $\Sigma$ is a simplicial fan, then each $I\subset g(\sigma)$ is
$g(\tau)$ for some $\tau\in\Sigma$, and we obtain the `underlying
complex' of~$\Sigma$ defined in the beginning of this section. If
$\Sigma$ is the normal fan of a non-simple polytope~$P$ (i.e. the
fan over the faces of the polar polytope~$P^*$), then $\sK_\Sigma$
is obtained by replacing each face of $\partial P^*$ by a simplex
with the same set of vertices.

Now we define the $U(\Sigma)$ as the complement of an arrangement
of coordinate subspaces in $\C^m$ determined by~$\sK_\Sigma$:
\begin{equation}\label{Usigma}
  U(\Sigma)=\C^m\big\backslash\bigcup_{\{i_1,\ldots,i_k\}\notin\sK_\Sigma}\bigl\{\mb z
  \in\C^m\colon z_{i_1}=\cdots=z_{i_k}=0\bigr\}.
\end{equation}

We observe that the subset $U(\Sigma)\subset\C^m$ depends only on
the combinatorial structure of the fan~$\Sigma$, while the
subgroup $G\subset(\C^\times)^m$ depends on the geometric data,
namely, the primitive generators of one-dimensional cones.

Since $U(\Sigma)\subset\C^m$ is invariant under the coordinatewise
action of $(\C^\times)^m$, we obtain a $G$-action on $U(\Sigma)$
by restriction.

\begin{theorem}[{Cox~\cite[Theorem~2.1]{cox95}}]\label{coxth}
Assume that the linear span of one-dimensional cones of $\Sigma$
is the whole space~$N_\R$.\\
{\rm(a)} The toric variety $V_\Sigma$ is naturally isomorphic
to the categorical quotient $U(\Sigma)\dbs G$.\\
{\rm(b)} $V_\Sigma$ is the geometric quotient $U(\Sigma)/G$ if and
only if the fan $\Sigma$ is simplicial.
\end{theorem}

The torus acting on $V_\Sigma=U(\Sigma)\dbs G$ is the quotient
torus $\C^\times_N=(\C^\times)^m/G$.

{\samepage
\begin{proposition}\label{freeaction}\
\begin{itemize}
\item[(a)] If $\Sigma$ is a simplicial fan, then the $G$-action on $U(\Sigma)$
is almost free;

\item[(b)] If $\Sigma$ is regular, then the $G$-action on $U(\Sigma)$ is free.
\end{itemize}
\end{proposition}
}
\begin{proof}
The stabiliser of a point $\mb z\in\C^m$ under the action of
$(\C^\times)^m$ is
\[
  (\C^\times)^{\omega(\mb z)}=\{(t_1,\ldots,t_m)\in(\C^\times)^m\colon
  t_i=1\text{ if }z_i\ne0\},
\]
where $\omega(\mb z)$ be the set of zero coordinates of~$\mb z$.
The stabiliser of $\mb z$ under the $G$-action is $G_{\mb
z}=(\C^\times)^{\omega(\mb z)}\cap G$. Since $G$ is the kernel of
the map $\exp A\colon(\C^\times)^m\to \C^\times_N$ induced by the
map of lattices $\Z^m\to N$, the subgroup $G_{\mb z}$ is the
kernel of the composite map
\begin{equation}\label{finitegroup}
  (\C^\times)^{\omega(\mb z)}\hookrightarrow(\C^\times)^m
  \stackrel{\exp A}\longrightarrow\C^\times_N.
\end{equation}
This homomorphism of tori is induced by the map of lattices
$\Z^{\omega(\mb z)}\to\Z^m\to N$, where $\Z^{\omega(\mb
z)}\to\Z^m$ is the inclusion of a coordinate sublattice.

Now let $\Sigma$ be a simplicial fan and $\mb z\in U(\Sigma)$.
Then $\omega(\mb z)=g(\sigma)$ for a cone $\sigma\in\Sigma$.
Therefore, the set of primitive generators $\{\mb a_i\colon
i\in\omega(\mb z)\}$ is linearly independent. Hence, the map
$\Z^{\omega(\mb z)}\to\Z^m\to N$ taking $\mb e_i$ to $\mb a_i$ is
a monomorphism, which implies that the kernel
of~\eqref{finitegroup} is a finite group.

If the fan $\Sigma$ is regular, then $\{\mb a_i\colon
i\in\omega(\mb z)\}$ is a part of basis of~$N$. In this
case~\eqref{finitegroup} is a monomorphism and $G_{\mb z}=\{1\}$.
\end{proof}

The relationship between the algebraic quotient construction of
$V_\Sigma$ and the symplectic reduction construction of~$V_P$
(described in the previous section), is as follows. Let $P$ be a
Delzant polytope given by~\eqref{ptope}. Then the Delzant
condition means exactly that the normal fan $\Sigma_P$ is regular.
The tori in the exact sequence~\eqref{cgrou} are maximal compact
subgroups in the algebraic tori of~\eqref{ggrou}. Also, it follows
from Proposition~\ref{easyzp} that the level set
$\mu_\varGamma^{-1}(\varGamma\mb b)$ (the moment-angle
manifold~$\zp$) is contained in $U(\Sigma_P)$.

\begin{theorem}\label{symplectictoric}
Let $P$ be a Delzant polytope with the normal fan $\Sigma_P$. Let
$V_P$ be the corresponding Hamiltonian toric manifold, and
$V_{\Sigma_P}$ the corresponding nonsingular projective toric
variety. The inclusion $\zp\subset U(\Sigma_P)$ induces a
diffeomorphism
\[
  V_P=\zp/T_\varGamma
  \stackrel{\cong}\longrightarrow U(\Sigma_P)/G=V_{\Sigma_P}.
\]
Therefore, any nonsingular projective toric variety can be
obtained as the symplectic quotient of~$\C^m$ by an action of an
$(m-n)$-torus.
\end{theorem}

A proof can be found in~\cite[Proposition~VI.3.1.1]{audi91} or
in~\cite[Appendix~2]{guil94}; we shall also give a proof of a more
general statement in Section~\ref{camam}.

\begin{remark}
Projective embeddings of $V_{\Sigma_P}$ correspond to
\emph{lattice} Delzant polytopes~$P$, i.e. Delzant polytopes with
vertices in the lattice~$N$. Any such embedding defines a
symplectic structure on $V_{\Sigma_P}$ by inducing the symplectic
form from the projective space. It can be
shown~\cite[Appendix~2]{guil94} that the diffeomorphism of
Theorem~\ref{symplectictoric} above preserves the cohomology class
of the symplectic form, or equivalently, the two symplectic
structures are $T_N$-equivariantly symplectomorphic.
\end{remark}

\begin{example}
Let $V_\sigma$ be the affine toric variety corresponding to an
$n$-dimensional simplicial cone~$\sigma$. We may write
$V_\sigma=V_\Sigma$ where $\Sigma$ is the simplicial fan
consisting of all faces of~$\sigma$. Then $m=n$, $U(\Sigma)=\C^n$,
and $A\colon\Z^n\to N$ is the monomorphism onto the full rank
sublattice generated by $\mb a_1,\ldots,\mb a_n$. Therefore, $G$
is a finite group and $V_\sigma=\C^n/G=\Spec\C[z_1,\ldots,z_n]^G$.

In particular, if we consider the cone $\sigma$ generated by $2\mb
e_1-\mb e_2$ and $\mb e_2$ in $\R^2$, then $G$ is $\Z_2$ embedded
as $\{(1,1),(-1,-1)\}$ in~$(\C^\times)^2$. The quotient
construction realises the quadratic cone
\[
  V_\sigma=\Spec\C[z_1,z_2]^G=\Spec\C[z_1^2,z_1z_2,z_2^2]=
  \{(u,v,w)\in\C^3\colon v^2=uw\}
\]
as a quotient of $\C^2$ by~$\Z_2$.
\end{example}

\begin{example}
Let $\Sigma$ be the complete fan in $\R^2$ with the three maximal
cones: $\sigma_0=\R_\ge(\mb e_1,\mb e_2)$, $\sigma_1=\R_\ge(\mb
e_2,-\mb e_1-\mb e_2)$, and $\sigma_2=\R_\ge(-\mb e_1-\mb e_2,\mb
e_1)$. Then $\sK_\Sigma$ is the boundary of a triangle, so the
only non-simplex is~$\{1,2,3\}$. Hence,
\[
  U(\Sigma)=\C^3\setminus\{z_1=z_2=z_3=0\}=\C^3\setminus\{\mathbf0\}
\]
The subgroup $G$ defined by~\eqref{gexpl} is the diagonal
$\C^\times$ in $(\C^\times)^3$. We therefore obtain
$V_\Sigma=U(\Sigma)/G=\C P^2$. Since $\Sigma$ is the normal fan of
the standard $2$-simplex, this agrees with the symplectic quotient
$V_P=\zp/T_\varGamma$ of Example~\ref{cp2sq}.
\end{example}

\begin{example}\label{C3minus3lines1}
Consider the fan $\Sigma$ in~$\R^2$ with three one-dimensional
cones generated by the vectors $\mb e_1$, $\mb e_2$ and $-\mb
e_1-\mb e_2$. This fan is not complete, but its one-dimensional
cones span $\R^2$, so we may apply Theorem~\ref{coxth}. The
simplicial complex $\mathcal K_\Sigma$ consists of 3 disjoint
points. The space $U(\Sigma)$ is the complement to 3 coordinate
lines in~$\C^3$:
\[
  U(\Sigma)=\C^3\big\backslash\bigl(\{z_1=z_2=0\}\cup
  \{z_1=z_3=0\}\cup\{z_2=z_3=0\}\bigr)
\]
The group $G$ is the diagonal $\C^\times$ in $(\C^\times)^3$.
Hence $V_\Sigma=U(\Sigma)/G$ is a quasiprojective variety obtained
by removing three points from~$\C P^2$.
\end{example}

\section{Moment-angle complexes and polyhedral
products}\label{macpp} For any simple polytope $P=P(A,\mb b)$
given by~\eqref{ptope}, we defined the moment-angle manifold
$\zp=\mathcal Z_{A,\mb b}$ from diagram~\eqref{cdiz}, or,
equivalently, as the intersection of quadrics given
by~\eqref{zpqua}. Here, using a combinatorial decomposition of $P$
into cubes, we represent $\zp$ as a union of products
$(D^2)^I\times(S^1)^{[m]\setminus I}$ of discs and circles
parametrised by simplices $I$ in the associated simplicial complex
$\sK_P=\partial P^*$. This construction may be generalised to
arbitrary simplicial complexes~$\sK$, leading to the notion of a
\emph{moment-angle complex}~$\zk$. We follow~\cite{bu-pa00} (and
more detailed treatment given in~\cite{bu-pa02}) in our
description of moment-angle complexes.

The basic building block in the `moment-angle' decomposition of
$\zk$ is the pair $(D^2,S^1)$ of a unit disc and circle, and the
whole construction can be extended naturally to arbitrary pairs of
spaces $(X,A)$. The resulting complex $(X,A)^\sK$ is now known as
the `polyhedral product space' over a simplicial complex~$\sK$;
this terminology was suggested by William Browder,
cf.~\cite{b-b-c-g10}. Many spaces important for toric topology
admit polyhedral product decompositions.

The construction of the moment-angle complex $\zk$ and its
generalisation $(X,A)^\sK$ is of truly universal nature, and has
remarkable functorial properties. The most basic of these is that
the construction of $\zk$ establishes a functor from simplicial
complexes and simplicial maps to spaces with torus actions and
equivariant maps. If $\sK$ is a triangulated sphere, then $\zk$ is
a manifold, and most important geometric examples of $\zk$ arise
in this way.

Another important aspect of the theory of moment-angle complexes
is their connection to coordinate subspace arrangements and their
complements. These have appeared as the `total spaces' $U(\Sigma)$
in the algebraic quotient construction of toric varieties,
reviewed in the previous section. Subspace arrangements and their
complements have also played an important role in singularity
theory, and, more recently, in the theory of linkages and robotic
motion planning. Arrangements of coordinate subspaces in $\C^m$
correspond bijectively to simplicial complexes $\mathcal K$ on the
set~$[m]$, and the complement of such an arrangement is homotopy
equivalent to the corresponding moment-angle complex~$\zk$
(see~\cite[Theorem~5.2.5]{bu-pa00} and Theorem~\ref{deret} below).

\subsection{Cubical decompositions}
\begin{construction}[cubical subdivision of a simple polytope]\label{cubpol}
Let $P$ be a simple $n$-polytope with $m$ facets $F_1,\ldots,F_m$.
We shall construct of a piecewise linear embedding of $P$ into the
standard unit cube~$\I^m\subset\R^m_\ge$, thereby inducing a
cubical subdivision $\mathcal C(P)$ of $P$ by the preimages of
faces of~$\I^m$.

Denote by $\mathcal S$ the set of barycentres of all faces of~$P$,
including the vertices and the barycentre of the whole polytope.
This will be the vertex set of~$\mathcal C(P)$. Every $(n-k)$-face
$G$ of $P$ is an intersection of $k$ facets:
$G=F_{i_1}\cap\cdots\cap F_{i_k}$. We map the barycentre of $G$ to
the vertex $(\varepsilon_1,\ldots,\varepsilon_m)\in \I^m$, where
$\varepsilon_i=0$ if $i\in\{i_1,\ldots,i_k\}$ and
$\varepsilon_i=1$ otherwise. The resulting map $\mathcal S\to\I^m$
can be extended linearly on the simplices of the barycentric
subdivision of $P$ to an embedding $c_P\colon P\to\I^m$. The case
$n=2$, $m=3$ is shown in Fig.~\ref{ip}.
\begin{figure}[h]
\begin{picture}(120,60)
  \put(10,10){\line(1,2){20}}
  \put(30,50){\line(1,-2){20}}
  \put(10,10){\line(1,0){40}}
  \put(20,30){\line(2,-1){10}}
  \put(40,30){\line(-2,-1){10}}
  \put(30,10){\line(0,1){15}}
  \put(50,30){\vector(1,0){12.5}}
  \put(70,10){\line(1,0){30}}
  \put(70,10){\line(0,1){30}}
  \put(70,40){\line(1,0){30}}
  \put(100,10){\line(0,1){30}}
  \put(70,40){\line(1,1){10}}
  \put(100,40){\line(1,1){10}}
  \put(100,10){\line(1,1){10}}
  \put(80,50){\line(1,0){30}}
  \put(110,20){\line(0,1){30}}
  \put(80,20){\line(-1,-1){3.6}}
  \put(75,15){\line(-1,-1){3.6}}
  \multiput(80,20)(3,0){10}{\line(1,0){2}}
  \multiput(80,20)(0,3){10}{\line(0,1){2}}
  \put(81,21){0}
  \put(7,5){$A$}
  \put(28,5){$F$}
  \put(50,5){$E$}
  \put(32,22){$G$}
  \put(15,29){$B$}
  \put(42,29){$D$}
  \put(29,52){$C$}
  \put(15,52){\Large $P$}
  \put(54,32){$c_P$}
  \put(67,5){$A$}
  \put(98,5){$F$}
  \put(67,41){$B$}
  \put(101,37){$G$}
  \put(80,52){$C$}
  \put(111,52){$D$}
  \put(111,18){$E$}
  \put(66,52){\Large $\I^m$}
\end{picture}
\caption{Embedding $c_P\colon P\to \I^m$ for $n=2$, $m=3$.}
\label{ip}
\end{figure}

Any face of $\I^m$ has the form
\[
  C_{J\subset I}=\{(y_1,\ldots,y_m)\in \I^m\colon y_j=0
  \text{ for }j\in
  J,\quad y_j=1\text{ for }j\notin I\}
\]
where $J\subset I$ is a pair of of embedded (possibly empty)
subsets of~$[m]$. We also set
\[
  C_I=C_{\varnothing\subset I}=
  \{(y_1,\ldots,y_m)\in \I^m\colon y_j=1\text{ for }j\notin I\}
\]
to simplify the notation.

The image $c_P(P)\subset\I^m$ is the union of all faces
$C_{J\subset I}$ such that $\bigcap_{i\in I}F_i\ne\varnothing$.
For such $C_{J\subset I}$, the preimage $c_P^{-1}(C_{J\subset I})$
is a face of the cubical complex~$\mathcal C(P)$. The vertex set
of $c_P^{-1}(C_{J\subset I})$ is the subset of $\mathcal S$
consisting of barycentres of all faces between the faces $G$ and
$H$ of~$P$, where $G=\bigcap_{j\in J}F_j$ and $H=\bigcap_{i\in
I}F_i$. Therefore, faces of $\mathcal C(P)$ correspond to pairs of
embedded faces $G\supset H$ of~$P$, and we denote them by
$C_{G\supset H}$. In particular, maximal ($n$-dimensional) faces
of $\mathcal C(P)$ correspond to pairs $G=P$, $H=v$, where $v$ is
a vertex of~$P$. For these maximal faces we use the abbreviated
notation $C_v=C_{P\supset v}$.

For every vertex $v=F_{i_1}\cap\cdots\cap F_{i_n}\in P$ with
$I_v=\{i_1,\ldots,i_n\}$ we have
\begin{equation}
\label{cubpolmap}
  c_P(C_v)=C_{I_v}=\bigl\{(y_1,\ldots,y_m)\in \I^m\colon y_j=1\text{ whenever }
  v\notin F_j\bigr\}.
\end{equation}
\end{construction}

We therefore obtain:

\begin{proposition}
\label{thcubpol} A simple polytope $P$ with $m$ facets admits a
cubical decomposition whose maximal faces $C_v$ correspond to the
vertices $v\in P$. The resulting cubical complex $\mathcal C(P)$
embeds canonically into~$\I^m$, as described
by~{\rm(\ref{cubpolmap})}.
\end{proposition}

\subsection{Moment-angle complexes}
The map $\mu\colon\C^m\to\R^m_\ge$ (see Example~\ref{simcm})
identifies the unit cube~$\I^m\subset\R^m_\ge$ with the quotient
of the unit \emph{polydisc}
\[
  \D^m=\bigl\{ (z_1,\ldots,z_m)\in\C^m\colon |z_i|\le1\bigr\}
\]
by the coordinatewise action of~$\T^m$.

Now we define the space $\widetilde{\mathcal Z}_P$ from a diagram
similar to~\eqref{cdiz} (which was used to define $\zp=\mathcal
Z_{A,\mb b}$), in which the bottom map is replaced by $c_P\colon
P\to\I^m$:
\begin{equation}\label{cdwiz}
\begin{CD}
  \widetilde{\mathcal Z}_P @>\widetilde i_Z>>\D^m\\
  @VVV\hspace{-0.2em} @VV\mu V @.\\
  P @>c_P>> \I^m
\end{CD}
\end{equation}

\begin{proposition}
The space $\widetilde{\mathcal Z}_P$ is $\T^m$-equivariantly
homeomorphic to the moment-angle manifold~$\zp$.
\end{proposition}
\begin{proof}
As we have seen in Proposition~\ref{zpids}, $\zp$ is
$\T^m$-homeomorphic to the identification space
\[
  P\times\T^m/{\sim\:}\quad\text{where }
  (\mb x,\mb t_1)\sim(\mb x,\mb t_2)\:\text{ if }\:\mb t_1^{-1}\mb t_2\in
  \T^{I_{\mb x}}.
\]
By restricting~\eqref{Cmids} to $\D^m\subset\C^m$ we obtain that
\[
  \D^m\cong   \I^m\times\T^m/{\sim\:}\quad\text{where }(\mb y,\mb t_1)\sim(\mb y,\mb
  t_2)\text{ if }\mb t_1^{-1}\mb t_2\in\T^{\omega(\mb y)}.
\]
As in the proof of Proposition~\ref{zpids}, $\widetilde{\mathcal
Z}_P$ is identified with $c_P(P)\times\T^m/{\sim\:}$. A point $\mb
x\in P$ is mapped by $c_P$ to $\mb y\in\I^m$ with $I_{\mb
x}=\omega(\mb y)=\{i\in[m]\colon\mb x\in F_i\}$. We therefore
obtain that both $\zp$ and $\widetilde{\mathcal Z}_P$ are
$\T^m$-homeomorphic to~$P\times\T^m/{\sim\:}$.
\end{proof}

We shall therefore not distinguish between the spaces $\zp$ and
$\widetilde{\mathcal Z}_P$; and think of the maps $i_Z$ and
$\widetilde i_Z$ of diagrams~\eqref{cdiz} and~\eqref{cdwiz} as
different embedding of the same manifold $\zp$ in~$\C^m$ (the
first one is smooth, and the second one is not).

Given a vertex $v=F_{i_1}\cap\cdots\cap F_{i_n}\in P$, we consider
the restriction of the map $\widetilde i_Z\colon\zp\to\D^m$ to the
subset $C_v\times\T^m/{\sim\:}\subset P\times\T^m/{\sim\:}=\zp$:
\begin{multline*}
  \widetilde i_Z(C_v\times\T^m/{\sim\:})=
  c_P(C_v)\times\T^m/{\sim\:}=C_{I_v}\times\T^m/{\sim\:}=\mu^{-1}(C_{I_v})\\
  =\bigl\{(z_1,\ldots,z_m)\in
  \D^m\colon |z_j|^2=1\text{ for }v\notin F_j\bigl\}.
\end{multline*}
Since $P=\bigcup_v C_v$, we obtain that
\[
  \widetilde i_Z(\zp)=\bigcup_v\;\mu^{-1}(C_{I_v}).
\]
Note that $\mu^{-1}(C_{I_v})$ is a product of $|I_v|=n$ discs and
$m-n$ circles. Since
$\mu^{-1}(C_I)\cap\mu^{-1}(C_J)=\mu^{-1}(C_{I\cap J})$ for any
$I,J\subset[m]$, we can rewrite the union above as
\begin{equation}\label{zpunion}
  \widetilde i_Z(\zp)=\bigcup_{I\in\sK_P}\mu^{-1}(C_I),
\end{equation}
where
\[
  \sK_P=\{I=\{i_1,\ldots,i_k\}\subset[m]\colon
  F_{i_1}\cap\cdots\cap F_{i_k}\ne\varnothing\}
\]
is the boundary simplicial complex of the polar polytope~$P^*$.

The decomposition~\eqref{zpunion} of $\zp$ into a union of
products of discs and circles can now be generalised to an
arbitrary simplicial complex:

\begin{definition}
Let $\sK$ be a simplicial complex on the set~$[m]$. We always
assume that $\varnothing\in\sK$. The \emph{moment-angle complex}
corresponding to~$\sK$ is defined as
\begin{equation}\label{zkbj}
  \zk=\bigcup_{I\in\sK}B_I,
\end{equation}
where
\[
  B_I=\mu^{-1}(C_I)=\bigl\{(z_1,\ldots,z_m)\in
  \D^m\colon |z_j|^2=1\text{ for }j\notin I\bigl\},
\]
and the union in~\eqref{zkbj} is understood as the union of
subsets inside the polydisc~$\D^m$. Topologically, each $B_I$ is a
product of $|I|$ discs~$D^2$ and $m-|I|$ circles~$S^1$. We
therefore may rewrite~\eqref{zkbj} as the following decomposition
of $\zk$ into a union of products of discs and circles:
\begin{equation}\label{zkd2s1}
  \zk=\bigcup_{I\in\mathcal K}
  \Bigl(\prod_{i\in I}D^2\times\prod_{i\notin I}S^1 \Bigl),
\end{equation}
From now on we shall denote the space $B_I$ by~$(D^2,S^1)^I$.
\end{definition}

We may rephrase~\eqref{zpunion} by saying that the map $\widetilde
i_Z\colon\zp\to\D^m$ identifies the moment-angle manifold $\zp$
with the moment-angle complex $\mathcal Z_{\sK_P}$ corresponding
to~$\sK_P=\partial P^*$.

A \emph{ghost vertex} of $\sK$ is a one-element subset
$\{i\}\in[m]$ which is not in~$\sK$ (i.e. is not a vertex). Since
facets of a simple polytope~$P$ correspond to vertices of $\sK_P$,
it is natural to add a ghost vertex to $\sK_P$ for each redundant
inequality in a generic presentation~\eqref{ptope}.

\begin{example}\

1. Let $\sK=\Delta^{m-1}$ be the full simplex (a simplicial
complex consisting of all subsets of~$[m]$). Then $\zk=\D^m$.

2. Let $\sK$ be a simplicial complex on $[m]$, and let $\sK^\circ$
be the complex on $[m+1]$ obtained by adding one ghost vertex
$\circ=\{m+1\}$ to~$\sK$. Then in the decomposition~\eqref{zkbj}
for $\mathcal Z_{\sK^\circ}$ each $B_I$ has factor $S^1$ in the
last coordinate, and
\[
  \mathcal Z_{\sK^\circ}=\zk\times S^1.
\]
In the case $\sK=\sK_P$ this agrees with
Proposition~\ref{zpcomb}~(b).

In particular, if $\sK$ is the `empty' simplicial complex on
$[m]$, consisting of the empty simplex $\varnothing$ only, then
$\zk=\mu^{-1}(1,\ldots,1)=\T^m$ is the standard $m$-torus.

For an arbitrary $\sK$ on $[m]$, the moment-angle complex $\zk$
contains the $m$-torus $\T^m$ (corresponding to $\sK=\varnothing$)
and is contained in the polydisc $\D^m$ (corresponding to
$\sK=\Delta^{m-1}$).

3. Let $\sK$ be the complex consisting of two disjoint points.
Then
\[
  \zk=(D^2\times S^1)\cup(S^1\times D^2)=\partial(D^2\times
  D^2)\cong S^3,
\]
the standard decomposition of a 3-sphere into the union of two
solid tori.

4. More generally, if $\sK=\partial\Delta^{m-1}$ (the boundary of
a simplex), then
\begin{align*}
\zk&=(D^2\times\cdots\times D^2\times S^1)\cup
     (D^2\times\cdots\times S^1\times D^2)\cup\cdots\cup
     (S^1\times\cdots\times D^2\times
     D^2)\\
   &=\partial\bigl((D^2)^m\bigr)\cong S^{2m-1}.
\end{align*}

5. Let $\mathcal K=\quad\begin{picture}(5,5)
\put(0,0){\circle*{1}} \put(0,5){\circle*{1}}
\put(5,0){\circle*{1}} \put(5,5){\circle*{1}}
\put(0,0){\line(1,0){5}} \put(0,0){\line(0,1){5}}
\put(5,0){\line(0,1){5}} \put(0,5){\line(1,0){5}}
\put(-2.5,-1){\scriptsize 1} \put(6.3,-1){\scriptsize 3}
\put(-2.5,4){\scriptsize 4} \put(6.3,4){\scriptsize 2}
\end{picture}\quad$,
the boundary of a 4-gon. Then we have four maximal simplices
$\{1,3\}$, $\{2,3\}$, $\{1,4\}$ and $\{2,4\}$, and
\begin{align*}
  \zk&=(D^2\times S^1\times D^2\times S^1)\cup
  (S^1\times D^2\times D^2\times S^1)\\
  &\qquad\cup(D^2\times S^1\times S^1\times D^2)\cup
  (S^1\times D^2\times S^1\times D^2)\\
  &=\bigl((D^2\times S^1)\cup(S^1\times D^2)\bigr)
  \times D^2\times S^1
  \cup\bigl((D^2\times S^1)\cup(S^1\times D^2)\bigr)\times
  S^1\times D^2\\
  &=\bigl((D^2\times S^1)\cup(S^1\times D^2)\bigr)\times
  \bigl((D^2\times S^1)\cup(S^1\times D^2)\bigr)
   \cong S^3\times S^3.
\end{align*}
\end{example}

The last example can be generalised as follows. Recall that the
\emph{join} of simplicial complexes $\sK_1$ and $\sK_2$ be on sets
$\mathcal V_1$ and $\mathcal V_2$ respectively is the simplicial
complex
\[
  \sK_1*\sK_2=\bigl\{  I\subset\mathcal V_1\sqcup\mathcal V_2\colon
  I= I_1\cup I_2,\; I_1\in\sK_1,  I_2\in\sK_2\bigr\}
\]
on the set $\mathcal V_1\sqcup\mathcal V_2$.

\begin{proposition}\label{zkjoin}
We have $\mathcal Z_{\mathcal K_1*\,\mathcal K_2}=\mathcal
Z_{\sK_1}\times\mathcal Z_{\sK_2}$.
\end{proposition}
\begin{proof}
Indeed,
\begin{multline*}
  \mathcal Z_{\sK_1*\,\sK_2}=\bigcup_{I_1\in\sK_1,\,I_2\in\sK_2}
  (D^2,S^1)^{I_1\sqcup I_2}=\bigcup_{I_1\in\sK_1,\,I_2\in\sK_2}
  (D^2,S^1)^{I_1}\times(D^2,S^1)^{I_2}\\
  =\Bigl(\bigcup_{I_1\in\sK_1}(D^2,S^1)^{I_1}\Bigr)\times
  \Bigl(\bigcup_{I_2\in\sK_2}(D^2,S^1)^{I_2}\Bigr)=
  \mathcal Z_{\sK_1}\times\mathcal Z_{\sK_2}.\qedhere
\end{multline*}
\end{proof}

\begin{corollary}
Let $P$ and $Q$ be two simple polytopes. Then $\mathcal Z_{P\times
Q}\cong\zp\times\mathcal Z_{Q}$.
\end{corollary}
\begin{proof}
Indeed, $\sK_{P\times Q}=\sK_P*\,\sK_Q$.
\end{proof}

Since $\mathcal Z_{\sK_P}\cong\zp$, the moment-angle complex
corresponding to the boundary of a simplicial polytope is a
manifold. This is also true for the moment-angle manifold complex
corresponding to any triangulated sphere (although not any
triangulation of a sphere is a boundary of a simplicial polytope,
see e.g.~\cite[\S2.3]{bu-pa02}):

\begin{theorem}[{\cite[Lemma~6.13]{bu-pa02}}]
Let $\mathcal K$ be a triangulation of $S^{n-1}$ with $m$
vertices. Then $\zk$ is a (closed) topological manifold of
dimension~$m+n$.
\end{theorem}

As we shall see in the next section, moment-angle complexes
corresponding to complete simplicial fans are smooth manifolds. In
general, it is not known whether a smooth structure exists on
moment-angle manifolds corresponding to arbitrary triangulated
spheres.

The topological structure of moment-angle complexes~$\zk$ is quite
complicated in general. The cohomology ring of~$\zk$ was described
in~\cite[\S4.2]{bu-pa00} (with field coefficients) and
in~\cite{b-b-p04} and~\cite{fran06} (with integer coefficients).
It is known~\cite{gr-th07} that if $\sK$ is the $k$-dimensional
skeleton of the simplex~$\Delta^{m-1}$ (for any~$k,m$), then the
corresponding moment-angle complex~$\zk$ is homotopy equivalent to
a wedge of spheres. Also, it is known that if $P$ is obtained from
a simplex by iteratively truncating vertices by hyperplanes (so
that the polar polytope $P^*$ is \emph{stacked}), then $\zp$ is
diffeomorphic to a connected sum of sphere products, with two
spheres in each product (this result is due to McGavran,
cf.~\cite[Theorem~6.3]{bo-me06}, see also~\cite{gi-lo09}). Finding
more series of polytopes or simplicial complexes for which the
topology of $\zk$ can be described explicitly is a challenging
task. Lots of nontrivial topological phenomena occur already in
the cohomology of~$\zk$. For instance, moment-angle manifolds are
generally not \emph{formal} (in the sense of rational homotopy
theory); first examples of $\zp$ with nontrivial Massey products
in cohomology appear already for 3-dimensional polytopes $P$,
see~\cite{bask03}.

\subsection{Polyhedral products}
Decomposition~\eqref{zkd2s1} of $\zk$ which uses the disc and
circle $(D^2,S^1)$ is readily generalised to arbitrary pairs of
spaces:

\begin{construction}[polyhedral product]\label{nsc}
Let $\sK$ be a simplicial complex on~$[m]$ and let
\[
  (\mb X,\mb A)=\{(X_1,A_1),\ldots,(X_m,A_m)\}
\]
be a set of $m$ pairs of spaces, $A_i\subset X_i$. For each
simplex $I\in\sK$ we set
\begin{equation}\label{XAI}
  (\mb X,\mb A)^I=\bigl\{(x_1,\ldots,x_m)\in
  \prod_{i=1}^m X_i\colon\; x_i\in A_i\quad\text{for }i\notin I\bigl\}
\end{equation}
and define the \emph{polyhedral product} of $(\mb X,\mb A)$
corresponding to $\sK$ by
\[
  (\mb X,\mb A)^\sK=\bigcup_{I\in\mathcal K}(\mb X,\mb A)^I=
  \bigcup_{I\in\mathcal K}
  \Bigl(\prod_{i\in I}X_i\times\prod_{i\notin I}A_i\Bigl).
\]

In the case when all the pairs $(X_i,A_i)$ are the same, i.e.
$X_i=X$ and $A_i=A$ for $i=1,\ldots,m$, we use the notation
$(X,A)^\sK$ for $(\mb X,\mb A)^\sK$.
\end{construction}

\begin{example}\label{exkpo}\

1. The moment-angle complex $\zk$ is the polyhedral product
$(D^2,S^1)^\sK$ (when considered abstractly) or $(\D,\mathbb
S)^\sK$ (when viewed as a subcomplex in~$\D^m$).

2. The cubical subcomplex $c_P(P)\subset\I^m$ of
Construction~\ref{cubpol} is given by
\[
  c_P(P)=(\I,1)^{\sK_P},
\]
where $\I=[0,1]$ is the unit interval and $1$ is its endpoint. For
general~$\sK$, the polyhedral product $(\I,1)^\sK$ is a cubical
subcomplex in~$\I^m$, which can be identified with the quotient of
$\zk$ by the action of~$\T^m$. It is homeomorphic to the cone
over~$\sK$, see~\cite[Proposition~4.10]{bu-pa02}. We denote
$\cc(\sK)=(\I,1)^\sK$.

3. If $\sK$ consists of $m$ disjoint points and
$A_i={\mbox{\textit{pt}}}$ (a point), then
\[
  (\mb X,pt)^\sK=X_1\vee X_2\vee\cdots\vee X_m
\]
is the \emph{wedge} (or \emph{bouquet}) of the $X_i$'s.
\end{example}

\begin{remark}
The decomposition of $\zk$ into a union of products of discs and
circles appeared in~\cite{bu-pa00}, were the term `moment-angle
complex' for $\zk=(D^2,S^1)^\sK$ was also introduced. Several
other examples of polyhedral products $(X,A)^\sK$ (including those
from Example~\ref{exkpo}) were also considered in~\cite{bu-pa00}.
The definition of $(X,A)^\sK$ for an arbitrary pair of spaces
$(X,A)$ was suggested to the authors by N.~Strickland (in a
private communication, and also in an unpublished note) as a
general framework for the constructions of~\cite{bu-pa00}; it was
also included in the final version of~\cite{bu-pa00} and
in~\cite{bu-pa02}. Further generalisations of $(X,A)^\sK$ to a set
of pairs of spaces $(\mb X,\mb A)$ were studied in the work of
Grbi\'c and Theriault~\cite{gr-th07}, as well as Bahri, Bendersky,
Cohen and Gitler~\cite{b-b-c-g10}, where the term `polyhedral
product' was introduced (following a suggestion of W.~Browder).
Since 2000, the terms `generalised moment-angle complex',
`$\sK$-product' and `partial product space' have been also used to
refer to the spaces~$(X,A)^\sK$.
\end{remark}

\subsection{Complements of coordinate subspace arrangements, revisited}
These provide another important class of examples of polyhedral
products. We can define the complement to a set of coordinate
subspaces similar to~\eqref{Usigma} for an arbitrary simplicial
complex~$\sK$:
\begin{equation}\label{uk}
  U(\sK)=\C^m\big\backslash\bigcup_{\{i_1,\ldots,i_k\}\notin\sK}
  \bigl\{\mb z \in\C^m\colon z_{i_1}=\cdots=z_{i_k}=0\bigr\}.
\end{equation}
It is easy to see that the complement to any set of coordinate
subspaces in~$\C^m$ has the form~$U(\sK)$ for some simplicial
complex~$\sK$ on~$[m]$. If the arrangement of coordinate planes
contains a hyperplane $\mb z_i=0$, then $\{i\}$ is a ghost vertex
of the corresponding simplicial complex~$\sK$.

\begin{proposition}
$U(\sK)=(\C,\C^\times)^\sK$.
\end{proposition}
\begin{proof}
Given $I=\{i_1,\ldots,i_k\}$, denote $L_I=\bigl\{\mb z
\in\C^m\colon z_{i_1}=\cdots=z_{i_k}=0\bigr\}$. For $\mb
z=(z_1,\ldots,z_m)\in\C^m$, we denoted $\omega(\mb
z)=\{i\in[m]\colon z_i=0\}\subset[m]$. We have
\begin{multline*}
U(\sK)=\C^m\setminus\bigcup_{I\notin\sK}L_I=
\C^m\setminus\bigcup_{I\notin\sK}\{\mb z\colon\omega(\mb z)\supset
I\}=\C^m\setminus\bigcup_{I\notin\sK}\{\mb z\colon\omega(\mb
z)=I\}\\=\bigcup_{I\in\sK}\{\mb z\colon\omega(\mb z)=I\}=
\bigcup_{I\in\sK}\{\mb z\colon\omega(\mb z)\subset I\}=
\bigcup_{I\in\sK}(\C,\C^\times)^I=(\C,\C^\times)^\sK.\qedhere
\end{multline*}
\end{proof}

Since each coordinate subspace is invariant under the standard
action of $\T^m$ on $\C^m$, the complement $U(\sK)$ is also a
$\T^m$-invariant subset in~$\C^m$.

Recall that a \emph{deformation retraction} of a space $X$ onto a
subspace $A$ is a continuous family of maps (a homotopy)
$F_t\colon X\to X$, $t\in\I$, such that $F_0=\id$ (the identity
map), $F_1(X)=A$ and $F_t|_A=\id$ for all~$t$. Often the term
`deformation retraction' refers only to the last map $f=F_1\colon
X\to A$ in the family. This map is a homotopy equivalence.

\begin{theorem}[{\cite{bu-pa02}}]\label{deret}
The moment-angle complex $\zk$ is a $\T^m$-invariant subspace
of~$U(\sK)$, and there is a $\T^m$-equivariant deformation
retraction
\[
  \zk\hookrightarrow U(\sK)\longrightarrow\zk.
\]
\end{theorem}
\begin{proof}
Since $\D\subset\C$ and $\mathbb S\subset\C^\times$, we have
$\zk=(\D,\mathbb S)^\sK\subset(\C,\C^\times)^\sK=U(\sK)$, and the
subset $\zk\subset U(\sK)$ is obviously $\T^m$-invariant.

Any simplicial complex $\sK$ can be obtained from $\Delta^{m-1}$
by subsequent removal of maximal simplices (so that we get a
simplicial complex at each intermediate step), and we shall
construct the deformation retraction $U(\sK)\to\zk$ by induction.

The base of induction is clear: if $\sK=\Delta^{m-1}$, then
$U(\sK)=\C^m$, $\zk=\D^m$, and the retraction $\C^m\to\D^m$ is
evident.

The orbit space $\zk/\T^m$ is the cubical complex
$\cc(\sK)=(\mathbb I,1)^\sK$ (see Example~\ref{exkpo}.2). The
orbit space $U(\sK)/\T^m$ can be identified with
\[
  U(\sK)_\ge=U(\sK)\cap\R^m_\ge=(\R_\ge,\R_>)^\sK
\]
where $\R^m_\ge$ is viewed as a subset in~$\C^m$.

We shall first construct a deformation retraction $r\colon
U(\sK)_\ge\to\cc(\sK)$ of orbit spaces, and then cover it by a
deformation retraction $\widetilde r\colon U(\sK)\to\zk$.

Now assume that $\sK$ is obtained from a simplicial complex $\sK'$
by removing one maximal simplex $J=\{j_1,\ldots,j_k\}$, i.e.
$\sK\cup J=\sK'$. Then the cubical complex $\cc(\sK')$ is obtained
from $\cc(\sK)$ by adding a single $k$-dimensional face
$C_J=(\mathbb I ,1)^J$. We also have $U(\sK)=U(\sK')\setminus
L_J$, so that
\[
  U(\sK)_\ge=U(\sK')_\ge\setminus\{\mb y\colon y_{j_1}=\cdots=y_{j_k}=0\}.
\]
We may assume by induction that there is a deformation retraction
$r'\colon U(\sK')_\ge\to\cc(\sK')$ such that $\omega(r'(\mb
y))=\omega(\mb y)$, where $\omega(\mb y)$ is the set of zero
coordinates of~$\mb y$. In particular, $r'$ restricts to a
deformation retraction
\[
  r'\colon U(\sK')_\ge\setminus\{\mb y\colon y_{j_1}=\cdots=y_{j_k}=0\}
  \longrightarrow\cc(\sK')\backslash\,\mb y_J
\]
where $\mb y_J$ is the point with coordinates
$y_{j_1}=\cdots=y_{j_k}=0$ and $y_j=1$ for $j\notin J$.

Since $J\notin\sK$, we have $\mb y_J\notin\cc(\sK)$. On the other
hand, $\mb y_J$ belongs to the extra face $C_J=(\mathbb I ,1)^J$
of $\cc(\sK')$. We therefore may apply the deformation retraction
$r_J$ shown in Fig.~\ref{retr} on the face $C_J$, with centre
at~$\mb y_J$.
\begin{figure}[h]
  \begin{picture}(120,45)
  \put(45,5){\circle{2}}
  \put(80,5){\circle*{2}}
  \put(45,40){\circle*{2}}
  \put(80,40){\circle*{2}}
  \put(46,5){\vector(1,0){24}}
  \put(46,5){\line(1,0){34}}
  \put(45.8,5.2){\vector(4,1){24}}
  \put(45.8,5.2){\line(4,1){34}}
  \put(45.5,5.5){\vector(2,1){24}}
  \put(45.5,5.5){\line(2,1){34}}
  \put(45.8,5.8){\vector(4,3){24}}
  \put(45.8,5.8){\line(4,3){34}}
  \put(46,6){\vector(1,1){24}}
  \put(46,6){\line(1,1){34}}
  \put(45,6){\vector(0,1){24}}
  \put(45,6){\line(0,1){34}}
  \put(45.2,5.8){\vector(1,4){6}}
  \put(45.2,5.8){\line(1,4){8.5}}
  \put(45.5,5.5){\vector(1,2){12}}
  \put(45.5,5.5){\line(1,2){17}}
  \put(45.5,5.5){\vector(3,4){18}}
  \put(45.5,5.5){\line(3,4){26}}
  \linethickness{1mm}
  \put(80,5){\line(0,1){35}}
  \put(45,40){\line(1,0){35}}
  \put(39,4){$\mb y_J$}
  \end{picture}
  \caption{Retraction $r_J\colon \cc(\sK')\backslash\,\mb y_J
  \to\cc(\sK)$.}
  \label{retr}
\end{figure}
In coordinates, a homotopy $F_t$ between the identity map
$\cc(\sK')\backslash\,\mb y_J\to\cc(\sK')\backslash\,\mb y_J$ (for
$t=0$) and the retraction $r_J\colon\cc(\sK')\backslash\,\mb
y_J\to\cc(\sK)$ (for $t=1$) is given by
\begin{align*}
  F_t\colon\cc(\sK')\backslash\,\mb y_J&\longrightarrow
  \cc(\sK')\backslash\,\mb y_J,\\
  (y_1,\ldots,y_m,t)&\longmapsto(y_1+t\alpha_1y_1,\ldots,
  y_m+t\alpha_my_m)
\end{align*}
where
\[
  \alpha_i=\begin{cases}
  \frac{1-\max_{j\in J}y_j}{\max_{j\in J}y_j},&
  \text{if $i\in J$,}\\
  0,&\text{if $i\notin J$,}\end{cases}\qquad\text{for }1\le i\le m.
\]
We observe that $\omega(F_t(\mb y))=\omega(\mb y)$ for any $t$ and
$\mb y\in\cc(\sK')$. Now, the composition
\begin{equation}\label{indretr}
  r\colon U(\sK)_\ge=U(\sK')_\ge\!\backslash\,
  \{\mb y\colon y_{j_1}=\cdots=y_{j_k}=0\}
  \stackrel{r'}\longrightarrow
  \cc(\sK')\backslash\,\mb y_J
  \stackrel{r_J}\longrightarrow\cc(\sK)
\end{equation}
is a deformation retraction, and it satisfies $\omega(r(\mb
y))=\omega(\mb y)$ as this is true for $r_J$ and~$r'$. The
inductive step is now complete. The required retraction
$\widetilde r\colon U(\sK)\to\zk$ covers $r$ as shown in the
following commutative diagram:
\[
\xymatrix{
  \zk \ar@{^{(}->}[r]\ar[d]^{\mu} &
  U(\sK) \ar[r]^{\widetilde r}\ar[d]^{\mu} & \zk \ar[d]^{\mu}\\
  \cc(\sK) \ar@{^{(}->}[r] & U_\ge(\sK) \ar[r]^r & \cc(\sK)
}
\]
Explicitly, $\widetilde r$ is decomposed inductively in a way
similar to~\eqref{indretr},
\[
  \widetilde r\colon U(\sK)=
  U(\sK')\backslash\,L_J
  \stackrel{\widetilde r'}\longrightarrow
  \mathcal Z_{\sK'}\backslash\,\mu^{-1}(\mb y_J)
  \stackrel{\widetilde r_J}\longrightarrow\zk,
\]
where $\mu^{-1}(\mb y_J)=\prod_{j\in J}\{0\}\times\prod_{j\notin
J}\mathbb S$, and $\widetilde r_J$ is given in coordinates
$(z_1,\ldots,z_m)=\bigl(\sqrt{y_1}e^{i\varphi_1},\ldots,\sqrt{y_m}e^{i\varphi_m}\bigr)$
by
\[
  \bigl(\sqrt{y_1}e^{i\varphi_1},\ldots,
  \sqrt{y_m}e^{i\varphi_m}\bigr)\mapsto
  \bigl(\sqrt{y_1+\alpha_1y_1}e^{i\varphi_1},\ldots,
  \sqrt{y_m+\alpha_my_m}e^{i\varphi_m}\bigr)
\]
with $\alpha_i$ as above.
\end{proof}

As we shall see in Section~\ref{mamsf}, in the case when
$\sK=\sK_\Sigma$ is the underlying complex of a complete
simplicial fan~$\Sigma$, the deformation retraction $U(\sK)\to\zk$
can be realised as the quotient map for an action of $\R^{m-n}$
on~$U(\sK)$.

In the remaining sections we shall concentrate on the geometric
aspects of the theory of moment-angle complexes, and moment-angle
manifolds corresponding to polytopes and complete simplicial fans
will be our main objects of interest. Nevertheless, the homotopy
theory of general moment-angle complexes has now gained its own
momentum, and we refer to~\cite[Ch.~6]{bu-pa02}, \cite{de-su07},
\cite{gr-th07}, \cite{pa-ra08} and~\cite{b-b-c-g10} for the main
stages of its development.

\section{LVM-manifolds}\label{lvmma}
Bosio and Meersseman~\cite{bo-me06} identified polytopal
moment-angle manifolds $\zp$ with a class of non-K\"ahler
complex-analytic manifolds introduced in the works of Lopez de
Medrano, Verjovsky and Meersseman (LVM-manifolds). This was the
starting point in the subsequent study of the complex geometry of
moment-angle manifolds. We review the construction of
LVM-manifolds and its connection to polytopal moment-angle
manifolds here.

The initial data of the construction of an LVM-manifold is a link
of a homogeneous system of quadrics similar to~\eqref{link}, but
with \emph{complex} coefficients:
\begin{equation}\label{clink}
  \mathcal L=\left\{\begin{array}{lrcl}
  \mb z\in\C^m\colon&\sum_{k=1}^m|z_k|^2&=&1,\\[1mm]
  &\sum_{k=1}^m \zeta_k|z_k|^2&=&\mathbf0
  \end{array}\right\},
\end{equation}
where $\zeta_k\in\C^s$. We can obviously turn this link into the
form~\eqref{link} by identifying $\C^s$ with $\R^{2s}$ in the
standard way, so that each $\zeta_k$ turns to $\mb
g_k\in\R^{m-n-1}$ with $n=m-2s-1$. We assume that the link is
nondegenerate, i.e. the system of complex vectors
$(\zeta_1,\ldots,\zeta_m)$ (or the corresponding system of real
vectors $(\mb g_1,\ldots,\mb g_m)$) satisfies the conditions (a)
and~(b) of Proposition~\ref{nondeglink}.

Now define the manifold $\mathcal N$ as the projectivisation of
the intersection of homogeneous quadrics in~\eqref{clink}:
\begin{equation}\label{ndef}
\mathcal N=\bigl\{\mb z\in\C
P^{m-1}\colon\zeta_1|z_1|^2+\cdots+\zeta_m|z_m|^2=\textbf
0\bigr\},\quad\zeta_k\in\C^s.
\end{equation}

We therefore have a principal $S^1$-bundle $\mathcal L\to\mathcal
N$.

\begin{theorem}[Meersseman~\cite{meer00}]\label{thlvm}
The manifold $\mathcal N$ has a holomorphic atlas describing it as
a compact complex manifold of complex dimension $m-1-s$.
\end{theorem}
\begin{proof}[Sketch of proof]
Consider a holomorphic action of $\C^s$ on $\C^m$ given by
\begin{equation}\label{csact}
\begin{aligned}
  \C^s\times\C^m&\longrightarrow\C^m\\
  (\mb w,\mb z)&\mapsto \bigl(z_1e^{\langle\zeta_1,\mb
  w\rangle},\ldots,z_me^{\langle\zeta_m,\mb w\rangle}\bigr),
\end{aligned}
\end{equation}
where $\mb w=(w_1,\ldots,w_s)\in\C^s$, and $\langle\zeta_k,\mb
w\rangle=\zeta_{1k}w_1+\cdots+\zeta_{sk}w_s$.

Let $\sK$ be the simplicial complex consisting of zero-sets of
points of the link~$\mathcal L$:
\[
  \sK=\{\omega(\mb z)\colon\mb z\in \mathcal L\}.
\]
Observe that $\sK=\sK_P$, where $P$ is the simple polytope
associated with the link~$\mathcal L$. Let $U=U(\sK)$ be the
corresponding subspace arrangement complement given by~\eqref{uk}.
Note that Proposition~\ref{galecomb} implies that $U$ can be also
defined as
\[
  U=\bigl\{(z_1,\ldots,z_m)\in\C^m\colon\mathbf 0\in\conv(\zeta_j\colon z_j\ne0)\bigr\}.
\]

An argument similar to that of the proof of Lemma~\ref{afree}
shows that the restriction of the action~\eqref{csact} to
$U\subset\C^m$ is free. Also, this restricted action is proper (we
shall prove this in more general context in
Theorem~\ref{zkcomplex} below), so the quotient $U/\C^s$ is
Hausdorff. Using a holomorphic atlas transverse to the orbits of
the free action of $\C^s$ on the complex manifold $U$ we obtain
that the quotient $U/\C^s$ has a structure of a complex manifold.

On the other hand, it can be shown that the function
$|z_1|^2+\cdots+|z_m|^2$ on $\C^m$ has a unique minimum when
restricted to an orbit of the free action of $\C^s$ on $U$. The
set of these minima can be described as
\[
  \mathcal T=\bigl\{\mb z\in\C^m\setminus\!\{\mathbf 0\}
  \colon\;\zeta_1|z_1|^2+\cdots+\zeta_m|z_m|^2=\textbf 0\bigr\}.
\]
It follows that the quotient $U/\C^s$ can be identified
with~$\mathcal T$, and therefore $\mathcal T$ acquires a structure
of a complex manifold of dimension $m-s$.

By projectivising the construction we identify $\mathcal N$ with
the quotient of a complement of coordinate subspace arrangement in
$\C P^{m-1}$ (the projectivisation of~$U$) by a holomorphic action
of~$\C^s$. In this way $\mathcal N$ becomes a compact complex
manifold.
\end{proof}

The manifold $\mathcal N$ with the complex structure of
Theorem~\ref{thlvm} is referred to as an \emph{LVM-manifold}.
These manifolds were described by Meersseman~\cite{meer00} as a
generalisation of the construction of Lopez de Medrano and
Verjovsky~\cite{lo-ve97}.

\begin{remark}
The embedding of $\mathcal T$ in $\C^m$ and of $\mathcal N$ in $\C
P^{m-1}$ given by~\eqref{ndef} is not holomorphic.
\end{remark}

A polytopal moment-angle manifold $\zp$ is diffeomorphic to a
link~\eqref{link}, which can be turned into a complex
link~\eqref{clink} whenever $m+n$ is odd. It follows that the
quotient $\zp/S^1$ of an odd-dimensional moment-angle manifold has
a complex-analytic structure as an LVM-manifold. By adding
redundant inequalities and using the $S^1$-bundle $\mathcal
L\to\mathcal N$, Bosio--Meersseman observed that $\zp$ or
$\zp\times S^1$ has a structure of an LVM-manifold, depending on
whether $m+n$ is even or odd.

We first summarise the effects that a redundant inequality
in~\eqref{ptope} has on different spaces appeared above:

\begin{proposition}
Assume that~\eqref{ptope} is a generic presentation. The following
conditions are equivalent:
\begin{itemize}
\item[(a)] $\langle\mb a_i,\mb x\:\rangle+b_i\ge0$ is a redundant
inequality in~\eqref{ptope} (i.e. $F_i=\varnothing$);
\item[(b)] $\zp\subset\{\mb z\in\C^m\colon z_i\ne0\}$;
\item[(c)] $\{i\}$ is a ghost vertex of~$\sK_P$;
\item[(d)] $U(\sK_P)$ has a factor $\C^\times$ on the $i$th
coordinate;
\item[(e)] $\mathbf 0\notin\conv(\mb g_k\colon k\ne i)$.
\end{itemize}
\end{proposition}
\begin{proof}
The equivalence of the first four conditions follows directly from
the definitions. The equivalence (a)$\Leftrightarrow$(e) follows
from Proposition~\ref{galecomb}.
\end{proof}

\begin{theorem}[\cite{bo-me06}]\label{bometh}
Let $\zp$ be the moment angle manifold corresponding to an
$n$-dimensional simple polytope~\eqref{ptope} defined by $m$
inequalities.
\begin{itemize}
\item[(a)] If $m+n$ is even then $\zp$ has a complex structure as
an LVM-manifold.
\item[(b)] If $m+n$ is odd then $\zp\times S^1$ has a complex structure as
an LVM-manifold.
\end{itemize}
\end{theorem}
\begin{proof}
(a) We add one redundant inequality of the form $1\ge0$
to~\eqref{ptope}, and denote the resulting manifold
of~\eqref{cdiz} by $\zp'$. We have $\zp'\cong\zp\times S^1$. By
Proposition~\ref{malink}, $\zp$ is diffeomorphic to a link given
by~\eqref{link}. Then $\zp'$ is given by the intersection of
quadrics
\[
  \left\{\begin{array}{lrcccrcr}
  \mb z\in\C^{m+1}\colon&|z_1|^2&+&\cdots&+&|z_m|^2&&=1,\\[1mm]
                 &\mb g_1|z_1|^2&+&\cdots&+&\mb g_m|z_m|^2&&=\mathbf 0,\\[1mm]
  &&&&&&&|z_{m+1}|^2=1,
  \end{array}\right\}
\]
which is diffeomorphic to the link given by
\[
  \left\{\begin{array}{lrcccrcr}
  \mb z\in\C^{m+1}\colon  &|z_1|^2&+&\cdots&+&|z_m|^2&+&|z_{m+1}|^2=1,\\[1mm]
  &\mb g_1|z_1|^2&+&\cdots&+&\mb g_m|z_m|^2&&=\mathbf 0,\\[1mm]
  &|z_1|^2&+&\cdots&+&|z_m|^2&-&|z_{m+1}|^2=0.
  \end{array}\right\}
\]
If we denote by $\varGamma^\star=(\mb g_1\ldots\:\mb g_m)$ the
$(m-n-1)\times m$-matrix of coefficients of the homogeneous
quadrics for~$\zp$, then the corresponding matrix for $\zp'$ is
\[
  {\varGamma^\star}^\prime=\begin{pmatrix}\mb g_1&\cdots&\mb
  g_m&0\\1&\cdots&1&-1
  \end{pmatrix}.
\]
Its height $m-n$ is even, so that we may think of its $i$th column
as a complex vector $\zeta_i$ (by identifying $\R^{m-n}$ with
$\C^{\frac{m-n}2}$), for $i=1,\ldots,m+1$. Now define
\begin{equation}\label{nprimeeven}
  \mathcal N'=\bigl\{\mb z\in\C P^m\colon
  \zeta_1|z_1|^2+\cdots+\zeta_{m+1}|z_{m+1}|^2=\mathbf
  0\bigr\}.
\end{equation}
Then $\mathcal N'$ has a complex structure as an LVM-manifold by
Theorem~\ref{thlvm}. On the other hand,
\[
  \mathcal N'\cong\zp'/S^1=(\zp\times S^1)/S^1\cong\zp,
\]
so that $\zp$ also acquires a complex structure.

\smallskip

(b) The proof here is similar, but we have to add two redundant
inequalities $1\ge0$ to~\eqref{ptope}. Then $\zp'\cong\zp\times
S^1\times S^1$ is given by
\[
  \left\{\begin{array}{lrcrclcr}
  \mb
  z\in\C^{m+2}\colon&|z_1|^2&+\;\cdots\;+&|z_m|^2&+&|z_{m+1}|^2&+&|z_{m+2}|^2=1,\\[1mm]
  &\mb g_1|z_1|^2&+\;\cdots\;+&\mb g_m|z_m|^2&&&&=\mathbf 0,\\[1mm]
  &|z_1|^2&+\;\cdots\;+&|z_m|^2&-&|z_{m+1}|^2&&=0,\\[1mm]
  &|z_1|^2&+\;\cdots\;+&|z_m|^2&&&-&|z_{m+2}|^2=0.
  \end{array}\right\}
\]
The matrix of coefficients of the homogeneous quadrics is
therefore
\[
  {\varGamma^\star}^\prime=\begin{pmatrix}\mb g_1&\cdots&\mb
  g_m&0&0\\1&\cdots&1&-1&0\\1&\cdots&1&0&-1
  \end{pmatrix}.
\]
We think of its columns as a set of $m+2$ complex vectors
$\zeta_1,\ldots,\zeta_{m+2}$, and define
\begin{equation}\label{nprimeodd}
  \mathcal N'=\bigl\{\mb z\in\C P^{m+1}\colon
  \zeta_1|z_1|^2+\cdots+\zeta_{m+2}|z_{m+2}|^2=\mathbf
  0\bigr\}.
\end{equation}
Then $\mathcal N'$ has a complex structure as an LVM-manifold. On
the other hand,
\[
  \mathcal N'\cong\zp'/S^1=(\zp\times S^1\times S^1)/S^1\cong\zp\times S^1,
\]
and therefore $\zp\times S^1$ has a complex structure.
\end{proof}

In the next two sections we describe a more direct method of
endowing $\zp$ with a complex structure, without referring to
projectivised quadrics and LVM-manifolds. This approach, developed
in~\cite{pa-us12}, works not only in the polytopal case, but also
for the moment-angle manifolds $\zk$ corresponding to underlying
complexes $\sK$ of complete simplicial fans.

\section{Moment-angle manifolds from simplicial
fans}\label{mamsf} Let $\sK=\sK_\Sigma$ be the underlying complex
of a complete simplicial fan~$\Sigma$, and $U(\sK)$ the complement
of the coordinate subspace arrangement~\eqref{uk} defined
by~$\sK$. Here we shall identify the moment-angle manifold $\zk$
with the quotient of $U(\sK)$ by a smooth action of non-compact
group isomorphic to~$\R^{m-n}$, thereby defining a smooth
structure on~$\zk$. A modification of this construction will be
used in the next section to endow $\zk$ with a complex structure.
These results were obtained in the work~\cite{pa-us12} of
Ustinovsky and the author.

We recall from Subsection~\ref{conesfans} that a simplicial fan
$\Sigma$ can be defined by the data $\{\sK;\mb a_1,\ldots,\mb
a_m\}$, where
\begin{itemize}
\item $\sK$ is a simplicial complex on $[m]$;
\item $\mb a_1,\ldots,\mb a_m$ is a configuration of vectors in
$N_\R\cong\R^n$ such that the subset $\{\mb a_i\colon i\in I\}$ is
linearly independent for any simplex $I\in\sK$.
\end{itemize}

Here is an important point in which our approach to fans differs
from the standard one adopted in toric geometry: since we allow
ghost vertices in~$\sK$, we do not require that each vector $\mb
a_i$ spans a one-dimensional cone of~$\Sigma$. The vector $\mb
a_i$ corresponding to a ghost vertex $\{i\}\in[m]$ may be zero.
This formalism was also used in~\cite{ba-za} under the name
\emph{triangulated vector configurations}.

\begin{construction}
For a set of vectors $\mb a_1,\ldots,\mb a_m$, consider the linear
map
\begin{equation}\label{lambdar}
  A\colon\R^m\to N_\R,\quad\mb e_i\mapsto\mb a_i,
\end{equation}
where $\mb e_1,\ldots,\mb e_m$ is the standard basis of~$\R^m$.
Let
\[
  \R^m_>=\{(y_1,\ldots,y_m)\in\R^m\colon y_i>0\}
\]
be the multiplicative group of $m$-tuples of positive real
numbers, and define
\begin{equation}\label{rsigma}
\begin{aligned}
  R&=\exp(\Ker A)=\bigl\{\bigl(e^{y_1},\ldots,e^{y_m}\bigr)
  \colon(y_1,\ldots,y_m)\in\Ker A\bigr\}\\
  &=\bigl\{(t_1,\ldots,t_m)\in\R^m_>\colon
  \prod_{i=1}^mt_i^{\langle\mb a_i,\mb u\rangle}=1
  \text{ for all }\mb u\in N^*_\R\bigr\}.
\end{aligned}
\end{equation}
\end{construction}

We let $\R^m_>$ act on the complement $U(\sK)\subset\C^m$ by
coordinatewise multiplications and consider the restricted action
of the subgroup $R\subset\R^m_>$. Recall that an action of a
topological group $G$ on a space $X$ is \emph{proper} if the
\emph{group action map} $h\colon G\times X\to X\times X$, \
$(g,x)\mapsto (gx,x)$ is proper (the preimage of a compact subset
is compact).

\begin{theorem}[\cite{pa-us12}]\label{zksmooth}
Assume given data $\{\sK;\mb a_1,\ldots,\mb a_m\}$ satisfying the
conditions above. Then
\begin{itemize}
\item[(a)]
the group $R$ given by \eqref{rsigma} acts on $U(\sK)$ freely;

\item[(b)] if the data $\{\sK;\mb a_1,\ldots,\mb a_m\}$ defines a simplicial fan~$\Sigma$,
then $R$ acts on $U(\sK)$ properly, so the quotient $U(\sK)/R$ is
a smooth Hausdorff $(m+n)$-dimensional manifold;

\item[(c)] if the fan $\Sigma$ is complete, then $U(\sK)/R$ is
homeomorphic to the moment-angle manifold~$\zk$.
\end{itemize}
Therefore, $\zk$ can be smoothed whenever $\sK=\sK_\Sigma$ for a
complete simplicial fan~$\Sigma$.
\end{theorem}

\begin{proof}
Statement~(a) is proved in the same way as
Proposition~\ref{freeaction}. Indeed, a point $\mb z\in U(\sK)$
has a nontrivial isotropy subgroup with respect to the action of
$\R^m_>$ only if some of its coordinates vanish. These
$\R^m_>$-isotropy subgroups are of the form $(\R_>,1)^I$,
see~\eqref{XAI}, for some $I\in\sK$. The restriction of $\exp A$
to any such $(\R_>,1)^I$ is an injection. Therefore, $R=\exp(\Ker
A)$ intersects any $\R^m_>$-isotropy subgroup only at the unit,
which implies that the $R$-action on $U(\sK)$ is free.

Let us prove~(b). Consider the map
\[
  h\colon R\times U(\sK)\to U(\sK)\times U(\sK), \quad (\mb
  g,\mb z)\mapsto (\mb g\mb z,\mb z),
\]
for $\mb g\in R$, $\mb z\in U(\sK)$. Let $V\subset U(\sK)\times
U(\sK)$ be a compact subset; we need to show that $h^{-1}(V)$ is
compact. Since $R\times U(\sK)$ is metrisable, it suffices to
check that any infinite sequence $\{(\mb g^{(k)},\mb
z^{(k)})\colon k=1,2,\ldots\}$ of points in $h^{-1}(V)$ contains a
converging subsequence. Since $V\subset U(\sK)\times U(\sK)$ is
compact, by passing to a subsequence we may assume that the
sequence
\[
  \{h(\mb g^{(k)},\mb z^{(k)})\}=\{(\mb g^{(k)}\mb z^{(k)},\mb z^{(k)})\}
\]
has a limit in $U(\sK)\times U(\sK)$. We set $\mb w^{(k)}=\mb
g^{(k)}\mb z^{(k)}$, and assume that
\[
  \{\mb w^{(k)}\}\to {\mb w}=(w_1,\dots,w_m),\quad
  \{\mb z^{(k)}\}\to {\mb z}=(z_1,\dots,z_m)
\]
for some $\mb w,\mb z\in U(\sK)$. We need to show that a
subsequence of $\{\mb g^{(k)}\}$ has limit in~$R$. We write
\[
  \mb g^{(k)}=\bigl(g_1^{(k)},\ldots,g_m^{(k)}\bigr)=
  \bigl(e^{\alpha^{(k)}_1},\ldots,
  e^{\alpha^{(k)}_m}\bigr)\in R\subset \R^m_>,
\]
$\alpha_j^{(k)}\in\R$. By passing to a subsequence we may assume
that each sequence $\{\alpha^{(k)}_j\}$, \ $j=1,\ldots,m$, has a
finite or infinite limit (including $\pm\infty$). Let
\[
  I_+=\{j\colon\alpha^{(k)}_j\to +\infty\}\subset[m],\quad
  I_-=\{j\colon\alpha^{(k)}_j\to -\infty\}\subset[m].
\]
Since the sequences $\{\mb z^{(k)}\}$, $\{\mb w^{(k)}=\mb
g^{(k)}\mb z^{(k)}\}$ are converging to $\mb z,\mb w\in U(\sK)$
respectively, we have $z_j=0$ for $j\in I_+$ and $w_j=0$ for $j\in
I_-$. Then it follows from the decomposition
$U(\sK)=\bigcup_{I\in\sK}(\C,\C^\times)^I$ that $I_+$ and $I_-$
are simplices of~$\sK$. Let $\sigma_+,\sigma_-$ be the
corresponding cones of the simplicial fan~$\Sigma$. Then
$\sigma_+\cap\sigma_-=\{\mb 0\}$ by definition of a fan. By
Lemma~\ref{seplemma}, there exists a linear function $\mb u\in
N^*_\R$ such that $\langle\mb u,\mb a\rangle>0$ for any nonzero
$\mb a\in \sigma_+$, and $\langle\mb u,\mb a\rangle<0$ for any
nonzero $\mb a\in \sigma_-$. Since $\mb g^{(k)}\in R$, it follows
from~\eqref{rsigma} that
\begin{equation}\label{alphak1}
  \sum_{j=1}^m \alpha^{(k)}_j\langle\mb u,\mb a_j\rangle=0.
\end{equation}
This implies that both $I_+$ and $I_-$ are empty, as otherwise the
latter sum tends to infinity. Thus, each sequence
$\{\alpha^{(k)}_j\}$ has a finite limit $\alpha_j$, and a
subsequence of $\{\mb g^{(k)}\}$ converges to
$(e^{\alpha_1},\ldots,e^{\alpha_m})$. Passing to the limit
in~\eqref{alphak1} we obtain that
$(e^{\alpha_1},\ldots,e^{\alpha_m})\in R$. This proves the
properness of the action. Since the Lie group $R(\Sigma)$ acts
smoothly, freely and properly on the smooth manifold $U(\sK)$, the
orbit space $U(\sK)/R$ is Hausdorff and smooth by the standard
result~\cite[Theorem~9.16]{lee00}.

In the case of complete fan it is possible to construct a smooth
atlas on $U(\sK)/R$ explicitly. To do this, it is convenient to
pre-factorise everything by the action of $\T^m$, as in the proof
of Theorem~\ref{deret}. We  have
\[
  U(\sK)/\T^m=(\R_\ge,\R_>)^\sK=\bigcup_{I\in\sK}(\R_\ge,\R_>)^I.
\]
Since the fan $\Sigma$ is complete, we may take the union above
only over $n$-element simplices $I=\{i_1,\ldots,i_n\}\in\sK$.
Consider one such simplex~$I$; the generators of the corresponding
$n$-dimensional cone $\sigma\in\Sigma$ are $\mb a_{i_1},\ldots,\mb
a_{i_n}$. Let $\mb u_1,\ldots,\mb u_n$ denote the dual basis of
$N_\R^*$, that is, $\langle\mb a_{i_k},\mb
u_j\rangle=\delta_{kj}$. Now consider the map
\begin{align*}
  p_I\colon(\R_\ge,\R_>)^I&\to\R^n_\ge\\
  (y_1,\ldots,y_m)&\mapsto
  \Bigr(\prod_{i=1}^my_i^{\langle\mb a_i,\mb u_1\rangle},\ldots,\:
  \prod_{i=1}^my_i^{\langle\mb a_i,\mb u_n\rangle}\Bigl),
\end{align*}
where we set $0^0=1$. Note that zero cannot occur with a negative
exponent in the right hand side, hence $p_I$ is well defined as a
continuous map. Each $(\R_\ge,\R_>)^I$ is $R$-invariant, and it
follows from~\eqref{rsigma} that $p_I$ induces an injective map
\[
  q_I\colon(\R_\ge,\R_>)^I/R\to\R^n_\ge.
\]
This map is also surjective since every
$(x_1,\ldots,x_n)\in\R^n_\ge$ is covered by $(y_1,\ldots,y_m)$
where $y_{i_j}=x_j$ for $1\le j\le n$ and $y_k=1$ for
$k\notin\{i_1,\ldots,i_n\}$. Hence, $q_I$ is a homeomorphism. It
is covered by a $\T^m$-equivariant homeomorphism
\[
  \overline q_I\colon
  (\C,\C^\times)^I/R\to\C^n\times\T^{m-n},
\]
where $\C^n$ is identified with the quotient
$\R_\ge^n\times\T^n/\!\sim\,$, see~\eqref{Cmids}. Since $U(\sK)/R$
is covered by open subsets $(\C,\C^\times)^I/R$, and
$\C^n\times\T^{m-n}$ embeds as an open subset in $\R^{m+n}$, the
set of homeomorphisms $\{\overline q_I\colon I\in\sK\}$ provides
an atlas for $U(\sK)/R$. The change of coordinates transformations
$\overline q_J\overline
q_I^{-1}\colon\C^n\times\T^{m-n}\to\C^n\times\T^{m-n}$ are smooth
by inspection; thus $U(\sK)/R$ is a smooth manifold.

\begin{remark}
The set of homeomorphisms
$\{q_I\colon(\R_\ge,\R_>)^I/R\to\R^n_\ge\}$ defines an atlas for
the smooth manifold with corners $\zk/\T^m$. If $\sK=\sK_P$ for a
simple polytope $P$, then this smooth structure with corners
coincides with that of~$P$.
\end{remark}

It remains to prove statement~(c), that is, identify $U(\sK)/R$
with $\zk$. If $X$ is a Hausdorff locally compact space with a
proper $G$-action, and $Y\subset X$ a compact subspace which
intersects every $G$-orbit at a single point, then $Y$ is
homeomorphic to the orbit space $X/G$. Therefore, we need to
verify that each $R$-orbit intersects $\zk\subset U(\sK)$ at a
single point. We first prove that the $R$-orbit of any $\mb y\in
U(\sK)/\T^m=(\R_\ge,\R_>)^\sK$ intersects $\zk/\T^m$ at a single
point. For this we use the cubical decomposition
$\cc(\sK)=(\I,1)^\sK$ of $\zk/\T^m$, see Example~\ref{exkpo}.2.

Assume first that $\mb y\in\R^m_>$. The $R$-action on $\R^m_>$ is
obtained by exponentiating the linear action of $\Ker A$ on
$\R^m$. Consider the subset $(\R_\le,0)^\sK\subset\R^m$, where
$\R_\le$ denotes the set of nonpositive reals. It is taken by the
exponential map $\exp\colon\R^m\to\R^m_>$ homeomorphically onto
$\cc^\circ(\sK)=((0,1],1)^\sK\subset\R^m_>$, where $(0,1]$ is
denotes the semi-interval $\{y\in\R\colon 0<y\le1\}$. The map
\begin{equation}\label{oto}
  A\colon(\R_\le,0)^\sK\to N_\R
\end{equation}
takes every $(\R_\le,0)^I$ to $-\sigma$, where $\sigma\in\Sigma$
is the cone corresponding to $I\in\sK$. Since $\Sigma$ is
complete, map~\eqref{oto} is one-to-one.

The orbit of $\mb y$ under the action of $R$ consists of points
$\mb w\in\R^m_>$ such that $\exp A\mb w=\exp A\mb y$. Since $A\mb
y\in N_\R$ and map~\eqref{oto} is one-to-one, there is a unique
point $\mb y'\in(\R_\le,0)^\sK$ such that $A\mb y'=A\mb y$. Since
$\exp A\mb y'\subset\cc^\circ(\sK)$, the $R$-orbit of $\mb y$
intersects the interior $\cc^\circ(\sK)$ and therefore $\cc(\sK)$
at a unique point.

Now let $\mb y\in(\R_\ge,\R_>)^\sK$ be an arbitrary point. Let
$\omega(\mb y)\in\sK$ be the set of zero coordinates of~$\mb y$,
and let $\sigma\in\Sigma$ be the cone corresponding to $\omega(\mb
y)$. The cones containing $\sigma$ constitute a fan $\st\sigma$
(called the \emph{star} of~$\sigma$) in the quotient space
$N_\R/\R\langle\mb a_i\colon i\in\omega(\mb y)\rangle$. The
underlying simplicial complex of $\st\sigma$ is the \emph{link}
$\lk\omega(\mb y)$ of $\omega(\mb y)$ in~$\sK$. Now observe that
the action of $R$ on the set
\[
  \{(y_1,\ldots,y_m)\in(\R_\ge,\R_>)^\sK\colon y_i=0\text{ for }i\in \omega(\mb
  y)\}\cong(\R_\ge,\R_>)^{\lk\omega(\mb y)}
\]
coincides with the action of the group $R_{\st\sigma}$ (defined by
the fan~$\st\sigma$). Now we can repeat the above arguments for
the complete fan $\st\sigma$ and the action of $R_{\st\sigma}$ on
$(\R_\ge,\R_>)^{\lk\omega(\mb y)}$. As a result, we obtain that
every $R$-orbit intersects $\cc(\sK)$ at a unique point.

To finish the proof of~(c) we consider the commutative diagram
\[
\begin{CD}
\zk @>>> U(\sK)\\
@VVV @VV\pi V\\
\cc(\sK) @>>> (\R_\ge,\R_>)^\sK
\end{CD}
\]
where the horizontal arrows are embeddings and the vertical ones
are projections onto the quotients of $\T^m$-actions. Note that
the projection $\pi$ commutes with the $R$-actions on $U(\sK)$ and
$(\R_\ge,\R_>)^\sK$, and the subgroups $R$ and $\T^m$ of
$(\C^\times)^m$ intersect trivially. It follows that every
$R$-orbit intersects the full preimage $\pi^{-1}(\cc(\sK))=\zk$ at
a unique point. Indeed, assume that $\mb z$ and $r\mb z$ are in
$\zk$ for some $\mb z\in U(\sK)$ and $r\in R$. Then $\pi(\mb z)$
and $\pi(r\mb z)=r\pi(\mb z)$ are in $\cc(\sK)$, which implies
that $\pi(\mb z)=\pi(r\mb z)$. Hence, $\mb z=\mb t r\mb z$ for
some $\mb t\in\T^m$. We may assume that $\mb z\in(\C^\times)^m$,
so that the action of both $R$ and $\T^m$ is free (otherwise
consider the action on $U(\lk\omega(\mb z))$). It follows that
$\mb t r=\mathbf 1$, which implies that $r=\mathbf 1$, since $R$
and $\T^m$ intersect trivially.
\end{proof}

We do not know if Theorem~\ref{zksmooth} generalises to other
sphere triangulations:

\begin{question}
Describe the class of sphere triangulations $\sK$ for which the
moment-angle manifold $\zk$ admits a smooth structure.
\end{question}

\begin{remark}
Even if $\zk$ admits a smooth structure for some simplicial
complexes $\sK$ not arising from fans, such a structure does not
come from a quotient $U(\sK)/R$ determined by data $\{\sK;\mb
a_1,\ldots,\mb a_m\}$. In fact, the $R$-action on $U(\sK)$ is
proper and the quotient $U(\sK)/R$ is Hausdorff \emph{precisely
when} $\{\sK;\mb a_1,\ldots,\mb a_m\}$ defines a fan, i.e. the
simplicial cones generated by any two subsets $\{\mb a_i\colon
i\in I\}$ and $\{\mb a_j\colon j\in J\}$ with $I,J\in\sK$ can be
separated by a hyperplane. This observation is originally due to
Bosio~\cite{bosi01}, see also~\cite[\S II.3]{a-d-h-l}
and~\cite{ba-za}.
\end{remark}

\section{Complex geometry of moment-angle manifolds}\label{camam}
Here we show that the even-dimensional moment-angle manifold $\zk$
corresponding to a complete simplicial fan $\Sigma$ admits a
structure of a complex manifold. The idea is to replace the action
of $R\cong\R^{m-n}_>$ on $U(\sK)$ (whose quotient is $\zk$) by a
holomorphic action of $\C^{\frac{m-n}2}$ on the same space.

In this section we assume that $m-n$ is even. We can always
achieve this by adding a ghost vertex with any corresponding
vector to our data $\{\sK;\mb a_1,\ldots,\mb a_m\}$; topologically
this results in multiplying $\zk$ by a circle. We set
$\ell=\frac{m-n}2$.

We identify $\C^m$ (as a real vector space) with $\R^{2m}$ using
the map
\[
  (z_1,\ldots,z_m)\mapsto(x_1,y_1,\ldots,x_m,y_m),
\]
where $z_k=x_k+iy_k$, and consider the $\R$-linear map
\[
  \Re\colon\C^m\to\R^m,\qquad (z_1,\ldots,z_m)\mapsto(x_1,\ldots,x_m).
\]

In order to obtain a complex structure on the quotient $\zk\cong
U(\sK)/R$ we replace the action of $R$ by the action of a
holomorphic subgroup $C\subset(\C^\times)^m$ by means of the
following construction.

\begin{construction}\label{psi}
Let $\mb a_1,\ldots,\mb a_m$ be a configuration of vectors that
span $N_\R\cong\R^n$. Assume further that $m-n=2\ell$ is even.
Some of the $\mb a_i$'s may be zero. Recall the map
$A\colon\R^m\to N_\R$, \ $\mb e_i\mapsto\mb a_i$.

We choose a complex $\ell$-dimensional subspace in $\C^m$ which
projects isomorphically onto the real $(m-n)$-dimensional subspace
$\Ker A\subset\R^m$. More precisely, let $\mathfrak
c\cong\C^\ell$, and choose a linear map $\varPsi\colon \mathfrak
c\to\C^m$ satisfying the two conditions:
\begin{itemize}
\item[(a)] the composite map
$\mathfrak c\stackrel{\varPsi}\longrightarrow\C^m
\stackrel{\Re}\longrightarrow\R^m$ is a monomorphism;

\item[(b)] the composite map
$\mathfrak c\stackrel{\varPsi}\longrightarrow\C^m
\stackrel{\Re}\longrightarrow\R^m \stackrel{A}\longrightarrow
N_\R$ is zero.
\end{itemize}
These two conditions are equivalent to the following:
\begin{itemize}
\item[(a')] $\varPsi(\mathfrak c)\cap\overline{\varPsi(\mathfrak c)\!}=\{\mathbf0\}$;
\item[(b')] $\varPsi(\mathfrak c)\subset\Ker(A_\C\colon\C^m\to N_\C)$,
\end{itemize}
where $\overline{\varPsi(\mathfrak c)\!}$ is the complex conjugate
space and $A_\C\colon\C^m\to N_\C$ is the complexification of the
real map $A\colon\R^m\to N_\R$. Consider the following commutative
diagram:
\begin{equation}\label{cdiag}
\begin{CD}
  \mathfrak c @>\varPsi>> \C^m @>\Re>> \R^m @>A>> N_\R\\
  @. @VV\exp V @VV\exp V\\
  \ @. (\C^\times )^m@>|\cdot|>> \R^m_>
\end{CD}
\end{equation}
where the vertical arrows are the componentwise exponential maps,
and $|\cdot|$ denotes the map
$(z_1,\ldots,z_m)\mapsto(|z_1|,\ldots,|z_m|)$. Now set
\begin{equation}\label{csigma}
  C_\varPsi=\exp\varPsi(\mathfrak c)
  =\bigl\{\bigl(e^{\langle\psi_1,\mb w\rangle},\ldots,
  e^{\langle\psi_m,\mb w\rangle}\bigr)\in(\C^\times )^m\bigr\}
\end{equation}
where $\mb w\in \mathfrak c$ and $\psi_i\in \mathfrak c^*$ is
given by the $i$th coordinate projection $\mathfrak
c\stackrel\varPsi\longrightarrow\C^m\to\C$. Then
$C_\varPsi\cong\C^\ell$ is a complex-analytic (but not algebraic)
subgroup in~$(\C^\times)^m$, and therefore there is a holomorphic
action of $C_\varPsi$ on $\C^m$ and $U(\sK)$ by restriction.
\end{construction}

\begin{example}\label{2torus}
Let $\mb a_1,\ldots,\mb a_m$ be the configuration of $m=2\ell$
zero vectors. We supplement it by the empty simplicial complex
$\sK$ on $[m]$ (with $m$ ghost vertices), so that the data
$\{\sK;\mb a_1,\ldots,\mb a_m\}$ defines a complete fan in
0-dimensional space. Then $A\colon\R^m\to\R^0$ is a zero map, and
condition~(b) of Construction~\ref{psi} is void. Condition~(a)
means that $\mathfrak c\stackrel{\varPsi}\longrightarrow\C^{2\ell}
\stackrel{\mathrm{Re}}\longrightarrow\R^{2\ell}$ is an isomorphism
of real spaces.

Consider the quotient $(\C^\times)^m/C_\varPsi$ (note that
$U(\sK)=(\C^\times)^m$ in our case). The exponential map
$\C^m\to(\C^\times)^m$ identifies $(\C^\times)^m$ with the
quotient of $\C^m$ by the imaginary lattice $\Gamma=\Z\langle2\pi
i\mb e_1,\ldots,2\pi i\mb e_m\rangle$. Condition~(a) implies that
the projection $p\colon\C^m\to\C^m/\varPsi(\mathfrak c)$ is
nondegenerate on the imaginary subspace of~$\C^m$. In particular,
$p\,(\Gamma)$ is a lattice of rank $m=2\ell$ in
$\C^m/\varPsi(\mathfrak c)\cong\C^\ell$. Therefore,
\[
  (\C^\times)^m/C_\varPsi\cong\bigl(\C^m/\Gamma\bigr)/\varPsi(\mathfrak c)
  =\bigl(\C^m/\varPsi(\mathfrak c)\bigr)
  \big/p\,(\Gamma)\cong\C^\ell/\Z^{2\ell}
\]
is a complex compact $\ell$-dimensional torus.

Any complex torus can be obtained in this way. Indeed, let
$\varPsi\colon\mathfrak c\to\C^m$ be given by an
$2\ell\times\ell$-matrix $\begin{pmatrix}-B\\I\end{pmatrix}$ where
$I$ is a unit matrix and $B$ is a square matrix of size~$\ell$.
Then $p\colon\C^m\to\C^m/\varPsi(\mathfrak c)$ is given by the
matrix $(I\,B)$ in appropriate bases, and
$(\C^\times)^m/C_\varPsi$ is isomorphic to the quotient of
$\C^\ell$ by the lattice $\Z\langle\mb e_1,\ldots,\mb e_\ell,\mb
b_1,\ldots,\mb b_\ell\rangle$, where $\mb b_k$ is the $k$th column
of~$B$. (Condition~(b) implies that the imaginary part of $B$ is
nondegenerate.)

For example, if $\ell=1$, then $\varPsi\colon\C\to\C^2$ is given
by $w\mapsto(\beta w,w)$ for some $\beta\in\C$, so that
subgroup~\eqref{csigma} is
\[
  C_\varPsi=\{(e^{\beta w},e^w)\}\subset(\C^\times )^2.
\]
Condition~(a) implies that $\beta\notin\R$. Then
$\exp\varPsi\colon\C\to(\C^\times )^2$ is an embedding, and
\[
  (\C^\times )^2/C_\varPsi\cong\C/(\Z\oplus\beta\Z)=T^1_\C(\beta)
\]
is a complex 1-dimensional torus with lattice parameter
$\beta\in\C$.
\end{example}

\begin{theorem}[\cite{pa-us12}]\label{zkcomplex}
Assume that the data $\{\sK;\mb a_1,\ldots,\mb a_m\}$ define a
complete fan~$\Sigma$ in~$N_\R\cong\R^n$, and $m-n=2\ell$. Let
$C_\varPsi\cong\C^\ell$ be given by~\eqref{csigma}. Then
\begin{itemize}
\item[(a)]
the holomorphic action of $C_\varPsi$ on $U(\sK)$ is free and
proper, and the quotient $U(\sK)/C_\varPsi$ has a structure of a
compact complex manifold;

\item[(b)] $U(\sK)/C_\varPsi$ is diffeomorphic to the moment-angle manifold~$\zk$.
\end{itemize}
Therefore, $\zk$ has a complex structure, in which each element of
$\T^m$ acts by a holomorphic transformation.
\end{theorem}

\begin{remark}
A result similar to Theorem~\ref{zkcomplex} was obtained by
Tambour~\cite{tamb12}. The approach of Tambour was somewhat
different; he constructed complex structures on manifolds $\zk$
arising from \emph{rationally} starshaped spheres~$\sK$
(underlying complexes of complete rational simplicial fans) by
relating them to a class of generalised LVM-manifolds described by
Bosio in~\cite{bosi01}.
\end{remark}

\begin{proof}[Proof of Theorem~\ref{zkcomplex}]
We first prove statement (a). The isotropy subgroups of the
$(\C^\times)^m$-action on $U(\sK)$ are of the form
$(\C^\times,1)^I$ for $I\in\sK$. In order to show that
$C_\varPsi\subset(\C^\times)^m$ acts freely we need to check that
$C_\varPsi$ has trivial intersection with any isotropy subgroup
of~$(\C^\times)^m$. Since $C_\varPsi$ embeds into $\R^m_>$
by~\eqref{cdiag}, it enough to check that the image of $C_\varPsi$
in $\R^m_>$ intersects the image of $(\C^\times,1)^I$ in $\R^m_>$
trivially. The former image is $R$ and the latter image is
$(\R_>,1)^I$; the triviality of their intersection follows from
Theorem~\ref{zksmooth}~(a).

Now we prove the properness of this action. Consider the
projection $\pi\colon U(\sK)\to(\R_\ge,\R_>)^\sK$ onto the
quotient of the $\T^m$-action, and the commutative square
\[
\begin{CD}
  C_{\varPsi}\times U(\sK) @>h_\C>> U(\sK)\times U(\sK)\\
  @VVf\times\pi V @VV\pi\times\pi V\\
  R\times (\R_\ge,\R_>)^\sK@>h_\R>> (\R_\ge,\R_>)^\sK\times(\R_\ge,\R_>)^\sK
\end{CD}
\]
where $h_\C$ and $h_\R$ denote the group action maps, and $f\colon
C_{\varPsi}\to R$ is the isomorphism given by the restriction of
$|\cdot|\colon (\C^\times )^m\to \R_>^m$. The preimage
$h^{-1}_\C(V)$ of a compact subset $V\in U(\sK)\times U(\sK)$ is a
closed subset in $W=(f\times\pi)^{-1}\circ
h_\R^{-1}\circ(\pi\times\pi)(V)$. The image $(\pi\times\pi)(V)$ is
compact, the action of $R$ on $(\R_\ge,\R_>)^\sK$ is proper by
Theorem~\ref{zksmooth}~(a), and the map $f\times \pi$ is proper as
the quotient projection for a compact group action. Hence, $W$ is
a compact subset in $C_{\varPsi}\times U(\sK)$, and $h^{-1}_\C(V)$
is compact as a closed subset in~$W$.

The group $C_{\varPsi}\cong\C^l$ acts holomorphically, freely and
properly on the complex manifold $U(\sK)$, therefore the quotient
manifold $U(\sK)/C_\varPsi$ has a complex structure.

As in the proof of Theorem~\ref{zksmooth}, it is possible to
describe a holomorphic atlas of $U(\sK)/C_{\varPsi}$. Since the
action of $C_{\varPsi}$ on the quotient
$U(\sK)/\T^m=(\R_\ge,\R_>)^\sK$ coincides with the action of $R$
on the same space, the quotient of $U(\sK)/C_{\varPsi}$ by the
action of $\T^m$ has exactly the same structure of a smooth
manifold with corners as the quotient of $U(\sK)/R$ by $\T^m$ (see
the proof of Theorem~\ref{zksmooth}). This structure is determined
by the atlas $\{q_I\colon(\R_\ge,\R_>)^I/R\to\R^n_\ge\}$, which
lifts to a covering of $U(\sK)/C_{\varPsi}$ by the open subsets
$(\C,\C^\times)^I/C_{\varPsi}$. For any $I\in\sK$, the subset
$(\C,\T)^I\subset(\C,\C^\times)^I$ intersects each orbit of the
$C_{\varPsi}$-action on $(\C,\C^\times)^I$ transversely at a
single point. Therefore, every
$(\C,\C^\times)^I/C_{\varPsi}\cong(\C,\T)^I$ acquires a structure
of a complex manifold. Since
$(\C,\C^\times)^I\cong\C^n\times(\C^\times)^{m-n}$, and the action
of $C_{\varPsi}$ on the $(\C^\times)^{m-n}$ factor is free, the
complex manifold $(\C,\C^\times)^I/C_{\varPsi}$ is the total space
of a holomorphic $\C^n$-bundle over the complex torus
$(\C^\times)^{m-n}/C_{\varPsi}$ (see Example~\ref{2torus}).
Writing trivialisations of these $\C^n$-bundles for every~$I$, we
obtain a holomorphic atlas for $U(\sK)/C_{\varPsi}$.

The proof of statement (b) follows the lines of the proof of
Theorem~\ref{zksmooth}~(b). We need to show that each
$C_{\varPsi}$-orbit intersects $\zk\subset U(\sK)$ at a single
point. First we show that the $C_{\varPsi}$-orbit of any point in
$U(\sK)/\T^m$ intersects $\zk/\T^m=\cc(\sK)$ at a single point;
this follows from the fact that the actions of $C_{\varPsi}$ and
$R$ coincide on $U(\sK)/\T^m$. Then we show that each
$C_{\varPsi}$-orbit intersects the preimage $\pi^{-1}(\cc(\sK))$
at a single point, using the fact that $C_{\varPsi}$ and $\T^m$
have trivial intersection in $(\C^\times)^m$.
\end{proof}

\begin{example}[Hopf manifold]\label{hopf}
Let $\mb a_1,\ldots,\mb a_{n+1}$ be a set of vectors which span
$N_\R\cong\R^n$ and satisfy a linear relation $\lambda_1\mb
a_1+\cdots+\lambda_{n+1}\mb a_{n+1}=\mathbf0$ with all
$\lambda_k>0$. Let $\Sigma$ be the complete simplicial fan in
$N_\R$ whose cones are generated by all proper subsets of $\mb
a_1,\ldots,\mb a_{n+1}$. To make $m-n$ even we add one more ghost
vector $\mb a_{n+2}$. Hence $m=n+2$, $\ell=1$, and we have one
more linear relation $\mu_1\mb a_1+\cdots+\mu_{n+1}\mb a_{n+1}+\mb
a_{n+2}=\mathbf0$ with $\mu_k\in\R$. The subspace $\Ker
A\subset\R^{n+2}$ is spanned by
$(\lambda_1,\ldots,\lambda_{n+1},0)$ and
$(\mu_1,\ldots,\mu_{n+1},1)$.

Then $\sK=\sK_\Sigma$ is the boundary of an $n$-dimensional
simplex with $n+1$ vertices and one ghost vertex, $\zk\cong
S^{2n+1}\times S^1$, and
$U(\sK)=(\C^{n+1}\setminus\{{\bf0}\})\times\C^\times$.

Conditions~(a) and~(b) of Construction~\ref{psi} imply that
$C_\varPsi$ is a 1-dimensional subgroup in $(\C^\times )^m$ given
in appropriate coordinates by
\[
  C_\varPsi=\bigl\{(e^{\zeta_1 w},\ldots,e^{\zeta_{n+1}w},e^w)
  \colon w\in\C\bigl\}\subset(\C^\times )^m,
\]
where $\zeta_k=\mu_k+\alpha\lambda_k$ for some
$\alpha\in\C\setminus\R$. By changing the basis of $\Ker A$ if
necessary, we may assume that $\alpha=i$. The moment-angle
manifold $\zk\cong S^{2n+1}\times S^1$ acquires a complex
structure as the quotient $U(\sK)/C_\varPsi$:
\begin{multline*}
  \bigl(\C^{n+1}\setminus\{\mathbf 0\}\bigr)\times\C^\times\bigl/\;
  \bigl\{(z_1,\ldots,z_{n+1},t)\!\sim
  (e^{\zeta_1w}z_1,\ldots,e^{\zeta_{n+1}w}z_{n+1},
  e^w t)\bigr\}
  \\
  \cong\bigl(\C^{n+1}\setminus\{\mathbf0\}\bigr)\bigl/\;
  \bigl\{(z_1,\ldots,z_{n+1})\!\sim
  (e^{2\pi i\zeta_1}z_1,\ldots,
  e^{2\pi i\zeta_{n+1}}z_{n+1})\bigr\},
\end{multline*}
where $\mb z\in\C^{n+1}\setminus\{0\}$, $t\in\C^\times$. The
latter is the quotient of $\C^{n+1}\setminus\{0\}$ by a
diagonalisable action of $\Z$. It is known as a \emph{Hopf
manifold}. For $n=0$ we obtain the complex torus (elliptic curve)
of Example~\ref{2torus}.
\end{example}

Theorem~\ref{zkcomplex} can be generalised to the quotients of
$\zk$ by freely acting subgroups $H\subset\T^m$, or \emph{partial
quotients} of~$\zk$ in the sense of~\cite[\S7.5]{bu-pa02}. These
include both toric manifolds and LVM-manifolds.

\begin{construction}\label{psi_part}
Let $\Sigma$ be a complete simplicial fan in $N_\R$ defined by the
data $\{\sK;\mb a_1,\ldots,\mb a_m\}$, and let $H\subset\T^m$ be a
subgroup which acts freely on the corresponding moment-angle
manifold~$\zk$. Then $H$ is a product of a torus and a finite
group, and $h=\dim H\le m-n$ by Proposition~\ref{freeaction} ($H$
must intersect trivially with an $n$-dimensional coordinate
subtorus in~$\T^m$). Under an additional assumption on~$H$, we
shall define a holomorphic subgroup $D$ in $(\C^\times)^m$ and
introduce a complex structure on $\zk/H$ by identifying it with
the quotient $U(\sK)/D$.

The additional assumption is the compatibility with the fan data.
Recall the map $A_\R\colon \R^m\to N_\R$, \ $\mb e_i\mapsto\mb
a_i$, and let $\mathfrak h\subset\R^m$ be the Lie algebra of
$H\subset\T^m$. We assume that $\mathfrak h\subset\Ker A_\R$. We
also assume that $2\ell=m-n-h$ is even (this can be satisfied by
adding a zero vector to $\mb a_1,\ldots,\mb a_m$). Let $T=\T^m/H$
be the quotient torus, $\mathfrak t$ its Lie algebra, and
$\rho\colon\R^m\to\mathfrak t$ the map of Lie algebras
corresponding to the quotient projection $\T^m\to T$.

Let $\mathfrak c\cong\C^\ell$, and choose a linear map
$\varOmega\colon \mathfrak c\to\C^m$ satisfying the two
conditions:
\begin{itemize}
\item[(a)] the composite map
$\mathfrak c\stackrel{\varOmega}\longrightarrow\C^m
\stackrel{\Re}\longrightarrow\R^m \stackrel{\rho}\longrightarrow
\mathfrak t$ is a monomorphism;
\item[(b)] the composite map
$\mathfrak c\stackrel{\varOmega}\longrightarrow\C^m
\stackrel{\Re}\longrightarrow\R^m \stackrel{A}\longrightarrow
N_\R$ is zero.
\end{itemize}
Equivalently, choose a complex subspace $\mathfrak
c\subset\mathfrak t_\C$ such that the composite map $\mathfrak
c\to\mathfrak t_\C\stackrel{\Re}\longrightarrow\mathfrak t$ is a
monomorphism.

As in Construction~\ref{psi}, $\exp\varOmega(\mathfrak
c)\subset(\C^\times)^m$ is a holomorphic subgroup isomorphic
to~$\C^\ell$. Let $H_\C\subset(\C^\times)^m$ be the
complexification of $H$ (it is a product of an algebraic torus of
dimension $h$ and a finite group). It follows from (a) that the
subgroups $H_\C$ and $\exp\varOmega(\mathfrak c)$ intersect
trivially in~$(\C^\times)^m$. We therefore define a complex
$(h+\ell)$-dimensional subgroup
\begin{equation} \label{csigma_part}
  D_{H,\varOmega}=H_\C\times\exp\varOmega(\mathfrak c)
  \subset(\C^\times)^m.
\end{equation}
\end{construction}

\begin{theorem}[{\cite[Th.~3.7]{pa-us12}}]\label{zkcomplex_part}
Let $\Sigma$, $\sK$ and $D_{H,\varOmega}$ be as above. Then
\begin{itemize}
\item[(a)] the holomorphic action of the group
$D_{H,\varOmega}$ on $U(\sK)$ is free and proper, and the quotient
$U(\sK)/D_{H,\varOmega}$ has a structure of a compact complex
manifold of complex dimension $m-h-\ell$;

\item[(b)] there is a diffeomorphism between
$U(\sK)/D_{H,\varOmega}$ and $\zk/H$ defining a complex structure
on the quotient $\zk/H$, in which each element of $T=\T^m/H$ acts
by a holomorphic transformation.
\end{itemize}
\end{theorem}

The proof is similar to that of Theorem~\ref{zkcomplex} and is
omitted.

\begin{example}\

1. If $H$ is trivial ($h=0$) then we obtain
Theorem~\ref{zkcomplex}.

2. Let $H$ be the diagonal circle in~$\T^m$. The condition
$\mathfrak h\subset\Ker A_\R$ implies that the vectors $\mb
a_1,\ldots,\mb a_m$ sum up to zero, which can always be achieved
by rescaling them (as $\Sigma$ is a complete fan). As the result,
we obtain a complex structure on the quotient $\zk/S^1$ by the
diagonal circle in~$\T^m$, provided that $m-n$ is odd. In the
polytopal case $\sK=\sK_P$, the quotient $\zk/S^1$ embeds into
$\C^m\setminus\{\mathbf 0\}/\C^\times=\C P^{m-1}$ as an
intersection of homogeneous quadrics~\eqref{ndef}, and the complex
structure on $\zk/S^1$ coincides with that of an
\emph{LVM-manifold}, see Section~\ref{lvmma}.

3. Let $h=\dim H=m-n$. Then $\mathfrak h=\Ker A$. Since $\mathfrak
h$ is the Lie algebra of a torus, the $(m-n)$-dimensional subspace
$\Ker A\subset\R^m$ is rational. By Gale duality, this implies
that the fan $\Sigma$ is also rational. We have $\ell=0$,
$D_{H,\varOmega}=H_\C\cong(\C^\times)^{m-n}$ and
$U(\sK)/H_\C=\zk/H$ is the toric variety corresponding
to~$\Sigma$.
\end{example}

As it is shown by Ishida~\cite{isid}, any compact complex manifold
with a \emph{maximal} effective holomorphic action of a torus is
biholomorphic to a quotient $\zk/H$ of the moment-angle manifold,
with a complex structure described by
Theorem~\ref{zkcomplex_part}. (An effective action of $T^k$ on an
$m$-dimensional manifold $M$ is called \emph{maximal} if there
exists a point $x\in M$ whose stabiliser has dimension $m-k$; the
two extreme cases are the free action of a torus on itself and the
half-dimensional torus action on a toric manifold.) The argument
of~\cite{isid} recovering a fan $\Sigma$ from a maximal
holomorphic torus action builds up on the works~\cite{i-f-m13}
and~\cite{is-ka}, where the result was proved in particular cases.
The main result of~\cite{is-ka} provides a purely complex-analytic
description of toric manifolds~$V_\Sigma$:

\begin{theorem}[{\cite[Theorem~1]{is-ka}}]
Let $M$ be a compact connected complex manifold of complex
dimension $n$, equipped with an effective action of $T^n$ by
holomorphic transformations. If the action has fixed points, then
there exists a complete regular fan $\Sigma$ and a
$T^n$-equivariant biholomorphism of $V_\Sigma$ with~$M$.
\end{theorem}

\section[Holomorphic principal bundles and Dolbeault
cohomology]{Holomorphic principal bundles over toric varieties and
Dolbeault cohomology} In the case of rational simplicial normal
fans $\Sigma_P$ a construction of
Meersseman--Verjovsky~\cite{me-ve04} identifies the corresponding
projective toric variety $V_P$ as the base of a holomorphic
principal \emph{Seifert fibration}, whose total space is the
moment-angle manifold~$\zp$ equipped with a complex structure of
an LVM-manifold, and fibre is a compact complex torus of complex
dimension $\ell=\frac{m-n}2$. (Seifert fibrations are
generalisations of holomorphic fibre bundles to the case when the
base is an orbifold.) If $V_P$ is a projective toric manifold,
then there is a holomorphic free action of a complex
$\ell$-dimensional torus $T^{\ell}_\C$ on $\zp$ with
quotient~$V_P$.

Using the construction of a complex structure on $\zk$ described
in the previous section, in~\cite{pa-us12} holomorphic (Seifert)
fibrations with total space $\zk$ were defined for arbitrary
complete rational simplicial fans~$\Sigma$. By an application of
the Borel spectral sequence to the holomorphic fibration $\zk\to
V_\Sigma$, the Dolbeault cohomology of $\zk$ can be described and
some Hodge numbers can be calculated explicitly.

Here we make additional assumption that the set of integral linear
combinations of the vectors $\mb a_1,\ldots,\mb a_m$ is a
full-rank lattice (a discrete subgroup isomorphic to~$\Z^n$)
in~$N_\R\cong\R^n$. We denote this lattice by $N_\Z$ or
simply~$N$. This assumption implies that the complete simplicial
fan $\Sigma$ defined by the data $\{\sK;\mb a_1,\ldots,\mb a_m\}$
is \emph{rational}. We also continue assuming that $m-n$ is even
and setting $\ell=\frac{m-n}2$.

Because of our rationality assumption, the algebraic group $G$ is
defined by~\eqref{gexpl}. Furthermore, since we defined $N$ as the
lattice generated by $\mb a_1,\ldots,\mb a_m$, the group $G$ is
isomorphic to~$(\C^\times)^{2\ell}$ (i.e. there are no finite
factors). We also observe that $C_{\varPsi}$ lies in~$G$ as an
$\ell$-dimensional complex subgroup. This follows from
condition~(b') of Construction~\ref{psi}.

The quotient construction (Subsection~\ref{quotv}) identifies the
toric variety $V_\Sigma$ with $U(\sK)/G$, provided that $\mb
a_1,\ldots,\mb a_m$ are \emph{primitive} generators of the edges
of~$\Sigma$. In our data $\{\sK;\mb a_1,\ldots,\mb a_m\}$, the
vectors $\mb a_1,\ldots,\mb a_m$ are not necessarily primitive in
the lattice $N$ generated by them. Nevertheless, the quotient
$U(\sK)/G$ is still isomorphic to~$V_\Sigma$,
see~\cite[Proposition~II.3.1.7]{a-d-h-l}. Indeed, let $\mb a'_i\in
N$ be the primitive generator along $\mb a_i$, so that $\mb
a'_i=r_i\mb a_i$ for some positive integer~$r_i$. Then we have a
finite branched covering
\[
  U(\sK)\to U(\sK),\quad(z_1,\ldots,z_m)\mapsto
  (z_1^{r_1},\ldots,z_m^{r_m}),
\]
which maps the group $G'$ defined by $\mb a'_1,\ldots,\mb a'_m$ to
the group $G$ defined by $\mb a_1,\ldots,\mb a_m$,
see~\eqref{gexpl}. We therefore obtain a covering $U(\sK)/G'\to
U(\sK)/G$ of the toric variety $V_\Sigma\cong U(\sK)/G'\cong
U(\sK)/G$ over itself. Having this in mind, we can relate the
quotients $V_\Sigma\cong U(\sK)/G$ and $\zk\cong U(\sK)/C_\varPsi$
as follows:

\begin{proposition}\label{toricfib}Assume that the data
$\{\sK;\mb a_1,\ldots,\mb a_m\}$ defines a complete simplicial
rational fan~$\Sigma$, and let $G$ and $C_\varPsi$ be the groups
defined by~\eqref{gexpl} and~\eqref{csigma}.
\begin{itemize}
\item[(a)]The toric variety $V_\Sigma$ is identified, as a
topological space, with the quotient of $\zk$ by the holomorphic
action of the complex compact torus~$G/C_{\varPsi}$.

\item[(b)]If the fan $\Sigma$ is regular, then $V_\Sigma$ is the base of a
holomorphic principal bundle with total space $\zk$ and fibre the
complex compact torus $G/C_{\varPsi}$.
\end{itemize}
\end{proposition}
\begin{proof}
To prove~(a) we just observe that
\[
  V_\Sigma=U(\sK)/G=
  \bigl(U(\sK)/C_{\varPsi}\bigr)\big/(G/C_{\varPsi})\cong
  \mathcal Z_{\sK}\big/(G/C_{\varPsi}),
\]
where we used Theorem~\ref{zkcomplex}. The quotient
$G/C_{\varPsi}$ is a compact complex $\ell$-torus by
Example~\ref{2torus}. To prove~(b) we observe that the holomorphic
action of $G$ on $U(\sK)$ is free by Proposition~\ref{freeaction},
and the same is true for the action of $G/C_{\varPsi}$ on~$\zk$. A
holomorphic free action of the torus $G/C_{\varPsi}$ gives rise to
a principal bundle.
\end{proof}

\begin{remark}
Like in the projective situation of~\cite{me-ve04}, if the fan
$\Sigma$ is not regular, then the quotient projection $\zk\to
V_\Sigma$ of Proposition~\ref{toricfib}~(a) is a holomorphic
principal \emph{Seifert fibration} for an appropriate orbifold
structure on~$V_\Sigma$.
\end{remark}

Let $M$ be a complex $n$-dimensional manifold. The space
$\varOmega_\C^*(M)$ of complex differential forms on $M$
decomposes into a direct sum of the subspaces of
\emph{$(p,q)$-forms}, $\varOmega_\C^*(M)=\bigoplus_{0\le p,q\le
n}\varOmega^{p,q}(M)$, and there is the \emph{Dolbeault
differential} $\bar\partial\colon\varOmega^{p,q}(M)\to
\varOmega^{p,q+1}(M)$. The dimensions $h^{p,q}(M)$ of the
Dolbeault cohomology groups $H_{\bar\partial}^{p,q}(M)$ are known
as the \emph{Hodge numbers} of~$M$. They are important invariants
of the complex structure of~$M$.

The Dolbeault cohomology of a compact complex $\ell$-torus
$T_\C^{\ell}$ is isomorphic to an exterior algebra on $2\ell$
generators:
\begin{equation}\label{dolbtorus}
  H_{\bar\partial}^{*,*}(T_\C^{\ell})\cong
  \Lambda[\xi_1,\ldots,\xi_\ell,\eta_1,\ldots,\eta_\ell],
\end{equation}
where $\xi_1,\ldots,\xi_\ell\in
H_{\bar\partial}^{1,0}(T_\C^{\ell})$ are the classes of basis
holomorphic 1-forms, and $\eta_1,\ldots,\eta_\ell\in
H_{\bar\partial}^{0,1}(T_\C^{\ell})$ are the classes of basis
antiholomorphic 1-forms. In particular, the Hodge numbers are
given by $h^{p,q}(T_\C^{\ell})=\binom\ell p\binom\ell q$.

The de Rham cohomology of a complete nonsingular toric variety
$V_\Sigma$ admits a Hodge decomposition with only nontrivial
components of bidegree~$(p,p)$, $0\le p\le
n$~{\cite[\S12]{dani78}}. This together with the cohomology
calculation due to Danilov--Jurkiewicz~\cite[\S10]{dani78} gives
the following description of the Dolbeault cohomology:
\begin{equation}\label{dolbtoric}
  H_{\bar\partial}^{*,*}(V_\Sigma)\cong
  \C[v_1,\ldots,v_m]/(\mathcal I_{\sK}+\mathcal J_{\Sigma}),
\end{equation}
where $v_i\in H_{\bar\partial}^{1,1}(V_\Sigma)$ are the cohomology
classes corresponding to torus-invariant divisors (one for each
one-dimensional cone of~$\Sigma$), the ideal $\mathcal I_{\sK}$ is
generated by the monomials $v_{i_1}\cdots v_{i_k}$ for which $\mb
a_{i_1},\ldots,\mb a_{i_k}$ do not span a cone of $\Sigma$ (the
\emph{Stanley--Reisner ideal} of~$\sK$), and $\mathcal J_{\Sigma}$
is generated by the linear forms $\sum_{j=1}^m\langle\mb a_j,\mb
u\rangle v_j$, $\mb u\in N^*$. We have $h^{p,p}(V_\Sigma)=h_p$,
where $(h_0,h_1,\ldots,h_n)$ is the \emph{$h$-vector}
of~$\sK$~\cite[\S2.1]{bu-pa02}, and $h^{p,q}(V_\Sigma)=0$ for
$p\ne q$.

\begin{theorem}[{\cite{pa-us12}}]\label{dolbzp}
Assume that the data $\{\sK;\mb a_1,\ldots,\mb a_m\}$ define a
complete regular fan~$\Sigma$ in~$N_\R\cong\R^n$, $m-n=2\ell$, and
let $\zk$ be the corresponding moment-angle manifold with a
complex structure defined by Theorem~\ref{zkcomplex}. Then the
Dolbeault cohomology algebra $H_{\bar\partial}^{*,*}(\zk)$ is
isomorphic to the cohomology of the differential bigraded algebra
\begin{equation}\label{zkmult}
\bigl[\Lambda[\xi_1,\ldots,\xi_\ell,\eta_1,\ldots,\eta_\ell]\otimes
  H_{\bar\partial}^{*,*}(V_\Sigma),d\bigr]
\end{equation}
with differential $d$ of bidegree $(0,1)$ defined on the
generators as follows:
\[
  dv_i=d\eta_j=0,\quad d\xi_j=c(\xi_j),\quad
  1\le i\le m,\;1\le j\le\ell,
\]
where $c\colon H^{1,0}_{\bar\partial}(T_\C^{\ell})\to
H^2(V_\Sigma, \C)=  H_{\bar\partial}^{1,1}(V_\Sigma)$ is the first
Chern class map of the principal $T^\ell_\C$-bundle $\zk\to
V_\Sigma$.
\end{theorem}
\begin{proof}
We use the notion of a minimal Dolbeault model of a complex
manifold~\cite[\S4.3]{f-o-t08}. Let $[B,d_B]$ be such a model
for~$V_\Sigma$, i.e. $[B,d_B]$ is a minimal commutative bigraded
differential algebra together with a quasi-isomorphism $f\colon
B^{*,*}\to \varOmega^{*,*}(V_\Sigma)$  (i.e. $f$ commutes with the
differentials $d_B$ and $\bar\partial$, and induces an isomorphism
in cohomology). Consider the differential bigraded algebra
\begin{equation}\label{dolbmod}
\begin{gathered}
  \bigl[\Lambda[\xi_1,\ldots,\xi_\ell,\eta_1,\ldots,\eta_\ell]\otimes
  B, d\bigr],\qquad\text{where}\\
  d|_B=d_B,\quad  d(\xi_i)=c(\xi_i)\in
  B^{1,1}= H^{1,1}_{\bar\partial}(V_\Sigma),\quad d(\eta_i)=0.
\end{gathered}
\end{equation}
By \cite[Corollary~4.66]{f-o-t08}, it gives a model for the
Dolbeault cohomology algebra of the total space $\zk$ of the
principal $T^{\ell}_\C$-bundle $\zk\to V_\Sigma$, provided that
$V_\Sigma$ is strictly formal. Recall
from~\cite[Definition~4.58]{f-o-t08} that a complex manifold $M$
is \emph{strictly formal} if there exists a differential bigraded
algebra $[Z,\delta]$ together with quasi-isomorphisms
\[
\xymatrix{
  [\varOmega^{*,*},\bar\partial]  &
  [Z,\delta] \ar[l]_{\simeq} \ar[r]^{\simeq} \ar[d]^{\simeq}&
  [\varOmega^{*},d_{\mathrm{DR}}]\\
  & [H^{*,*}_{\bar\partial}(M),0]
}
\]
linking together the de Rham algebra, the Dolbeault algebra and
the Dolbeault cohomology.

The toric manifold $V_\Sigma$ is formal in the usual (de Rham)
sense by~\cite[Corollary~7.2]{pa-ra08}. The Hodge decomposition
of~\cite[\S12]{dani78} implies that $V_\Sigma$ satisfies the
$\partial\bar\partial$-lemma~\cite[Lemma~4.24]{f-o-t08}. Therefore
$V_\Sigma$ is strictly formal by the same argument
as~\cite[Theorem~4.59]{f-o-t08}, and~\eqref{dolbmod} is a model
for its Dolbeault cohomology.

The usual formality of $V_\Sigma$ implies the existence of a
quasi-isomorphism $\phi_B\colon B\to
H_{\bar\partial}^{*,*}(V_\Sigma)$, which extends to a
quasi-isomorphism
\[
  \mbox{id}\otimes\phi_B\colon
  \bigl[\Lambda[\xi_1,\ldots,\xi_\ell,\eta_1,\ldots,\eta_\ell]
  \otimes B, d\bigr]\to
  \bigl[\Lambda[\xi_1,\ldots,\xi_\ell,\eta_1,\ldots,\eta_\ell]
  \otimes H_{\bar\partial}^{*,*}(V_\Sigma), d\bigr]
\]
by~\cite[Lemma~14.2]{f-h-t01}. Thus, the differential algebra
$\bigl[\Lambda[\xi_1,\ldots,\xi_\ell,\eta_1,\ldots,\eta_\ell]\otimes
H^{*,*}_{\bar\partial}(V_\Sigma), d\bigr]$ provides a model for
the Dolbeault cohomology of $\zk$, as claimed.
\end{proof}

\begin{remark}
If $V_\Sigma$ is projective, then it is K\"ahler; in this case the
model of Theorem~\ref{dolbzp} coincides with the model for the
Dolbeault cohomology of the total space of a holomorphic torus
principal bundle over a K\"ahler
manifold~\cite[Theorem~4.65]{f-o-t08}.
\end{remark}

The first Chern class map $c$ from Theorem~\ref{dolbzp} can be
described explicitly in terms of the map $\varPsi$ defining the
complex structure on $\zk$. We recall the map $A_\C\colon\C^m\to
N_\C$, $\mb e_i\mapsto\mb a_i$ and the Gale dual $(m-n)\times m$
matrix $\varGamma=(\gamma_{jk})$ whose rows form a basis of
relations between $\mb a_1,\ldots,\mb a_m$. By
Construction~\ref{psi}, $\Im\varPsi\subset\Ker A_\C$. Denote by
$\Ann U$ the annihilator of a linear subspace $U\subset\C^m$, i.e.
the subspace of linear functions on $\C^m$ vanishing on~$U$.

\begin{lemma}\label{mumatrix}
The first Chern class map
\[
  c\colon H^{1,0}_{\bar\partial}(T_\C^{\ell})\to H^2(V_\Sigma, \C)=
  H_{\bar\partial}^{1,1}(V_\Sigma)
\]
of the principal $T^\ell_\C$-bundle $\zk\to V_\Sigma$ is given by
the composition
\[
\begin{CD}
  \Ann\Im\varPsi/\Ann\Ker A_\C @>i>> \C^m/\Ann\Ker A_\C
  @>p>> \C^{m-k}/\Ann\Ker A_\C
\end{CD}
\]
where $i$ is the inclusion, $k$ is the number of zero vectors
among $\mb a_1,\ldots,\mb a_m$, and $p$ is the projection
forgetting the coordinates in $\C^m$ corresponding to zero
vectors. Explicitly, the value of $c$ on the generators of
$H^{1,0}_{\bar\partial}(T_\C^{\ell})$ is given by
\[
  c(\xi_j)=\mu_{j1}v_1+\cdots+\mu_{jm}v_m,\quad 1\le j\le \ell,
\]
where $M=(\mu_{jk})$ is an $\ell\times m$-matrix satisfying the
two conditions:
\begin{itemize}
\item[(a)] $\varGamma M^t\colon\C^\ell\to\C^{2\ell}$ is a
monomorphism;

\item[(b)] $M\varPsi=0$.
\end{itemize}
\end{lemma}
\begin{proof}
Let $A^t_\C\colon N_\C^*\to\C^m$, $\mb u\mapsto(\langle\mb a_1,\mb
u\rangle,\ldots,\langle\mb a_m,\mb u\rangle)$, be the dual map. We
have $H^1(T_\C^{\ell};\C)=\C^m/\Im A^t_\C=(\Ker A_\C)^*$ and
$H^2(V_\Sigma;\C)=\C^{m-k}/\Im A^t_\C$. The first Chern class map
$c\colon H^1(T_\C^{\ell};\C)\to H^2(V_\Sigma;\C)$ (the
transgression) is then given by $p\colon \C^m/\Im
A^t_\C\to\C^{m-k}/\Im A^t_\C$. In order to separate the
holomorphic part of $c$ we need to identify the subspace of
holomorphic differentials
$H^{1,0}_{\bar\partial}(T_\C^{\ell})\cong\C^\ell$ inside the space
of all 1-forms $H^1(T_\C^{\ell};\C)\cong\C^{2\ell}$. Since
\[
  T_\C^{\ell}=G/C_{\varPsi}=(\Ker\exp A_\C)/(\exp\Im\varPsi),
\]
holomorphic differentials on $T_\C^{\ell}$ correspond to
$\C$-linear functions on $\Ker A_\C$ which vanish on $\Im\varPsi$.
The space of functions on $\Ker A_\C$ is $\C^m/\Im
A^t_\C=\C^m/\Ann\Ker A_\C$, and the functions vanishing on
$\Im\varPsi$ form the subspace $\Ann\Im\varPsi/\Ann\Ker A_\C$.
Condition~(b) says exactly that the linear functions on $\C^m$
corresponding to the rows of $M$ vanish on $\Im\varPsi$.
Condition~(a) says that the rows of $M$ constitute a basis in the
complement of $\Ann\Ker A_\C$ in $\Ann\Im\varPsi$.
\end{proof}

It is interesting to compare Theorem~\ref{dolbzp} with the
following description of the de Rham cohomology of~$\zk$.

\begin{theorem}[{\cite[Theorem~7.36]{bu-pa02}}]\label{cohomzpred}
Let $\zk$ and $V_\Sigma$ be as in Theorem~\ref{dolbzp}. The de
Rham cohomology $H^*(\zk)$ is isomorphic to the cohomology of the
differential graded algebra
\[
  \bigl[\Lambda[u_1,\ldots,u_{m-n}]\otimes
  H^*(V_\Sigma),d\bigr],
\]
with $\deg u_j=1$, $\deg v_i=2$, and differential $d$ defined on
the generators as
\[
  dv_i=0,\quad du_j=\gamma_{j1}v_1+\cdots+\gamma_{jm}v_m,\quad
  1\le j\le m-n.
\]
\end{theorem}

This follows from the more general
result~\cite[Theorem~7.7]{bu-pa02} describing the cohomology
of~$\zk$. For more information about $H^*(\zk)$ see~\cite{bu-pa02}
and~\cite[\S4]{pano08}.

There are two classical spectral sequences for the Dolbeault
cohomology. First, the \emph{Borel spectral
sequence}~\cite{bore66} of a holomorphic bundle $E\to B$ with a
compact K\"ahler fibre~$F$, which has
$E_2=H_{\bar\partial}(B)\otimes H_{\bar\partial}(F)$ and converges
to $H_{\bar\partial}(E)$. Second, the \emph{Fr\"olicher spectral
sequence}~\cite[\S3.5]{gr-ha78}, whose $E_1$-term is the Dolbeault
cohomology of a complex manifold $M$ and which converges to the de
Rham cohomology of~$M$. Theorem~\ref{dolbzp} implies a collapse
result for these spectral sequences:

\begin{corollary}\
\begin{itemize}
\item[(a)]
The Borel spectral sequence of the holomorphic principal bundle
$\zk\to V_\Sigma$ collapses at the $E_3$-term, i.e.
$E_3=E_\infty$;

\item[(b)]
the Fr\"olicher spectral sequence of $\zk$ collapses at the
$E_2$-term.
\end{itemize}
\end{corollary}
\begin{proof}
To prove~(a) we just observe that the differential
algebra~\eqref{zkmult} is the $E_2$-term of the Borel spectral
sequence, and its cohomology is the $E_3$-term.

By comparing the Dolbeault and de Rham cohomology algebras of
$\zk$ given by Theorems~\ref{dolbzp} and~\ref{cohomzpred} we
observe that the elements $\eta_1,\ldots,\eta_\ell\in E_1^{0,1}$
cannot survive in the~$E_\infty$-term of the Fr\"olicher spectral
sequence. The only possible nontrivial differential on these
elements is $d_1\colon E_1^{0,1}\to E_1^{1,1}$. By
Theorem~\ref{cohomzpred}, the cohomology algebra of $[E_1,d_1]$ is
exactly the de Rham cohomology of~$\zk$, proving~(b).
\end{proof}

Theorem~\ref{cohomzpred} can also be interpreted as a collapse
result for the Leray--Serre spectral sequence of the principal
$T^{m-n}$-bundle $\zk\to V_\Sigma$.

In order to proceed with calculation of Hodge numbers, we need the
following bounds for the dimension of $\Ker c$ in
Lemma~\ref{mumatrix}:

\begin{lemma}\label{cbounds}
Let $k$ be the number of zero vectors among $\mb a_1,\ldots,\mb
a_m$. Then
\[
  k-\ell\le\dim_\C\Ker\bigl(c\colon
  H^{1,0}_{\bar\partial}(T_\C^{\ell})\to
  H_{\bar\partial}^{1,1}(V_\Sigma)\bigr)\le{\textstyle\frac k2}.
\]
In particular, if $k\le1$ then $c$ is monomorphism.
\end{lemma}
\begin{proof}
Consider the diagram
\[
\begin{CD}
  \Ann\Im\varPsi/\Ann\Ker A_\C @> i>> \C^m/\Ann\Ker A_\C
  @>p>> \C^{m-k}/\Ann\Ker A_\C\\
  @VV{\cong}V @VV\mathrm{Re}V @VV\mathrm{Re}V\\
  \R^{m-n}@=\R^{m-n} @>p'>> \R^{m-n-k}.
\end{CD}
\]
The left vertical arrow is an ($\R$-linear) isomorphism, as it has
the form $H^{1,0}_{\bar\partial}(T_\C^{\ell})\to
H^1(T_\C^{\ell},\C)\to H^1(T_\C^{\ell},\R)$, and any real-valued
function on the lattice $\Gamma$ defining the torus
$T_\C^{\ell}=\C^\ell/\Gamma$ is the real part of the restriction
to $\Gamma$ of a $\C$-linear function on~$\C^\ell$.

Since the diagram above is commutative, the kernel of $c=p\circ i$
has real dimension at most~$k$, which implies the upper bound on
its complex dimension. For the lower bound, $\dim_\C\Ker c\ge\dim
H^{1,0}_{\bar\partial}(T_\C^{\ell})-\dim
H_{\bar\partial}^{1,1}(V_\Sigma) =\ell-(2\ell-k)=k-\ell$.
\end{proof}

\begin{theorem}\label{hodge}
Let $\zk$ be as in Theorem~\ref{dolbzp}, and let $k$ be the number
of zero vectors among $\mb a_1,\ldots,\mb a_m$. Then the Hodge
numbers $h^{p,q}=h^{p,q}(\zk)$ satisfy
\begin{itemize}
\item[(a)]
$\binom{k-\ell}p \le h^{p,0}\le\binom {[k/2]}p$ for $p\ge0$; in
particular, $h^{p,0}=0$ for $p>0$ if $k\le1$;
\item[(b)] $h^{0,q}=\binom\ell q$ for $q\ge0$;
\item[(c)] $h^{1,q}
  =(\ell-k)\binom\ell{q-1}+h^{1,0}\binom{\ell+1}q$ for $q\ge1$;
\item[(d)] $\frac{\ell(3\ell+1)}2-h_2(\sK)-\ell k+(\ell+1)h^{2,0}
 \le h^{2,1}\le\frac{\ell(3\ell+1)}2-\ell k+(\ell+1)h^{2,0}$.
\end{itemize}
\end{theorem}
\begin{proof}
Let $A^{p,q}$ denote the bidegree $(p,q)$ component of the
differential algebra from Theorem~\ref{dolbzp}, and let
$Z^{p,q}\subset A^{p,q}$ denote the subspace of $d$-cocycles. Then
$d^{1,0}\colon A^{1,0}\to Z^{1,1}$ coincides with the map~$c$, and
the required bounds for $h^{1,0}=\Ker d^{1,0}$ are already
established in Lemma~\ref{cbounds}. Since $h^{p,0}=\dim\Ker
d^{p,0}$, and $\Ker d^{p,0}$ is the $p$th exterior power of the
space $\Ker d^{1,0}$, statement~(a) follows.

The differential is trivial on $A^{0,q}$, hence $h^{0,q}=\dim
A^{0,q}$, proving~(b).

The space $Z^{1,1}$ is spanned by the cocycles $v_i$ and
$\xi_i\eta_j$ with $\xi_i\in\Ker d^{1,0}$. Hence $\dim
Z^{1,1}=2\ell-k+h^{1,0}\ell$. Also, $\dim
d(A^{1,0})=\ell-h^{1,0}$, hence $h^{1,1}=\ell-k+h^{1,0}(\ell+1)$.
Similarly, $\dim
Z^{1,q}=(2\ell-k)\binom\ell{q-1}+h^{1,0}\binom\ell q$ (with basis
of $v_i\eta_{j_1}\cdots\eta_{j_{q-1}}$ and
$\xi_i\eta_{j_1}\cdots\eta_{j_q}$ where $\xi_i\in\Ker d^{1,0}$,
$j_1<\cdots<j_q$), and $d\colon A^{1,q-1}\to Z^{1,q}$ hits a
subspace of dimension $(\ell-h^{1,0})\binom\ell{q-1}$. This
proves~(c).

We have $A^{2,1}=V\oplus W$, where $U$ has basis of monomials
$\xi_iv_j$ and $W$ has basis of monomials $\xi_i\xi_j\eta_k$.
Therefore,
\begin{equation}\label{h21}
  h^{2,1}=\dim U-\dim dU+\dim W-\dim dW-\dim dA^{2,0}.
\end{equation}
Now $\dim U=\ell(2\ell-k)$, $0\le\dim dU\le h_2(\sK)$ (since
$dU\subset H_{\bar\partial}^{2,2}(V_\Sigma))$, $\dim W-\dim
dW=\dim\Ker d|_W=\ell h^{2,0}$, and $\dim
dA^{2,0}=\binom\ell2-h^{2,0}$. By substituting all this
into~\eqref{h21} we obtain the inequalities of~(d).
\end{proof}

\begin{remark}
At most one ghost vertex needs to be added to $\sK$ to make
$\dim\zk=m+n$ even. Since $h^{p,0}(\zk)=0$ when $k\le1$, the
manifold $\zk$ does not have holomorphic forms of any degree in
this case.

If $\zk$ is a torus (so that $\sK$ is empty), then $m=k=2\ell$,
and $h^{1,0}(\zk)=h^{0,1}(\zk)=\ell$. Otherwise
Theorem~\ref{hodge} implies that $h^{1,0}(\zk)<h^{0,1}(\zk)$, and
therefore the moment-angle manifold $\zk$ is not K\"ahler (in the
polytopal case this was observed in~\cite[Theorem~3]{meer00}).
\end{remark}

\begin{example}Let $\zk\cong S^1\times S^{2n+1}$ be a Hopf manifold of
Example~\ref{hopf}. Our rationality assumption is that $\mb
a_1\ldots,\mb a_{n+2}$ span an $n$-dimensional lattice $N$
in~$N_\R\cong\R^n$; in particular, the fan $\Sigma$ defined by the
proper subsets of $\mb a_1,\ldots,\mb a_{n+1}$ is rational. We
assume further that $\Sigma$ is regular (this is equivalent to the
condition $\mb a_1+\cdots+\mb a_{n+1}=\bf0$), so that $\Sigma$ is
a the normal fan of a Delzant $n$-dimensional simplex~$\Delta^n$.
We have $V_\Sigma=\C P^n$, and~\eqref{dolbtoric} describes its
cohomology as the quotient of $\C[v_1,\ldots,v_{n+2}]$ by the two
ideals: $\mathcal I$ generated by $v_1\cdots v_{n+1}$ and
$v_{n+2}$, and $\mathcal J$ generated by
$v_1-v_{n+1},\ldots,v_{n}-v_{n+1}$. The differential algebra of
Theorem~\ref{dolbzp} is therefore given by
$\bigl[\Lambda[\xi,\eta]\otimes\C[t]/t^{n+1},d\bigr]$, with
$dt=d\eta=0$ and $d\xi=t$ for a proper choice of~$t$. The
nontrivial cohomology classes are represented by the cocycles $1$,
$\eta$, $\xi t^n$ and $\xi\eta t^n$, which gives the following
nonzero Hodge numbers of~$\zk$:
$h^{0,0}=h^{0,1}=h^{n+1,n}=h^{n+1,n+1}=1$. Observe that the
Dolbeault cohomology and Hodge numbers do not depend on a choice
of complex structure (the map~$\varPsi$).
\end{example}

\begin{example}[Calabi--Eckmann manifold]
Let $\{\sK;\mb a_1,\ldots,\mb a_{n+2}\}$ be the data defining the
normal fan of the product $P=\Delta^p\times\Delta^q$ of two
Delzant simplices with $p+q=n$, $1\le p\le q\le n-1$. That is,
$\mb a_1,\ldots,\mb a_p,\mb a_{p+2},\ldots,\mb a_{n+1}$ is a basis
of lattice~$N$ and there are two relations $\mb a_1+\cdots+\mb
a_{p+1}=\bf0$ and $\mb a_{p+2}+\cdots+\mb a_{n+2}=\bf0$. The
corresponding toric variety $V_\Sigma$ is $\C P^p\times \C P^q$
and its cohomology ring is isomorphic to $\C[x,y]/(x^{p+1},
y^{q+1})$. The map
\[
  \varPsi\colon\C\to\C^{n+2},\quad w\mapsto
  (1,\ldots,1,\alpha w,\ldots,\alpha w),
\]
where the number of units is $p+1$ and $\alpha\in\C\setminus\R$,
satisfies the conditions of Construction~\ref{psi}. The resulting
complex structure on $\zp\cong S^{2p+1}\times S^{2q+1}$ is that of
a \emph{Calabi--Eckmann manifold}. We denote complex manifolds
obtained in this way by $\mbox{\textit{C\!E}}(p,q)$ (the complex
structure depends on the choice of $\varPsi$, but we do not
reflect this in the notation). Each manifold
$\mbox{\textit{C\!E}}(p,q)$ is the total space of a holomorphic
principal bundle over $\C P^p\times \C P^q$ with fibre the complex
1-torus~$\C/(\Z\oplus\alpha\Z)$.

Theorem~\ref{dolbzp} and Lemma~\ref{mumatrix} provide the
following description of the Dolbeault cohomology of
$\mbox{\textit{C\!E}}(p,q)$:
\[
  H^{*,*}_{\bar\partial}\bigl(\mbox{\textit{C\!E}}(p,q)\bigr)\cong
  H\bigl[\Lambda[\xi,\eta]\otimes\C[x,y]/(x^{p+1},y^{q+1}),d\bigr],
\]
where $dx=dy=d\eta=0$ and $d\xi=x-y$ for an appropriate choice of
$x,y$. We therefore obtain
\begin{equation}\label{dolbce}
  H^{*,*}_{\bar\partial}\bigl(\mbox{\textit{C\!E}}(p,q)\bigr)\cong
  \Lambda[\omega,\eta]\otimes\C[x]/(x^{p+1}),
\end{equation}
where $\omega\in
H^{q+1,q}_{\bar\partial}\bigl(\mbox{\textit{C\!E}}(p,q)\bigr)$ is
the cohomology class of the cocycle
$\xi\frac{x^{q+1}-y^{q+1}}{x-y}$. This calculation is originally
due to~\cite[\S9]{bore66}. We note that Dolbeault cohomology of a
Calabi--Eckmann manifold depends only on $p,q$ and does not depend
on the complex parameter~$\alpha$ (or the map~$\varPsi$).
\end{example}

\begin{example}
Now let $P=\Delta^1\times\Delta^1\times\Delta^2\times\Delta^2$.
Then the moment-angle manifold $\zp$ has two structures of a
product of Calabi--Eckmann manifolds, namely,
$\mbox{\textit{C\!E}}(1,1)\times\mbox{\textit{C\!E}}(2,2)$ and
$\mbox{\textit{C\!E}}(1,2)\times\mbox{\textit{C\!E}}(1,2)$. Using
isomorphism~\eqref{dolbce} we observe that these two complex
manifolds have different Hodge numbers: $h^{2,1}=1$ in the first
case, and $h^{2,1}=0$ in the second. This shows that the choice of
the map $\varPsi$ affects not only the complex structure of~$\zk$,
but also its Hodge numbers, unlike the previous examples of
complex tori, Hopf and Calabi--Eckmann manifolds. Certainly it is
not highly surprising from the complex-analytic point of view.
\end{example}

\section{Hamiltonian-minimal Lagrangian submanifolds}
In this last section we apply the accumulated knowledge on
topology of moment-angle manifolds in a somewhat different area,
Lagrangian geometry. Systems of real quadrics, which we used in
Sections~\ref{intquad} and~\ref{mampol} to define moment-angle
manifolds, also give rise to a new large family of
Hamiltonian-minimal Lagrangian submanifolds in a complex space or
more general toric varieties.

Hamiltonian minimality ($H$-minimality for short) for Lagrangian
submanifolds is a symplectic analogue of minimality in Riemannian
geometry. A Lagrangian immersion is called $H$-minimal if the
variations of its volume along all Hamiltonian vector fields are
zero. This notion was introduced in the work of
Y.-G.~Oh~\cite{oh93} in connection with the celebrated
\emph{Arnold conjecture} on the number of fixed points of a
Hamiltonian symplectomorphism. The simplest example of an
$H$-minimal Lagrangian submanifold is the coordinate
torus~\cite{oh93} $S^1_{r_1}\times\dots \times S^1_{r_m}\subset
{\mathbb C}^m$, where $S^1_{r_k}$ denotes the circle of radius
$r_k>0$ in the $k$th coordinate subspace of~$\C^m$. More examples
of $H$-minimal Lagrangian submanifolds in a complex space were
constructed in the
works~\cite{ca-ur98},~\cite{he-ro02},~\cite{an-ca11}, among
others.

In~\cite{miro04} Mironov suggested a general construction of
$H$-minimal Lagrangian immersions $N\looparrowright\C^m$ from
intersections of real quadrics. These systems of quadrics are the
same as those we used to define moment-angle manifolds, and
therefore one can apply toric methods for analysing the
topological structure of~$N$. In~\cite{mi-pa13f} an effective
criterion was obtained for $N\looparrowright\C^m$ to be an
embedding: the polytope corresponding to the intersection of
quadrics must be Delzant. As a consequence, any Delzant polytope
gives rise to an $H$-minimal Lagrangian submanifold
$N\subset\C^m$. Like in the case of moment-angle manifolds, the
topology of $N$ is quite complicated even for low-dimensional
polytopes: for example, a Delzant 5-gon gives rise to $N$ which is
the total space of a bundle over a 3-torus with fibre a surface of
genus~5. Furthermore, by combining Mironov's construction with
symplectic reduction, a new family of $H$-minimal Lagrangian
submanifolds in of toric varieties was defined in~\cite{mi-pa13u}.
This family includes many previously constructed explicit examples
in~$\C^m$ and~$\C P^{m-1}$.

\subsection{Preliminaries}
Let $(M,\omega)$ be a symplectic manifold of dimension $2n$. An
immersion $i\colon N\looparrowright M$ of an $n$-dimensional
manifold $N$ is called \emph{Lagrangian} if $i^*(\omega)=0$. If
$i$ is an embedding, then $i(N)$ is a \emph{Lagrangian
submanifold} of~$M$. A vector field $X$ on $M$ is
\emph{Hamiltonian} if the 1-form $\omega(X,\,\cdot\,)$ is exact.

Now assume that $M$ is K\"ahler, so that it has compatible
Riemannian metric and symplectic structure. A Lagrangian immersion
$i\colon N\looparrowright M$ is called \emph{Hamiltonian minimal}
(\emph{$H$-minimal}) if the variations of the volume of $i(N)$
along all Hamiltonian vector fields with compact support are zero,
that is,
\[
  \frac d{dt}\mathop{\mathrm{vol}}\bigl(i_t(N)\bigr)\big|_{t=0}=0,
\]
where $i_t(N)$ is a deformation of $i(N)$ along a Hamiltonian
vector field, $i_0(N)=i(N)$, and $\mathop{\mathrm{vol}}(i_t(N))$
is the volume of the deformed part of $i_t(N)$. An immersion $i$
is \emph{minimal} if the variations of the volume of $i(N)$ along
\emph{all} vector fields are zero.

Our basic example is $M=\C^m$ with the Hermitian metric
$2\sum_{k=1}^m d\overline{z}_k\otimes dz_k$. Its imaginary part is
the symplectic form of Example~\ref{simcm}. At the end we consider
a more general case when $M$ is a toric manifold.

\subsection{The construction}\label{construction}
We consider an intersection of quadrics similar to~\eqref{zgamma},
but in the real space:
\begin{equation}\label{rgamma}
  \mathcal R=\Bigl\{\mb u=(u_1,\ldots,u_m)\in\R^m\colon
  \sum_{k=1}^m\gamma_{jk}u_k^2=\delta_j,\quad\text{for }
  1\le j\le m-n\Bigr\}.
\end{equation}

We assume the nondegeneracy and rationality conditions on the
coefficient vectors
$\gamma_i=(\gamma_{1i},\ldots,\gamma_{m-n,i})^t\in\R^{m-n}$,
$i=1,\ldots,m$:
\begin{itemize}
\item[(a)] $\delta\in
\R_\ge\langle\gamma_1,\ldots,\gamma_m\rangle$;

\item[(b)] if $\delta\in\R_\ge\langle
\gamma_{i_1},\ldots\gamma_{i_k}\rangle$, then $k\ge m-n$;

\item[(c)]
the vectors $\gamma_1,\ldots,\gamma_m$ generate a lattice
$L\cong\Z^{m-n}$ in~$\R^{m-n}$.
\end{itemize}

These conditions guarantee that $\mathcal R$ is a smooth
$n$-dimensional submanifold in $\R^m$ (by the argument of
Proposition~\ref{zgsmooth}) and that
\[
  T_\varGamma=
  \bigl\{\bigr(e^{2\pi i\langle\gamma_1,\varphi\rangle},
  \ldots,e^{2\pi i\langle\gamma_m,\varphi\rangle}\bigl)
  \in\T^m\bigr\}
\]
is an $(m-n)$-dimensional torus subgroup in~$\T^m$. We identify
the torus $T_\varGamma$ with $\R^{m-n}/\displaystyle L^*$ and
represent its elements by $\varphi\in\R^{m-n}$. We also define
\[
  D_\varGamma=\frac12 L^*/L^*\cong(\Z_2)^{m-n}.
\]
Note that $D_\varGamma$ embeds canonically as a subgroup in
$T_\varGamma$.

Now we view the intersection $\mathcal R$ as a subset in the
intersection $\mathcal Z$ or Hermitian quadrics (or as a subset in
the whole complex space $\C^m$), and `spread' it by the action of
$T_\varGamma$, that is, consider the set of $T_\varGamma$-orbits
through $\mathcal R$. More precisely, we consider the map
\begin{align*}
  j\colon\mathcal R\times T_\varGamma &\longrightarrow \C^m,\\
  (\mb u,\varphi) &\mapsto \mb u\cdot\varphi=\bigl(u_1e^{2\pi
i\langle\gamma_1,\varphi\rangle},\ldots,u_me^{2\pi
i\langle\gamma_m,\varphi\rangle}\bigr)
\end{align*}
and observe that $j(\mathcal R\times T_\varGamma)\subset\mathcal
Z$. We let $D_\varGamma$ act on $\mathcal R_\varGamma\times
T_\varGamma$ diagonally; this action is free, since it is free on
the second factor. The quotient
\[
  N=\mathcal R\times_{D_\varGamma} T_\varGamma
\]
is an $m$-dimensional manifold.

For any $\mb u=(u_1,\ldots,u_m)\in\mathcal R$, we have the
sublattice
\[
  L_{\mb u}=\Z\langle\gamma_k\colon u_k\ne0\rangle\subset
  L=\Z\langle\gamma_1,\ldots,\gamma_m\rangle.
\]
The set of $T_\varGamma$-orbits through $\mathcal R$ is an
immersion of~$N$:

\begin{lemma}\label{immer}\
\begin{itemize}
\item[(a)] The map $j\colon\mathcal R\times
T_\varGamma \to\C^m$ induces an immersion $i\colon
N\looparrowright\C^m$.

\item[(b)] The immersion $i$ is an embedding if and only
if $L_{\mb u}=L$ for any $\mb u\in\mathcal R$.
\end{itemize}
\end{lemma}
\begin{proof}
Take $\mb u\in\mathcal R$, $\varphi\in T_\varGamma$ and $g\in
D_\varGamma$. We have $\mb u\cdot g\in\mathcal R$, and $j(\mb
u\cdot g,g\varphi)=\mb u\cdot g^2\varphi=\mb u\cdot\varphi=j(\mb
u,\varphi)$. Hence the map $j$ is constant on
$D_\varGamma$-orbits, and therefore induces a map of the quotient
$N=(\mathcal R\times T_\varGamma)/D_\varGamma$, which we denote
by~$i$.

Assume that $j(\mb u,\varphi)=j(\mb u',\varphi')$. Then $L_{\mb
u}=L_{\mb u'}$ and
\begin{equation}\label{uu'}
  u_ke^{2\pi i\langle\gamma_k,\varphi\rangle}=u'_ke^{2\pi
  i\langle\gamma_k,\varphi'\rangle}\quad
  \text{for }k=1,\ldots,m.
\end{equation}
Since both $u_k$ and $u'_k$ are real, this implies that $e^{2\pi
i\langle\gamma_k,\varphi-\varphi'\rangle}=\pm1$ whenever
$u_k\ne0$, or, equivalently,
$\varphi-\varphi'\in\frac12\displaystyle L^*_{\mb u}/L^*$. In
other words,~\eqref{uu'} implies that $\mb u'=\mb u\cdot g$ and
$\varphi'=g\varphi$ for some $g\in\frac12{\displaystyle L^*_{\mb
u}/L^*}$. The latter is a finite group by Lemma~\ref{afree}; hence
the preimage of any point of $\C^m$ under $j$ consists of a finite
number of points. If $L_{\mb u}=L$, then $\frac12{\displaystyle
L^*_{\mb u}/L^*}=\frac12\displaystyle L^*/L^*=D_\varGamma$; hence
$(\mb u,\varphi)$ and $(\mb u',\varphi')$ represent the same point
in~$N$. Statement~(b) follows; to prove~(a), it remains to observe
that we have $L_{\mb u}=L$ for generic $\mb u$ (with all
coordinates nonzero).
\end{proof}

\begin{theorem}[{\cite[Th.~1]{miro04}}]\label{hmin}
The immersion $i\colon N\looparrowright\C^m$ is $H$-minimal
Lagrangian. Moreover, if $\sum_{k=1}^m\gamma_k=0$, then $i$ is a
minimal Lagrangian immersion.
\end{theorem}
\begin{proof}
We only prove that $i$ is a Lagrangian immersion here. Let
\[
  (\mb x,\varphi)\mapsto
  \mb z(\mb x,\varphi)=\Bigl(u_1(\mb x)e^{2\pi
  i\langle\gamma_1,\varphi\rangle},\ldots,u_m(\mb x)e^{2\pi
  i\langle\gamma_m,\varphi\rangle}\Bigr)
\]
be a local coordinate system on $N=\mathcal R\times_{D_\varGamma}
T_\varGamma$, where $\mb x=(x_1,\ldots,x_n)\in\R^n$ and
$\varphi=(\varphi_1,\ldots,\varphi_{m-n})\in\R^{m-n}$. Let
$\langle\xi,\eta\rangle_\C=\sum_{i=1}^m\overline\xi_i\eta_i
=\langle\xi,\eta\rangle+i\omega(\xi,\eta)$ be the Hermitian scalar
product of $\xi,\eta\in\C^m$. Then
\[
  \Bigl\langle\frac{\partial\mb z}{\partial x_k},
  \frac{\partial\mb z}{\partial\varphi_j}\Bigr\rangle_\C=
  2\pi i\Bigr(\gamma_{j1}u_1\frac{\partial u_1}{\partial x_k}+\cdots+
  \gamma_{jm}u_m\frac{\partial u_m}{\partial x_k}\Bigl)=0
\]
where the second identity follows by differentiating the quadrics
equations~\eqref{rgamma}. Also, $\bigl\langle\frac{\partial\mb
z}{\partial x_k},\frac{\partial\mb z}{\partial
x_j}\bigr\rangle_\C\in\R$ and $\bigl\langle\frac{\partial\mb
z}{\partial\varphi_k},\frac{\partial\mb
z}{\partial\varphi_j}\bigr\rangle_\C\in\R$. It follows that
\[
  \omega\Bigl(\frac{\partial\mb z}{\partial x_k},
  \frac{\partial\mb z}{\partial\varphi_j}\Bigr)=
  \omega\Bigl(\frac{\partial\mb z}{\partial x_k},
  \frac{\partial\mb z}{\partial x_j}\Bigr)=
  \omega\Bigl(\frac{\partial\mb z}{\partial \varphi_k},
  \frac{\partial\mb z}{\partial\varphi_j}\Bigr)=0,
\]
i.e. the restriction of the symplectic form to the tangent space
of~$N$ is zero.
\end{proof}

\begin{remark}
The identity $\sum_{k=1}^m\gamma_k=0$ can not hold for a compact
$\mathcal R$ (or~$N$).
\end{remark}

We recall from Theorem~\ref{polquad} that a nondegenerate
intersection of quadrics~\eqref{zgamma} or~\eqref{rgamma} defines
a simple polyhedron~\eqref{ptope}, and $\mathcal Z$ is identified
with the moment-angle manifold~$\zp$. Now we can summarise the
results of the previous sections in the following criterion for
$i\colon N\to\C^m$ to be an embedding:

\begin{theorem}\label{nembed}
Let $\mathcal Z$ and $\mathcal R$ be the intersections of
Hermitian and real quadrics defined by~\eqref{zgamma}
and~\eqref{rgamma} respectively, and satisfying
conditions~{\rm(a)--\,(c)} above. Let $P$ be the corresponding
simple polyhedron, and $N=\mathcal R\times_{D_\varGamma}
T_\varGamma$. The following conditions are equivalent:
\begin{itemize}
\item[(a)] $i\colon N\to\C^m$ is an embedding of an
$H$-minimal Lagrangian submanifold;

\item[(b)] $L_{\mb u}=L$ for every $\mb u\in\mathcal R$;

\item[(c)] $T_\varGamma$ acts freely on the moment-angle manifold $\mathcal Z=\zp$.

\item[(d)] $P$ is a Delzant polyhedron.
\end{itemize}
\end{theorem}
\begin{proof}
Equivalence (a)$\,\Leftrightarrow\,$(b) follows from
Lemma~\ref{immer} and Theorem~\ref{hmin}. Equivalence
(b)$\,\Leftrightarrow\,$(c) is Lemma~\ref{afree}. Equivalence
(c)$\,\Leftrightarrow\,$(d) is Theorem~\ref{propmmap}~(c).
\end{proof}

Toric topology provides large families of explicitly constructed
Delzant polytopes. Basic examples include simplices and cubes in
all dimensions. It is easy to see that the Delzant condition is
preserved under several operations on polytopes, such as taking
products or cutting vertices or faces by well-chosen hyperplanes.
This is sufficient to show that many important families of
polytopes, such as \emph{associahedra} (Stasheff polytopes),
\emph{permutahedra}, and general \emph{nestohedra}, admit Delzant
realisations (see, for example,~\cite{post09} and~\cite{buch08}).

\subsection{Topology of Lagrangian submanifolds~$N$}
We start by reviewing three simple properties linking the
topological structure of $N$ to that of the intersections of
quadrics $\mathcal Z$ and~$\mathcal R$.

\begin{proposition}\label{nprop}\
\begin{itemize}
\item[(a)] The immersion of $N$ in $\C^m$ factors as
$N\looparrowright \mathcal Z\hookrightarrow\C^m$;
\item[(b)] $N$ is the total space of a bundle over the torus
$T^{m-n}$ with fibre $\mathcal R$;
\item[(c)] if $N\to\C^m$ is an embedding, then $N$ is the total space of a principal
$T^{m-n}$-bundle over the $n$-dimensional manifold $\mathcal
R/D_\varGamma$.
\end{itemize}
\end{proposition}
\begin{proof}
Statement (a) is clear. Since $D_\varGamma$ acts freely on
$T_\varGamma$, the projection $N=\mathcal
R\times_{D_\varGamma}T_\varGamma\to T_\varGamma/D_\varGamma$ onto
the second factor is a fibre bundle with fibre~$\mathcal R$.
Then~(b) follows from the fact that $T_\varGamma/D_\varGamma\cong
T^{m-n}$.

If $N\to\C^m$ is an embedding, then $T_\varGamma$ acts freely
on~$\mathcal Z$ by Theorem~\ref{nembed} and the action of
$D_\varGamma$ on $\mathcal R$ is also free. Therefore, the
projection $N=\mathcal R\times_{D_\varGamma}T_\varGamma\to
\mathcal R/D_\varGamma$ onto the first factor is a principal
$T_\varGamma$-bundle, which proves~(c).
\end{proof}

\begin{remark}
The quotient $\mathcal R/D_\varGamma$ is a \emph{real toric
variety}, or a \emph{small cover}, over the corresponding
polytope~$P$, see~\cite{da-ja91} and~\cite{bu-pa02}.
\end{remark}

\begin{example}[one quadric]\label{1quad}
Suppose that $\mathcal R$ is given by a single equation
\begin{equation}\label{1q}
  \gamma_1u_1^2+\cdots+\gamma_mu_m^2=\delta
\end{equation}
in $\R^m$, where $\gamma_k\in\R$. If $\mathcal R$ is compact, then
$\mathcal R\cong S^{m-1}$, and the associated polytope $P$ is an
$n$-simplex~$\varDelta^n$. In this case, $N\cong
S^{m-1}\times_{\Z_2}S^1$, where the generator of $\Z_2$ acts by
the standard free involution on $S^1$ and by a certain involution
$\tau$ on~$S^{m-1}$. The topological type of $N$ depends
on~$\tau$. Namely,
\[
  N\cong\begin{cases}S^{m-1}\times S^1&\text{if $\tau$ preserves the orientation of }S^{m-1},\\
  \mathcal K^{m}&\text{if $\tau$ reverses the orientation of }S^{m-1},\end{cases}
\]
where $\mathcal K^m$ is the \emph{$m$-dimensional Klein bottle}.

\begin{proposition}\label{1qemb}
In the case $m-n=1$ (one quadric) we obtain an $H$-minimal
Lagrangian embedding of $N\cong S^{m-1}\times_{\Z_2}S^1$ in $\C^m$
if and only if $\gamma_1=\cdots=\gamma_m$ in~\eqref{1q}. In this
case, the topological type of $N=N(m)$ depends only on the parity
of~$m$ and is given by
\begin{align*}\label{m=1}
  N(m)&\cong S^{m-1}\times S^1&&\text{if $m$ is even},\\
  N(m)&\cong\mathcal K^{m}&&\text{if $m$ is odd}.
\end{align*}
\end{proposition}
\begin{proof}
Since there is $\mb u\in\mathcal R$ with only one nonzero
coordinate, Theorem~\ref{nembed} implies that $N$ embeds in $\C^m$
if only if $\gamma_i$ generates the same lattice as the whole set
$\gamma_1,\ldots,\gamma_m$ for each~$i$. Therefore,
$\gamma_1=\cdots=\gamma_m$. In this case $D_\varGamma\cong\Z_2$
acts by the standard antipodal involution on $S^{m-1}$, which
preserves orientation if $m$ is even and reverses orientation
otherwise.
\end{proof}

Both examples of $H$-minimal Lagrangian embeddings given by
Proposition~\ref{1qemb} are well known. The Klein bottle $\mathcal
K^m$ with even $m$ does not admit Lagrangian embeddings in~$\C^m$
(see~\cite{nemi09} and~\cite{shev09}).
\end{example}

\begin{example}[two quadrics]
In the case $m-n=2$, the topology of $\mathcal R$ and $N$ can be
described completely by analysing the action of the two commuting
involutions on the intersection of quadrics. We consider the
compact case here.

Using Proposition~\ref{propcf}, we write $\mathcal R$ in the form
\begin{equation}\label{2q}
\begin{aligned}
  \gamma_{11}u_1^2+\cdots+\gamma_{1m}u_m^2&=c,\\
  \gamma_{21}u_1^2+\cdots+\gamma_{2m}u_m^2&=0,
\end{aligned}
\end{equation}
where $c>0$ and $\gamma_{1i}>0$ for all $i$.

\begin{proposition}
There is a number $p$, \ $0<p<m$, such that $\gamma_{2i}>0$ for
$i=1,\ldots,p$ and $\gamma_{2i}<0$ for $i=p+1,\ldots,m$
in~\eqref{2q}, possibly after a reordering of the coordinates
$u_1,\ldots,u_m$. The corresponding manifold $\mathcal R=\mathcal
R(p,q)$, where $q=m-p$, is diffeomorphic to $S^{p-1}\times
S^{q-1}$. Its associated polytope $P$ either coincides with
$\Delta^{m-2}$ (if one of the inequalities in~\eqref{ptope} is
redundant) or is combinatorially equivalent to the product
$\Delta^{p-1}\times\Delta^{q-1}$ (if there are no redundant
inequalities).
\end{proposition}
\begin{proof}
We observe that $\gamma_{2i}\ne0$ for all~$i$ in~\eqref{2q}, as
$\gamma_{2i}=0$ implies that $\delta=(c\; 0)^t$ is in the cone
generated by one vector $\gamma_i$, which contradicts
Proposition~\ref{zgsmooth}~(b). By reordering the coordinates, we
can achieve that the first $p$ of $\gamma_{2i}$ are positive and
the rest are negative. Then $1<p<m$, because otherwise~\eqref{2q}
is empty. Now,~\eqref{2q} is the intersection of the cone over the
product of two ellipsoids of dimensions $p-1$ and $q-1$ (given by
the second quadric) with an $(m-1)$-dimensional ellipsoid (given
by the first quadric). Therefore, $\mathcal R(p,q)\cong
S^{p-1}\times S^{p-1}$. The statement about the polytope follows
from the combinatorial fact that a simple $n$-polytope with up to
$n+2$ facets is combinatorially equivalent to a product of
simplices; the case of one redundant inequality corresponds to
$p=1$ or $q=1$.
\end{proof}

An element $\varphi\in
D_\varGamma=\frac12L^*/L^*\cong\Z_2\times\Z_2$ acts on $\mathcal
R(p,q)$ by
\[
  (u_1,\ldots,u_m)\mapsto
  (\varepsilon_1(\varphi)u_1,\ldots,\varepsilon_m(\varphi)u_m),
\]
where $\varepsilon_k(\varphi)=e^{2\pi
i\langle\gamma_k,\varphi\rangle}=\pm1$ for $1\le k\le m$.

\begin{lemma}\label{free1}
Suppose that $D_\varGamma$ acts on $\mathcal R(p,q)$ freely and
$\varepsilon_i(\varphi)=1$ for some~$i$, $1\le i\le p$, and
$\varphi\in D_\varGamma$. Then $\varepsilon_l(\varphi)=-1$ for all
$l$ with $p+1\le l\le m$.
\end{lemma}
\begin{proof}
Assume the opposite, that is, that $\varepsilon_i(\varphi)=1$ for
some $1\le i\le p$ and $\varepsilon_j(\varphi)=1$ for some $p+1\le
j\le m$. Then $\gamma_{2i}>0$ and $\gamma_{2j}<0$ in~\eqref{2q},
so we can choose $\mb u\in\mathcal R(p,q)$ whose only nonzero
coordinates are $u_i$ and~$u_j$. The element $\varphi\in
D_\varGamma$ fixes this $\mb u$, leading to a contradiction.
\end{proof}

\begin{lemma}\label{free2}
Suppose that $D_\varGamma$ acts on $\mathcal R(p,q)$ freely. Then
there exist two generating involutions $\varphi_1,\varphi_2\in
D_\varGamma\cong\Z_2\times\Z_2$ whose action on $\mathcal R(p,q)$
is described by either~{\rm(a)} or~{\rm(b)} below, possibly after
a reordering of coordinates:
\begin{itemize}
\item[(a)]
$\begin{aligned}
  \varphi_1\colon(u_1,\ldots,u_m)&\mapsto
  (u_1,\ldots,u_k,-u_{k+1},\ldots,-u_p,-u_{p+1},\ldots,-u_m),\\[-2pt]
  \varphi_2\colon(u_1,\ldots,u_m)&\mapsto
  (-u_1,\ldots,-u_k,u_{k+1},\ldots,u_p,-u_{p+1},\ldots,-u_m);
\end{aligned}$\\
\item[(b)]
$\begin{aligned}
  \varphi_1\colon(u_1,\ldots,u_m)&\mapsto
  (-u_1,\ldots,-u_p,u_{p+1},\ldots,u_{p+l},-u_{p+l+1},\ldots,-u_m),\\[-2pt]
  \varphi_2\colon(u_1,\ldots,u_m)&\mapsto
  (-u_1,\ldots,-u_p,-u_{p+1},\ldots,-u_{p+l},u_{p+l+1},\ldots,u_m);
\end{aligned}$
\end{itemize}
here $0\le k\le p$ and $0\le l\le q$.
\end{lemma}
\begin{proof}
By Lemma~\ref{free1}, for each of the three nonzero elements
$\varphi\in D_\varGamma$, we have either
$\varepsilon_i(\varphi)=-1$ for $1\le i\le p$ or
$\varepsilon_i(\varphi)=-1$ for $p+1\le i\le m$. Therefore, we can
choose two different nonzero elements $\varphi_1,\varphi_2\in
D_\varGamma$ such that either $\varepsilon_i(\varphi_j)=-1$ for
$j=1,2$ and $p+1\le i\le m$, or $\varepsilon_i(\varphi_j)=-1$ for
$j=1,2$ and $1\le i\le p$. This corresponds to the cases (a) and
(b) above, respectively. In the former case, after reordering the
coordinates, we may assume that $\varphi_1$ acts as in~(a). Then
$\varphi_2$ also acts as in~(a), since otherwise the sum
$\varphi_1\cdot \varphi_2$ cannot act freely by Lemma~\ref{free1}.
The second case is treated similarly.
\end{proof}

Each of the actions of $D_\varGamma$ described in
Lemma~\ref{free2} can be realised by a particular intersection of
quadrics~\eqref{2q}. For example,
\begin{equation}\label{2qex}
\begin{aligned}
  2u_1^2+\cdots+2u_k^2+u_{k+1}^2+\cdots+u_p^2+
  u_{p+1}^2+\cdots+u_m^2&=3,\\
  u_1^2+\cdots+u_k^2+2u_{k+1}^2+\cdots+2u_p^2-
  u_{p+1}^2-\cdots-u_m^2&=0
\end{aligned}
\end{equation}
gives the first action of Lemma~\ref{free2}; the second action is
realised similarly. Note that the lattice $L$ corresponding
to~\eqref{2qex} is a sublattice of index 3 in~$\Z^2$. We can
rewrite~\eqref{2qex} as
\begin{equation}\label{2qex1}
\begin{aligned}
  u_1^2+\cdots+u_k^2&+u_{k+1}^2+\cdots+u_p^2&&=1,\\
  u_1^2+\cdots+u_k^2&&+u_{p+1}^2+\cdots+u_m^2&=2,
\end{aligned}
\end{equation}
in which case $L=\Z^2$. The action of the two involutions
$\psi_1,\psi_2\in D_\varGamma=\frac12\Z^2/\Z^2$ corresponding to
the standard basis vectors of $\frac12\Z^2$ is given by
\begin{equation}\label{2inv}
\begin{aligned}
  \psi_1\colon(u_1,\ldots,u_m)&\mapsto
  (-u_1,\ldots,-u_k,-u_{k+1},\ldots,-u_p,u_{p+1},\ldots,u_m),\\
  \psi_2\colon(u_1,\ldots,u_m)&\mapsto
  (-u_1,\ldots,-u_k,u_{k+1},\ldots,u_p,-u_{p+1},\ldots,-u_m).
\end{aligned}
\end{equation}

We denote the manifold $N$ corresponding to~\eqref{2qex1} by
$N_k(p,q)$. We have
\begin{equation}\label{nkpq}
  N_k(p,q)\cong\mathcal (S^{p-1}\times
  S^{q-1})\times_{\Z_2\times\Z_2}(S^1\times S^1),
\end{equation}
and the action of the two involutions on $S^{p-1}\times S^{q-1}$
is given by~\eqref{2inv}. Note that $\psi_1$ acts trivially on
$S^{q-1}$ and acts antipodally on $S^{p-1}$. Therefore,
\[
  N_k(p,q)\cong N(p)\times_{\Z_2}(S^{q-1}\times S^1),
\]
where $N(p)$ is the manifold from Proposition~\ref{1qemb}. If
$k=0$ then the second involution $\psi_2$ acts trivially on
$N(p)$, and $N_0(p,q)$ coincides with the product $N(p)\times
N(q)$ of the two manifolds from Example~\ref{1quad}. In general,
the projection
\[
  N_k(p,q)\to S^{q-1}\times_{\Z_2}S^1=N(q)
\]
describes $N_k(p,q)$ as the total space of a fibration over $N(q)$
with fibre~$N(p)$.

We summarise the above facts and observations in the following
topological classification result for compact $H$-minimal
Lagrangian submanifolds $N\subset\C^m$ obtained from intersections
of two quadrics.

\begin{theorem}\label{2qemb}
Let $N\to\C^m$ is the embedding of the $H$-minimal Lagrangian
submanifold corresponding to a compact intersection of two
quadrics. Then $N$ is diffeomorphic to some $N_k(p,q)$ given
by~\eqref{nkpq}, where $p+q=m$, $0<p<m$ and $0\le k\le p$.
Moreover, any such triple $(k,p,q)$ can be realised by~$N$.
\end{theorem}
\end{example}

In the case of up to two quadrics considered above, the topology
of $\mathcal R$ is relatively simple, and in order to analyse the
topology of $N$, one only needs to describe the action of
involutions on~$\mathcal R$. When the number of quadrics is more
than two, the topology of $\mathcal R$ becomes an issue as well.

\begin{example}[three quadrics]\label{3quad}
In the case $m-n=3$, the topology of compact manifolds $\mathcal
R$ and $\mathcal Z$ was fully described
in~\cite[Theorem~2]{lope89}. Each of these manifolds is
diffeomorphic to a product of three spheres or to a connected sum
of products of spheres with two spheres in each product.

Note that, for $m-n=3$, the manifolds $\mathcal R$ (or~$\mathcal
Z$) can be distinguished topologically by looking at the planar
Gale diagrams of the associated simple polytopes~$P$ (see
Section~\ref{galediag}). This chimes with the classification of
$n$-dimensional simple polytopes with $n+3$ facets, well-known in
combinatorial geometry.

The smallest polytope with $m-n=3$ is a pentagon. It has many
Delzant realisations, for instance,
\[
  P=\bigl\{(x_1,x_2)\in\R^2\colon x_1\ge0,\;x_2\ge0,\;-x_1+2\ge0,\;-x_2+2\ge0,\;
  -x_1-x_2+3\ge0\bigr\}.
\]
In this case, $\mathcal R$ is an oriented surface of genus~5
(see~\cite[Example~6.40]{bu-pa02}), and the moment-angle manifold
$\mathcal Z$ is diffeomorphic to a connected sum of 5 copies of
$S^3\times S^4$.

We therefore obtain an $H$-minimal Lagrangian submanifold
$N\subset\C^5$, which is the total space of a bundle over $T^3$
with fibre a surface of genus~5.
\end{example}

Now assume that the polytope $P$ associated with intersection of
quadrics~\eqref{rgamma} is a polygon (i.e., $n=2$). If there are
no redundant inequalities then $P$ is an $m$-gon and $\mathcal R$
is an orientable surface $S_g$ of genus $g=1+2^{m-3}(m-4)$
by~\cite[Example~6.40]{bu-pa02}. If there are $k$ redundant
inequalities, then $P$ is an $(m-k)$-gon. In this case $\mathcal
R\cong\mathcal R'\times(S^0)^k$, where $\mathcal R'$ corresponds
to an $(m-k)$-gon without redundant inequalities. That is,
$\mathcal R$ is a disjoint union of $2^k$ surfaces of genus
$1+2^{m-k-3}(m-k-4)$.

The corresponding $H$-minimal Lagrangian submanifold
$N\subset\C^m$ is the total space of a bundle over $T^{m-2}$ with
fibre~$S_g$. This is an aspherical manifold for $m\ge4$.

\subsection{Generalisation to toric manifolds}
Consider two sets of quadrics:
\begin{align*}
  \mathcal Z_{\varGamma}&=\Bigl\{\mb z\in\C^m\colon
  \sum\nolimits_{k=1}^m\gamma_k|z_k|^2=\mb c\Bigr\},\quad \gamma_k,\mb c\in\R^{m-n};\\
  \mathcal Z_{\varDelta}&=\Bigl\{\mb z\in\C^m\colon
  \sum\nolimits_{k=1}^m\delta_k|z_k|^2=\mb d\Bigr\},\quad \delta_k,\mb
  d\in\R^{m-\ell};
\end{align*}
such that $\mathcal Z_{\varGamma}$, $\mathcal Z_{\varDelta}$ and
$\mathcal Z_{\varGamma}\cap\mathcal Z_{\varDelta}$ satisfy the
nondegeneracy and rationality conditions (a)--(c) from
Subsection~\ref{construction}. Assume also that the polyhedra
associated with $\mathcal Z_{\varGamma}$, $\mathcal Z_{\varDelta}$
and $\mathcal Z_{\varGamma}\cap\mathcal Z_{\varDelta}$ are
Delzant.

The idea is to use the first set of quadrics to produce a toric
manifold $V$ via symplectic reduction (as described in
Section~\ref{symred}), and then use the second set of quadrics to
define an $H$-minimal Lagrangian submanifold in~$V$.

\begin{construction}
Define the real intersections of quadrics $\mathcal
R_{\varGamma}$, $\mathcal R_{\varDelta}$, the tori
$T_{\varGamma}\cong\T^{m-n}$, $T_{\varDelta}\cong\T^{m-\ell}$, and
the groups $D_{\varGamma}\cong\Z_2^{m-n}$,
$D_{\varDelta}\cong\Z_2^{m-\ell}$ as before.

We consider the toric variety $V$ obtained as the symplectic
quotient of $\C^m$ by the torus corresponding to the first set of
quadrics: $V=\mathcal Z_{\varGamma}/T_{\varGamma}$. It is a
K\"ahler manifold of real dimension~$2n$. The quotient $\mathcal
R_{\varGamma}/D_{\varGamma}$ is the set of real points of $V$ (the
fixed point set of the complex conjugation, or the real toric
manifold); it has dimension~$n$. Consider the subset of $\mathcal
R_{\varGamma}/D_{\varGamma}$ defined by the second set of
quadrics:
\[
  \mathcal S=(\mathcal R_{\varGamma}\cap\mathcal
  R_{\varDelta})/D_{\varGamma},
\]
we have $\dim\mathcal S=n+\ell-m$. Finally define the
$n$-dimensional submanifold of $V$:
\[
  N=\mathcal S\times_{D_{\varDelta}}T_{\varDelta}.
\]
\end{construction}

\begin{theorem}
$N$ is an $H$-minimal Lagrangian submanifold in~$V$.
\end{theorem}
\begin{proof}
Let $\widehat V$ be the symplectic quotient of $V$ by the torus
corresponding to the second set of quadrics, that is, $\widehat
V=(V\cap\mathcal Z_\varDelta)/T_\Delta=(\mathcal
Z_{\varGamma}\cap\mathcal Z_{\varDelta})/(T_{\varGamma}\times
T_{\varDelta})$. It is a toric manifold of real dimension
$2(n+\ell-m)$. The submanifold of real points
\[
  \widehat N=N/T_{\varDelta}=(\mathcal R_{\varGamma}\cap\mathcal
  R_{\varDelta})/(D_{\varGamma}\times D_{\varDelta})\hookrightarrow(\mathcal
  Z_{\varGamma}\cap\mathcal Z_{\varDelta})/(T_{\varGamma}\times
  T_{\varDelta})=\widehat V
\]
is the fixed point set of the complex conjugation, hence it is a
totally geodesic submanifold. In particular, $\widehat N$ is a
minimal submanifold in~$\widehat V$. According
to~\cite[Corollary~2.7]{dong07}, $N$ is an $H$-minimal submanifold
in~$V$.
\end{proof}

\begin{example}\

1. If $m-\ell=0$, i.e. $\mathcal Z_{\Delta}=\varnothing$, then
$V=\C^m$ and we get the original construction of $H$-minimal
Lagrangian submanifolds $N$ in~$\C^m$.

2. If $m-n=0$, i.e. $\mathcal Z_{\varGamma}=\varnothing$, then $N$
is set of real points of~$V$. It is minimal (totally geodesic).

3. If $m-\ell=1$, i.e. $\mathcal Z_{\Delta}\cong S^{2m-1}$, then
we get $H$-minimal Lagrangian submanifolds in $V=\C P^{m-1}$. This
includes the families of projective examples
of~\cite{miro03},~\cite{ma05} and~\cite{mi-zu08}.
\end{example}

\end{document}